\documentclass[reqno,11pt]{amsart}
\usepackage{amsmath, amsfonts,amsthm,amssymb,amsbsy,upref,color,graphicx,amscd,hyperref,enumerate}
\usepackage[active]{srcltx}
\usepackage[latin1]{inputenc}
\usepackage{bbm}
\usepackage{algorithm}
\usepackage{algorithmicx}
\usepackage{algpseudocode}
\usepackage{graphicx}
\usepackage{subfig}

\usepackage[toc,title,page]{appendix}

\usepackage{pgf}
\usepackage{xfrac}

\def\ds{\displaystyle}
\def\eps{{\varepsilon}}
\def\N{\mathbb{N}}

\def\R{\mathbb{R}}

\def\HH{\mathcal{H}}

\newcommand{\be}{\begin{equation}}
\newcommand{\ee}{\end{equation}}

\newcommand{\de}{\partial}
\newcommand{\dist}{{\rm {dist}}}

\newcommand{\B}{\mathcal{B}}

\newcommand{\CC}{\mathcal{C}}
\newcommand{\graph}{{\rm graph}}

\newcommand\res{\mathop{\hbox{\vrule height 7pt width .3pt depth 0pt
\vrule height .3pt width 5pt depth 0pt}}\nolimits}

\theoremstyle{plain}
\newtheorem{theo}{Theorem}

\numberwithin{equation}{section}
\theoremstyle{plain}
\newtheorem{teo}{Theorem}[section]
\newtheorem{lemma}[teo]{Lemma}
\newtheorem{cor}[teo]{Corollary}
\newtheorem{prop}[teo]{Proposition}
\newtheorem{deff}[teo]{Definition}
\theoremstyle{remark}
\newtheorem{oss}[teo]{Remark}
\newcommand{\ind}{\mathbbm{1}}
\usepackage{xcolor}

\title[Uniqueness of blow-up at isolated singularity]{Uniqueness of the blow-up at isolated singularities for the Alt-Caffarelli functional}

\author{Max Engelstein, Luca Spolaor, Bozhidar Velichkov}

\thanks{The first author was supported by an NSF MSPRF DMS-1703306. Much of this work was done while the first and second author were visiting the Universit\'e Grenoble Alpes and while the first author was visiting Princeton University; they thank the universities for their hospitality. The third author has been partially supported by the LabEx PERSYVAL-Lab (ANR-11-LABX-0025-01) project GeoSpec and the project ANR CoMeDiC. The authors thank David Jerison for helpful comments in the preparation of this manuscript. The authors also thank Aapo Kauranen and Marti Prats for their many suggestions and especially for noticing a mistake in the original exposition of Section 4.2. Finally, the authors thank several anonymous referees for their close reading and helpful suggestions.}


\address {Max Engelstein: \newline \indent
	University of Minnesota 
	\newline \indent
	206 Church St. SE \newline \indent
Minneapolis, MN 55455, USA
	}
\email{maxe@mit.edu}

\address {Luca Spolaor: \newline \indent
	UC San Diego
	\newline \indent
9500 Gilman Drive\newline \indent La Jolla, CA 92093, USA}

\email{lspolaor@ucsd.edu}

\address {Bozhidar Velichkov: \newline \indent
Dipartimento di Matematica e Applicazioni ``Renato Caccioppoli" \newline \indent
Universit\`a degli Studi di Napoli Federico \newline \indent
Via Cintia, Monte S. Angelo I-80126 Napoli, Italy}
\email{bozhidar.velichkov@unina.it}

\makeindex

\begin{document}


\begin{abstract}
In this paper we prove uniqueness of blow-ups and $C^{1,\log}$-regularity for the free-boundary of minimizers of the Alt-Caffarelli functional at points where one blow-up has an isolated singularity. We do this by establishing a (log-)epiperimetric inequality for the Weiss energy for traces close to that of a cone with isolated singularity, whose free-boundary is graphical and smooth over that of the cone in  the sphere. With additional assumptions on the cone, we can prove a classical epiperimetric inequality which can be applied to deduce a $C^{1,\alpha}$ regularity result. We also show that these additional assumptions are satisfied by the De Silva-Jerison-type cones, which are the only known examples of minimizing cones with isolated singularity. Our approach draws a connection between epiperimetric inequalities and the \L ojasiewicz inequality, and, to our knowledge, provides the first regularity result at singular points in the one-phase Bernoulli problem. 
\end{abstract}

\maketitle



\section{Introduction}

In this paper we prove a \emph{uniqueness of blow-up (with logarithmic decay) and regularity result} for the free-boundary of minimizers of the Alt-Caffarelli functional at points where one blow-up has an isolated singularity. In the special case where one of the blow-ups is \emph{integrable through rotation} (which includes the only known cones with an isolated singularity), we improve the rate of convergence to a H\"older one. We do this by establishing a log-epiperimetric inequality for the Weiss' boundary adjusted energy \eqref{eq:weissmonotonicity} around isolated singularities. 
We remark that this is the  {\emph{first regularity result at singular points of minimizers to the (one-phase) Alt-Caffarelli functional}} that we are aware of. Before stating the theorem, we need to introduce some notation. We shall denote by $\mathcal{E}$ the Alt-Caffarelli functional:
\begin{equation}\label{eq:altcaffunctional}
\mathcal{E}(u):=\int_{D}|\nabla u|^2\ dx+\big|\{u>0\}\cap D\big|\,,\qquad u\in H^1(D; \R^+),
\end{equation} 
(here and throughout $D \subset \mathbb R^d$ is a connected open set with Lipschitz regular boundary).

The functional in \eqref{eq:altcaffunctional} was first studied systematically by Alt and Caffarelli \cite{AlCa}, who in particular showed that minimizers exist, and satisfy the following, overdetermined, boundary value problem in a weak sense: 
\begin{equation}\label{e:stationa}
u \geq 0, \qquad \Delta u = 0\ \text{ in }\ \{ u > 0 \} \cap D, \qquad |\nabla u| = 1\ \text{ on }\ \partial \{u > 0 \} \cap D . 
\end{equation}

In \cite{Weiss1}, Weiss introduced the quantity, $W(u,x_0, r)$, (or $W(u,r)$, when $x_0=0$, and $W(u)$, when $x_0=0$ and $r=1$) which monotonically increases with $r$ for every minimizer $u$,
\begin{equation}\label{eq:weissmonotonicity}
W(u,x_0,r):=\frac{1}{r^{d}}\int_{B_r(x_0)} |\nabla u|^2\ dx -\frac{1}{r^{d+1}}\int_{\de B_r(x_0)}u^2\ d\HH^{d-1}+\frac{1}{r^{d}}\big|\{u>0\}\cap B_{r}(x_0)\big|\,,
\end{equation}
where $x_0\in \de \{u>0\}\cap D$ and $0<r<\dist(x_0, \de D)$. A consequence of this monotonicity is that every blow-up (see Section \ref{ss:notation}) is a 1-homogenous globally defined minimizer to \eqref{eq:altcaffunctional}.

 Non-flat 1-homogenous minimizers (and thus singular points) do not exist when $d \leq 4$ (for $d= 2$ this is due to \cite{AlCa}, for $d=3$ to \cite{CaJeKe}, and for $d =4$ to \cite{JeSa}). In contrast, De Silva and Jerison \cite{DeJe} constructed a class of cones with isolated singularities at the origin which are minimizers in dimensions $d \geq 7$. We refer to this class of cones as $\B$; specifically, 
$$
\B:=\big\{b_{\nu,\theta_0}\colon \R^d \to \R^+  \,\mbox{ $1$-homogeneous minimizers of $\mathcal{E}$ with $\{b_{\nu,\theta_0}>0\}=\CC_{\nu,\theta_0}$}\big\},
$$
where for given $\theta_0\in (0,\sfrac\pi2)$ and $\nu\in \mathbb S^{d-1}$ we define the cone
$$
\CC_{\nu,\theta_0}:=\left\{x \in \R^d\setminus\{0\}\,:\, \left| \frac x{|x|}\cdot \nu \right|<\sin(\theta_0)  \right\}\,.
$$
Notice that $W$ restricted to the class $\B$ depends only on $\theta_0$, and we set $W(b_{\nu,\theta_0}):=\Theta_{\theta_0}$.


Left open in the works of \cite{Weiss1, CaJeKe, DeJe, JeSa} is the question of whether or not blowups at singular points are unique. {\it A priori}, it is possible that the free boundary around a singular point asymptotically approaches one singular cone at a certain set of small scales, but approaches a different singular cone at another set of scales. The problem of ``uniqueness for blowups" is a central one in geometric analysis and free boundary problems (see, e.g. \cite{Simon0,GaPe}). Our main theorem is that if one blowup at $x_0 \in \partial \{u > 0\}$ is a cone $b$ with isolated singularity, then every blowup at $x_0$ is equal to $b$.

\begin{theo}[Regularity for isolated singularities]\label{t:main}
	Let $u\in H^1(D)$ be a minimizer of the Alt-Caffarelli functional $\mathcal{E}$ on a domain $D\subset \R^d$ and let $x_0\in \de\{u>0\}\cap D$ be a singular point of the free-boundary such that a blow up for $u$ at $x_0$, call it $b$, has an \emph{isolated singularity}. Then $b$ is the unique blow up and, furthermore, there exists $r_0>0$ such that $\de\{u>0\}\cap B_{r_0}(x_0)$ is a $C^{1,\log}$ graph over $\de \{b>0\}\cap B_{r_0}(x_0)$.
\end{theo}

If we have additional information on the blowup, $b$ (namely that it is integrable through rotations, see Definition \ref{d:integrability}) we can improve the rate of convergence to the minimal cone.

\begin{theo}[Regularity for isolated and integrable singularities]\label{t:main2}
	Let $u\in H^1(D)$ be a minimizer of the Alt-Caffarelli functional $\mathcal{E}$ on a domain $D\subset \R^d$ and let $x_0\in \de\{u>0\}\cap  D$ be a singular point of the free-boundary such that a blow up for $u$ at $x_0$, call it $b$, has an \emph{isolated singularity} and is \emph{integrable through rotation} (see Definition \ref{d:integrability}).
	Then $b$ is the unique blow up and there exists $r_0>0$ such that $\de\{u>0\}\cap B_{r_0}(x_0)$ is a $C^{1,\alpha}$ graph over $\de \{b>0\}\cap B_{r_0}(x_0)$.
\end{theo}

We will prove that all the cones in the class $\mathcal B$ satisfy the above integrability condition, thus leading to the following corollary.

\begin{cor}[Regularity for isolated singularities of De Silva-Jerison type]\label{c:DeSiJe}
	Let $u\in H^1(D)$ be a minimizer of the Alt-Caffarelli functional $\mathcal{E}$ on a domain $D\subset \R^d$ and let $x_0\in \de\{u>0\}\cap D$ be a singular point of the free-boundary such that for some $\nu \in \mathbb S^{d-1}$ and $\theta_0\in (0,\sfrac\pi2)$ the function $b_{\nu,\theta_0}$ is a blow-up for $u$ at $x_0$.
	Then $b_{\nu,\theta_0}$ is the unique blow-up and furthermore there exists $r_0>0$ such that $\de\{u>0\}\cap B_{r_0}(x_0)$ is a $C^{1,\alpha}$ graph over $\de\CC_{\nu,\theta_0}\cap B_{r_0}(x_0)$. 
\end{cor}


The main ingredient in the proofs of Theorem \ref{t:main} and \ref{t:main2} is a (log-)epiperimetric inequality for minimizers whose trace on $\partial B_1$ has a free boundary which can be written as a smooth graph over $\de \Omega_b$, where $\Omega_b= \{b>0\}\cap \de B_1$. We will often consider graphs on the sphere of the following form: given a function $\zeta \colon \de\Omega_b\to \R$, we define the set on the sphere 
\begin{equation}\label{e:graph}
\graph_{\de \Omega_b}(\zeta):=\big\{\exp_x(\zeta(x)\,\nu(x)) \in \mathbb S^{d-1}\,:\,  x\in \de \Omega_b\big\},
\end{equation}
where $\nu(x)\in T_x(\de \{b>0\}\cap \mathbb S^{d-1})$ is the unit normal in the sphere to $\de \{b>0\}$ pointing outside $\{b>0\}$, and $\exp :T \mathbb S^{d-1}\to \mathbb S^{d-1}$ is the exponential map
$$
\exp_x(v):=\cos(\|v\|)\,x+\sin(\|v\|)\frac{v}{\|v\|}\ ,\quad\mbox{for every}\quad v\in T_x\mathbb S^{d-1}\,.
$$

\begin{theo}[(Log-)Epiperimetric inequality for traces with smooth free-boundary]\label{t:epi_smooth}
	Let $b\in H^1(B_1)$ be a one-homogeneous minimizer of the Alt-Caffarelli functional $\mathcal{E}$ with an isolated singularity at the origin.  There exist constants $\eps=\eps(d,b)>0, \gamma = \gamma(d,b) \in [0, 1)$ and $\delta_0=\delta_0(d,b)>0$, depending on $b$ and on the dimension $d$, such that the following holds. 
	
	If $c\in H^{1}(\de B_1,\R_+)$ is such that there exists $\zeta\in{C^{2,\alpha}(\de\Omega_b)}$ satisfying
	\begin{equation}\label{e:epi_ass}
	\de\{c>0\}=\graph_{\de \Omega_b}(\zeta)\,,
	\quad \mbox{with}\quad
	\|\zeta\|_{C^{2,\alpha}}\leq C_d \|\zeta\|_{L^2}<\delta\,,
	\quad\mbox{and}\quad \|c - b\|_{L^2(\partial B_1)} < \delta\,,
	\end{equation}
	then there exists a function $h\in H^1(B_1,\R_+)$ such that $h=c$ on $\de B_1$ and 
	\begin{equation}\label{e:epii}
	W(h)-W(b)\leq \Big(1-\eps \big|W(z)-W(b)\big|^\gamma \Big) \big(W(z)-W(b)\big),
	\end{equation}
	where $z$ is the $1$-homogeneous extension of $c$ to $B_1$.
	
	In the case where $b$ is integrable through rotations (see Definition \ref{d:integrability}), we can take $\gamma = 0$ in \eqref{e:epii} above. 
\end{theo}




Epiperimetric inequalities have been used in both minimal surfaces and obstacle-type free boundary problems; they can be divided into two categories: 

{\it The epiperimetric inequalities by contradiction}  (see for instance \cite{taylor1,taylor2,Weiss2,FoSp, GaPeVe}) are based on linearization techniques and are at the moment out of reach for the Bernoulli problem, due to a lack of compactness for the sequence of linearizations. This is essentially due to the nature of the Alt-Caffarelli functional which is fundamentally different from both minimal surfaces and obstacle-type problems in the fact that, while the energy and measure terms balance perfectly, small perturbations can drastically change one term while only slightly modifying the other. Additionally, the contradiction arguments present in the literature only apply to regular points or to singular points under some quite restrictive assumptions, which in our framework might cover only the integrable case. 

{\it The direct epiperimetric inequalities} were introduced in the context of minimal surfaces by Reifenberg \cite{Reif2} and White \cite{Wh}, and extended to the free-boundary setting by the second and third named authors in \cite{SpVe}; the direct method is based on an explicit construction of the competitor
and is more adapted to establish decay estimates around singular points. Indeed, in \cite{CoSpVe,CoSpVe2}, the second and third named author, together with Maria Colombo, proved for the first time a {\it logarithmic epiperimetric inequality} in the context of obstacle-type problems, that is a quantitative estimate on the optimality of the homogeneous extensions which gives a logarithmic decay to the blow-up. The additional term in this inequality, responsible for the logarithmic decay rate, is related to the possible presence of elements in the kernel of a suitable linearized operator which are not infinitesimal generators of isometries of the space. In particular, this type of inequality seems to be more suitable to deal with non-integrable singularities in general. Nevertheless, the method from \cite{CoSpVe,CoSpVe2} as well as the previously available direct methods, are based on a fine but explicit construction of the competitor. In the case of the Alt-Caffarelli functional the final competitor depends on an internal variation around the limit cone, which cannot be written down explicitly since in higher dimensions the one-homogeneous solutions are not classified and, even around the known one-homogeneous global solutions, the space of perturbations is very complex. 

{\it In this paper we establish a new constructive approach to the epiperimetric inequalities.} We provide a method to reduce the (log-)epiperimetric inequality \eqref{e:epii} to a quantitative estimate for a functional defined on the unit sphere, which in the case of Alt-Caffarelli depends  
 on the first eigenvalue and the measure of the associated spherical domain with boundary given by a smooth graph over the limit cone.
Using the second variation of this functional with respect to perturbations of the free boundary and a Lyapunov-Schmidt decomposition, we define a flow of spherical domains which by the \L ojasiewicz inequality decreases the energy in a quantitative way. We then construct the competitor slice by slice, defining it on each sphere of radius $r$ as the eigenfunction of the spherical domain at time $1-r$. 

This approach offers a new perspective on epiperimetric inequalities and provides a general method for the construction of a competitor through a flow on the sphere, which only depends on understanding the functional and its second variation restricted to the sphere. We expect it to be flexible enough to apply to a wide variety of problems; this is evidenced by \cite{EnSpVe}, in which we use the ideas in this paper to prove an analogous (log-)epiperimetric inequality for multiplicity-one stationary cones, from which we deduce uniqueness of the blow-up for almost area-minimizing currents. This approach to (log-)epiperimetry also draws a precise relationship, heretofore not understood, between the kernel of the second variation and the logarithmic decay term in the epiperimetric inequality. 

Let us notice that, in the context of minimal surfaces, the logarithmic decay around isolated singularities was obtained by L. Simon \cite{Simon0} using the \L ojasiewicz inequality. Simon's results hinges on the convergence properties for solutions of certain parabolic PDEs. In the contexts he is concerned with, e.g. minimal surfaces, this inequality is applied to a flow naturally associated to the vanishing of the mean curvature: in our case, an analogous flow does not exist, as stationarity is given by a combination of an external variation, which gives $\Delta u=0$ on the positivity set, and an internal variation, which leads to $|\nabla u|=1$ on the free-boundary. 
As such, even though we also use the \L ojasiewicz inequality, our approach is very different from that of \cite{Simon0} in that we do not use any stationarity condition, but rather only the properties of the one-homogeneous blow-up and the functional. Indeed, Theorem \ref{t:epi_smooth} is completely independent of any minimizing or stationary condition, and is simply the existence of a quantitatively good energy competitor with respect to one homogeneous functions. However its application is achieved through an energy comparison argument, which is the reason why we cannot extend our result to stationary points (and indeed we are not aware of any application of the epiperimetric inequality to critical points which are not almost minimizing).
However, if not for a technical obstruction related to {\it a priori} regularity (that is the $C^{2,\alpha}$ regularity of the free-boundary), our approach would work for ``almost-minimizers" to the Alt-Caffarelli function in the sense of \cite{GuyTo} and \cite{EnGuyTo}, which satisfy no PDE. We solve this issue in \cite{EnSpVe}, where we prove an analogous result for almost area-minimizers.

\subsection{Sketch of the proof and plan of the paper}
Before we introduce some additional terminology and notation, let us briefly sketch the main ideas of the proof. Recall that the epiperimetric inequality asks if, given a trace $c:\partial B_1\to \R^+$ on the sphere, which is sufficiently close to a one-homogeneous minimizer $b$, one can construct a competitor $h \in H^1(B_1)$, with $h|_{\partial B_1} = c$, whose energy is quantitatively smaller than that of the one-homogenous extension of $c$. 

We prove Theorem \ref{t:epi_smooth} in  Section \ref{ss:theorem2}. The first step of the proof is to write the trace $c$ as a sum of the Dirichlet eigenfunctions on the spherical domain $\{c > 0\}$. In Subsection \ref{sub:proof_of_epi}, we show that the epiperimetric inequality reduces to two separate estimates, one on the first eigenfunction of $\{c > 0\}$ and another on the higher modes (Proposition \ref{prop:higher_modes} and Proposition \ref{prop:first_mode}). We deal with the higher modes in Subsection \ref{sub:higher_modes} by using a harmonic extension and a cut-off argument in the spirit of \cite{SpVe}. The rest of Section \ref{ss:theorem2} is dedicated to the proof of Proposition \ref{prop:first_mode}, that is the case when the trace is (a multiple of) the first eigenfunction on its support. We notice that in this case replacing the one-homogeneous extension of the trace with the harmonic function on the cone is not sufficient for the quantitative gain from the (log-)epiperimetric inequality since the homogeneity of the harmonic function is very close to one. Thus, we need an ``internal variation" argument in order to smoothly move the trace $c$ towards  the ``closest" one-homogeneous global solution $b$. Nevertheless, writing explicitly such an internal variation in dimension $d\ge 3$ is an extremely difficult (not to say impossible) task since we have infinite directions in which we can perturb a spherical domain. In particular, we cannot exhibit an explicit competitor as in \cite{SpVe,CoSpVe1,CoSpVe2}.   
Thus, instead of trying to write down an explicit internal variation, we look at the problem from another perspective. We briefly explain the main idea below. 

We read the competitor $h:B_1\to\R$ slice be slice, on each sphere $\partial B_r$, as $h(r,\theta)=r \phi_{t(r)}(\theta)$, where $\{t\mapsto\phi_{t}\}_{t\ge 0}$ is a ``flow" of functions on $\partial B_1$, which is then reparametrized over each sphere $\partial B_r$. In particular, we ask that each $\phi_{t}$ is (up to a constant\footnote{In this brief discussion we completely omit the fact that the constant multiplying the first eigenfunction might (and should) also change along the flow and, instead, set it to be a constant, $\kappa_0$. In fact, this constant appears in all the functionals involved and represents a non trivial issue which interferes with definition of the competitor and the notion of integrability.}) the first eigenfunction on the spherical set $\Omega_t=\{\phi_t>0\}$, whose boundary is given by the spherical graph $\zeta(t)$ over $\partial\{b>0\}$, with $\lambda(\zeta(t))$ and $m(\zeta(t))$ being the first eigenvalue and the surface measure of the spherical domain $\Omega_t$. The key observation is that the Weiss' boundary adjusted energy can be dis-integrated along the spheres (slices) of radius $r$ as follows: 
$$W(h)=\int_0^1\big[\mathcal F(\zeta(t(r)))+E.T.(r)\big]r^{d-1}\,dr,\quad \text{where}\quad \mathcal F(\zeta)=\kappa_{0}^2\big(\lambda(\zeta)-(d-1)\big)+\big(m(\zeta)-m(0)\big)$$
and $E.T.(r)$ is an error term such that $E.T.(r)\sim |t'(r)|^2$, which becomes a lower order term once we choose the velocity of the parametrization, $t(r)$, small enough. Thus, we reduce the epiperimetric inequality to the study, in a neighborhood of zero, of the functional $\mathcal F(\zeta)$ defined on the functions $\zeta:\partial\{b>0\}\to\R$. The first variation of $\mathcal F$ vanishes in zero since the trace of the global minimizer $b$ on the sphere inherits the stationarity condition $|\nabla b|=1$ on the free boundary. Thus, starting from a trace determined by the graph of $\zeta$, at second order we find 
\begin{align}
W(h)-(1-\eps)W(z)&\sim \int_0^1\Big(\mathcal F(\zeta(t(r)))-(1-\eps)\mathcal F(\zeta)\Big)r^{d-1}\,dr\notag\\
&\sim \int_0^1\Big(\delta^2\mathcal F(0)\big[\zeta(t(r)),\zeta(t(r))\big]-(1-\eps)\delta^2\mathcal F(0)[\zeta,\zeta]\Big)r^{d-1}\,dr\label{e:blabla}
\end{align}
The second variation $\delta^2\mathcal F(0)$ is a non-local quadratic form over $H^{1/2}(\partial \{b>0\})$ and requires a more careful analysis (we gather the main results in Lemma \ref{l:main} and carry out the proof in Appendix \ref{s:app}). Indeed, the diagonalization of $\delta^2\mathcal F(0)$ splits the space into: a finite dimensional ``index" containing the directions in which $\delta^2\mathcal F(0)$ is negative; a finite dimensional kernel, $\ker(\delta^2\mathcal F(0))$; and an infinite dimensional subspace of $H^{1/2}$ containing all the positive eigenspaces of $\delta^2\mathcal F(0)$. Now, we decompose the initial graph, $\zeta$, as $\zeta=\zeta_-+\zeta_0+\zeta_+$, where $\zeta_-$ is in the index, $\zeta_0\in \ker(\delta^2\mathcal F(0))$ and $\zeta_+$ is a positive direction. The precise definition of the flow requires a Lyapunov-Schmidt reduction (see Appendix \ref{s:LS}); here, in order to simplify the presentation as much as possible, we omit this technical step. The idea
is to define the flow $\zeta(t(r))$ along each direction as follows: along the negative directions the flow remains constant, since these directions have only negative contribution to the energy $\mathcal F$; along the positive direction the flow moves to zero at a small constant speed, that is $\zeta_+(t(r))\sim (1-(d+1)\eps(1-r))\zeta_+$, which assures that the right-hand side of \eqref{e:blabla} is negative and that the error term $E.T.$ is of lower order. Notice that at this point, if we have the additional assumption of integrability through rotations (see Subsection \ref{sub:first_mode}), then we can disregard elements in the kernel of the second variation and assume $\zeta_0=0$, which will conclude the proof.


When the cone is not integrable through rotations, we must address the kernel elements. In this case, let us assume for simplicity that $\zeta=\zeta_0$ and $\mathcal F(\zeta_0)>0$, and define $\zeta_0(t)$ to be the solution at time $t$ of the (rescaled) gradient flow of $\mathcal F$ on the finite dimensional kernel $\ker (\delta^2\mathcal F(0))$
$$\frac{d}{dt}\zeta_0(t)=-\frac{\nabla f(\zeta_0(t))}{|\nabla f(\zeta_0(t))|},\qquad\text{where}\qquad f=\mathcal F\res \ker (\delta^2\mathcal F(0)).$$
The fact that the gradient flow always gives a quantitative improvement follows from the \L ojasiewicz inequality, $ |f|^{1-\gamma}\lesssim |\nabla f|$, for $\gamma\in(0,\sfrac12]$, (see \cite{Loj}) applied to the analytic function $f$ and the fact that $f(\zeta_0(t))$ is decreasing in $t$:
\begin{align*}
\mathcal F(\zeta_0(t))-(1-\eps)\mathcal F(\zeta_0)&=\eps\mathcal F(\zeta_0)+\int_0^t \zeta_0(s)\cdot \nabla f(\zeta_0(s))\,ds=\eps\mathcal F(\zeta_0)-\int_0^t |\nabla f(\zeta_0(s))|\,ds\\
&\le\eps\mathcal F(\zeta_0)-\int_0^t |\mathcal F(\zeta_0(s))|^{1-\gamma}\,ds\le\eps\mathcal F(\zeta_0)-t |\mathcal F(\zeta_0(t))|^{1-\gamma}.
\end{align*}
Assuming that $\mathcal F(\zeta_0(t))\ge \frac12\mathcal F(\zeta_0)$ (otherwise the left-hand side is trivially negative), we obtain 
\begin{align*}
\mathcal F(\zeta_0(t(r)))-(1-\eps)\mathcal F(\zeta_0)&\le\eps\mathcal F(\zeta_0)-2^{\gamma-1}t(r) |\mathcal F(\zeta_0)|^{1-\gamma}.
\end{align*}
Setting $t(r)=\delta(1-r)$, taking into account the error term $E.T.\sim\delta^2$ and integrating in $r$, we get that along the directions of the kernel 
$$W(h)-(1-\eps)W(z)\lesssim \eps|\mathcal F(\zeta_0)|-\delta |\mathcal F(\zeta_0)|^{1-\gamma}+\delta^2,$$
which becomes negative if we choose $\delta\sim |\mathcal F(\zeta_0)|^{1-\gamma}$ for the time of existence of the flow and $\eps\sim |\mathcal F(\zeta_0)|^{1-2\gamma}\sim W(z)^{1-2\gamma}$ for the quantitative gain in the epiperimetric inequality.
\medskip

In Section \ref{ss:conclusion} we apply the (log-)epiperimetric inequality to obtain our main Theorems \ref{t:main} and \ref{t:main2}. Some of this is standard; once one gets the (log-)epiperimetric inequality, the decay of $W(r)$ and thus the power-rate (or logarithmic rate) of convergence to the blowup follows by purely elementary considerations. However, there is an additional difficulty that our (log-)epiperimetric inequality, Theorem \ref{t:epi_smooth}, applies only to minimizers whose free boundary restricted to the sphere is a smooth graph over the cone, $b$. To apply the theorem to arbitrary minimizers which are close enough to a cone with isolated singularity, we employ a compactness argument and use the $\varepsilon$-regularity result of Alt-Caffarelli. 
{Finally, combining this estimate with the (log-)epiperimetric inequality, we show that the argument can be reiterated on a dyadic scale.} 

Finally, in Section \ref{s:DeSiJecone}, we finish the proof of Corollary \ref{c:DeSiJe} by showing that $b_{\nu, \theta_0} \in \mathcal B$ is integrable through rotations. First we produce, with spherical harmonics, an orthogonal basis of $H^{1/2}(\partial \{b_{\nu, \theta_0} > 0\}\cap B_1)$ and, using the symmetries of the cone, show that this is an eigenbasis for the second variation, around $b_{\nu, \theta_0}$, of the Alt-Caffarelli functional restricted to the sphere. Integrability follows after explicitly computing some of the associated eigenvalues.

\section{Notations and preliminary results}\label{ss:notation}

\subsection{Notations}
It will be convenient to split the Weiss energy into two pieces, one without the measure term and one with it, so we set
$$W_0(u):=\int_{B_1}|\nabla u|^2\,dx-\int_{\partial B_1}u^2\,d\HH^{d-1}\qquad\text{and}\qquad W(u):=W_0(u)+\big|B_1\cap\{u>0\}\big|.$$

\noindent We will often work in spherical coordinates: a point $x\in \R^d$ can be written as $(r,\theta)$, where $r=|x|$ and $\theta\in \mathbb{S}^{d-1}=\partial B_1$.
Given a function $u\colon \Omega \to \R$ (almost always a minimizer), a point $x_0$ (almost always in $\partial \{u > 0\}$) and a radius $r > 0$ we can define the rescaled function 
\begin{equation}\label{rescaledu}
u_{r, x_0}(x) : = \frac{u(rx + x_0)}{r} \qquad \mbox{for }0<r<\dist(x_0,\de \Omega).
\end{equation}
When $x_0$ is clear from the context and $r \in \{r_j\}_{j \in \mathbb N}$ we will write $u_j$ to mean $u_{r_j, x_0}$. 

\noindent Finally, let $u$ be a minimizer to \eqref{eq:altcaffunctional} and $x_0 \in \partial \{u > 0\}$. Assume there is a sequence $r_j \downarrow 0$ such that $u_{r_j, x_0} \rightarrow u_\infty$ uniformly on compacta as $j\rightarrow \infty$. We then say that $u_\infty$ is a \emph{blow-up of $u$ at $x_0$}. It follows by the convergence theorems of \cite{AlCa} and the work of \cite{Weiss1} that $u_\infty$ is a one-homogenous, globally-defined minimizer to \eqref{eq:altcaffunctional}.

\subsection{Eigenvalues and eigenfunction on a spherical set $S$}\label{sub:prel:eigen}
Let $S$ be an open subset of the sphere $\partial B_1\subset\R^d$. 
On $S$ we consider the family of Dirichlet eigenfunctions for the Laplace-Beltrami operator, $\{\phi_j^S\}_{j\ge 0}$, and the corresponding sequence of eigenvalues 
$$0\le \lambda_0^S\le\lambda_1^S\le\dots\le\lambda_j^S\le\dots ,$$ 
counted with their multiplicity. Each function $\phi_j$ is a solution of the the PDE
\begin{equation}\label{e:eigenfunction}
-\Delta_{\partial B_1} \phi_j^S=\lambda_j^S\phi_j^S\quad\text{in}\quad S,\qquad \phi_j^S=0\quad\text{on}\quad \partial S,\qquad\int_{\partial B_1}|\phi_j^S|^2\,d\HH^{d-1}=1,
\end{equation}
where $\HH^{d-1}$ is the $(d-1)$-dimensional Hausdorff measure, $\Delta_{\partial B_1}$ is the Laplace-Beltrami operator on the unit sphere $\partial B_1$. Thus, every function $c\in H^1_0(S)$ can be expressed in a unique way in Fourier series as 
$$c(\theta)=\sum_{j=0}^\infty c_j\phi_j^S(\theta),\qquad\text{where, for every $j\ge 0$,}\qquad c_j:=\int_S c(\theta)\phi_j^S(\theta)\,d\HH^{d-1}(\theta).$$
The harmonic extension $h_c$ of $c$ to the cone 
\begin{equation}\label{e:C_S}
\CC_S:=\big\{(r,\theta)\in\R^+\times\partial B_1\ :\ r>0,\ \theta\in S\big\},
\end{equation}
is the solution of the equation 
$$-\Delta h_c=0\quad\text{in}\quad \CC_S\,,\qquad h_c=c\quad\text{on}\quad S\,,\qquad h_c=0\quad\text{on}\quad B_1\cap \partial \CC_S\,,$$
and can be written explicitly in polar coordinates as 
$$h_c(r,\theta)=\sum_{j=0}^\infty c_j r^{\alpha_j^S} \phi_j^S(\theta),$$
where for every $j\in\N$ the homogeneity constant $\alpha_j^S>0$ is uniquely determined by the equation 
$$\alpha_j^S(\alpha_j^S+d-2)=\lambda_j^S\,.$$
In what follows we will often drop the index $S$ and we simply use the notation $\phi_j$, $\lambda_j$ and $\alpha_j$.

\begin{oss}\label{oss:1hom}
[The case of $1$-homogeneous minimizers] Let $b$ be a $1$-homogeneous minimizer of the Alt-Caffarelli functional \eqref{eq:altcaffunctional} and let us denote its trace on $\de B_1$, with a slight abuse of notation, still by $b$. In polar coordinates $b(r,\theta)=r\,b(\theta)$. Let $\Omega:=\{b>0\}\cap\partial B_1$ be the positivity set of $b$ on the unit sphere. Then, using \eqref{e:stationa}, it is easy to see that $\Omega$ is a connected open set and there is a constant $\kappa_0 = \kappa_0(\Omega)>0$ such that 
\begin{equation}\label{e:k_0}
	b\equiv \kappa_0\,\phi_1^{\Omega} \,,
 \qquad\lambda_1^{\Omega}=d-1
\qquad \mbox{and}\qquad \lambda_2^{\Omega}> d-1,
	\end{equation}
where $\phi_1^{\Omega}$ is the normalized eigenfunction given by \eqref{e:eigenfunction} with $S=\Omega$. Moreover, on (the regular part of) the boundary of the spherical set $\Omega$ the following extremality condition is satisfied:  
\begin{equation*}\label{e:k_00}\kappa_0\, \de_\nu \phi_1^{\Omega}(x)=\de_\nu b(x)=-1 \quad \mbox{for every }\ x\in \de \Omega\,,
	\end{equation*}
where $\nu$ denotes the outward pointing normal to $\de \Omega$. 
\end{oss} 
Notice that if two smooth connected open sets are close in $C^1$, then their spectra are also close. Thus $\lambda_2^{S}$ remains strictly greater than $d-1$ for spherical sets $S$ that are close to $\Omega$. We give the precise statement in the following remark. 
\begin{oss}
Suppose that  $\de\Omega$ is smooth and that $\partial S$ is the graph of the function $\zeta:\partial\Omega\to\R$ with $\|\zeta\|_{C^{1,\alpha}}\leq \delta$, where $\delta > 0$ is small enough depending only on the dimension and $b$. Then  \begin{equation}\label{gapforlambda2} 
\lambda_2^S - (d-1) > ( \lambda_2^S-\lambda_2^{\Omega})+ (\lambda_2^{\Omega}-(d-1)) \geq  \gamma(d,b)/2,
\end{equation}
where $\gamma(d,b):=\lambda_2^{\Omega}-(d-1)=\lambda_2^{\Omega}-\lambda_1^{\Omega}>0$ depends only on $b$ and the dimension $d$. 
\end{oss}

\subsection{Internal variations on spherical domains and integrability}\label{ss:var}
Given a smooth domain $\Omega\subset \partial B_1$, determined by a $1$-homogeneous minimizer $b$ as in Remark \ref{oss:1hom}, and a function, $\zeta \in C^{2,\alpha}(\de \Omega)$, we consider the functional
$$
\mathcal F(\zeta):=\kappa_0^2\,\lambda(\zeta)+m(\zeta)-\big(\kappa_0^2\,(d-1)+m(0)\big)\,,
$$
where the various terms are defined in the following way:
\begin{itemize}
	\item $\kappa_0$ is the constant given by \eqref{e:k_0};
	\item $\Omega_\zeta$ is the domain in $\de B_1$ whose boundary is the graph of $\zeta$ over $\de \Omega$ in the sense of \eqref{e:graph};
	\item $\lambda(\zeta)$ is the first Dirichlet-Laplacian eigenvalue of $\Omega_\zeta$ and $\phi_\zeta$ is the corresponding eigenfunction given by \eqref{e:eigenfunction} with $S=\Omega_\zeta$;
	\item $m(\zeta):=\HH^{d-1}(\Omega_\zeta)$ and in particular $m(0):=\HH^{d-1}(\Omega)$.
\end{itemize}
Given two functions $g,\zeta\in C^{2,\alpha}(\partial \Omega)$, the first and the second variations of $\mathcal F$ are given by 
$$\delta\mathcal F(g)[\zeta]=\frac{d}{dt}\Big\vert_{t=0}\mathcal F(g+t\zeta)\qquad\text{and}\qquad \delta^2\mathcal F(g)[\zeta,\zeta]=\frac{d^2}{dt^2}\Big\vert_{t=0}\mathcal F(g+t\zeta).$$
The most important properties of the functional $\mathcal F$ are listed in the following lemma, which is essential for the proof of Theorem \ref{t:epi_smooth}.

\begin{lemma}\label{l:main} 
Let $g,\zeta\in C^{2,\alpha}(\de \Omega)$. With the above definitions the following holds.
\begin{enumerate}[(i)]
\item There is a constant $C$, depending only on $d$ and $b$, such that if $\zeta_i\in C^{2,\alpha}(\partial\Omega)$, for $i=1,\dots, n$, are such that $\|g\|_{C^{2,\alpha}(\partial\Omega)}\le C$ and $\|\zeta_i\|_{C^{2,\alpha}(\partial\Omega)}\le C$, then the function $(t_1,\dots,t_n)\mapsto\mathcal F(g+\sum_{i=1}^nt_i\zeta_i)$ is analytic on $]-2,2[$.
\item The first variation vanishes at zero, that is $\delta\mathcal F(0)[\zeta]=0$, for every $\zeta\in C^{2,\alpha}(\partial\Omega)$.
\item The function $t\mapsto\phi_{g+t\zeta}$ is differentiable in $t\in]-2,2[$ and 
the derivative $\ds\frac{d}{dt}\Big\vert_{t=0}\phi_{g+t\zeta}=:\delta\phi_g[\zeta]\in L^2(\partial B_1)$ satisfies the estimate 		
\begin{equation}\label{e:bound_on_g}
\int_{\partial B_1} |\delta\phi_g[\zeta]|^2\,d\HH^{d-1} < C\,\|\zeta\|_{L^2(\partial\Omega)}^2   \,,
\end{equation}
where $C$ depends only on $d$, $b$ and $\|g\|_{C^{2,\alpha}(\partial\Omega)}$.
\item $\delta^2\mathcal F(g)$ is a quadratic form on $H^{\sfrac12}(\partial\Omega)\times H^{\sfrac12}(\partial\Omega)$.
\item 
There is a function $\omega(t):[0,1] \rightarrow \mathbb R^+$, satisfying $\ds\lim_{t\downarrow 0} \omega(t) = 0$, such that 
	\begin{equation}\label{e:modulus_of_continuity}
	\left|\mathcal \delta^2 \mathcal F(g)[\zeta, \zeta]- \delta^2 \mathcal F (0)[\zeta,\zeta]\right| < \omega\big(\|g\|_{C^{2,\alpha}(\partial\Omega)}\big) \|\zeta\|_{H^{1/2}(\partial\Omega)}^2.                   
	\end{equation} 
	\item There exists an orthonormal basis $(\xi_i)_i$ of $H^{\sfrac12}(\partial\Omega)$ such that $\xi_i\in C^{2,\alpha}(\de \Omega)$, for $i\in\N$, and 
		\begin{equation}\label{e:diagonalization}
		\delta^2\mathcal F(0)[\xi_i,\xi_j]:=\lambda_i \delta_{ij}\qquad \mbox{for every }i,j\in \N\,,
		\end{equation}
		where $\lambda_i\to \infty$ is a non-decreasing sequence of eigenvalues. In particular, there is a constant $C=C(d, b)>0$ such that 
		\begin{equation}\label{e:itr}
		 \dim \ker (\delta^2\mathcal F(0))\leq C<\infty.
		\end{equation}
\end{enumerate}
\end{lemma}
\noindent The proof of this Lemma is rather technical so we give it in Appendix \ref{s:app} in order not to disrupt the flow of the proof. We only remark here that (i) follows from \cite{micheletti,nagy}, (ii) and (iv) follow by the general representation of the first two variations given in Subsection \ref {sub:first_and_second_variation}, (iii) is proved in Subsection \ref{sub:bounds_on_phi'}, (v) in Subsection \ref{sub:modulus_of_continuity} and (vi) in Subsection \ref{sub:variations_in_zero}. Notice that \eqref{e:itr} leads naturally to the definition of integrability through rotations.

\begin{deff}[Integrability]\label{d:integrability}
Let $b\in H^1(B_1)$ be a $1$-homogeneous minimizer of $\mathcal E$. We say that $b$ is integrable through rotation if all the eigenfunctions corresponding to the eigenvalue $0$ are infinitesimal generators of rotations of $\Omega=\{b>0\}$ and if the constant perturbation is orthogonal in $L^2(\partial \Omega)$ to all the negative eigenfunctions, that is
\begin{equation*}
\int_{\partial \Omega} \xi d\mathcal H^{d-2} = 0,\; \mbox{ for every negative eigenfunction $\xi$ of } \delta^2 \mathcal F(0)\,.
\end{equation*}
\end{deff}

\noindent This definition differs slightly from the usual one given in the context of minimal surfaces, in that it requires orthogonality for the \emph{negative eigenfunctions} that is, the eigenfunctions associated to negative eigenvalues. This is due to the fact that $\mathcal F$ can be perturbed not only geometrically (i.e. by changing the domain) but also by changing $\kappa_0$ (the coefficient in front of the first eigenfunction of the domain). We also notice that, by Subsection \ref{sub:first_and_second_variation}, $\ds\int_{\partial \Omega} \xi d\mathcal H^{d-2}=\delta\lambda(0)[\xi]$. 

We shall prove in Section \ref{s:DeSiJecone} that the De Silva-Jerison cone is integrable through rotations by explicitly finding the basis $\{\xi_i\}$ and computing the associated eigenvalues.

We should also note that the De Silva-Jerison cones are the only known minimizing cones with isolated singularity and, since they satisfy Definition \ref{d:integrability}, it is not clear whether there are any cones which are not integrable. However, in examples from other variational problems (e.g. minimal surfaces \cite{AdSi} and obstacle problems \cite{CoSpVe,FiSe})  $C^{1,\alpha}$ regularity may not hold at even isolated critical points and $C^{1,\log}$ regularity, as in Theorem \ref{t:main}, is optimal. If this phenomena also occurs for minimizers of the Alt-Caffarelli functional it must be the case that there exist minimizing cones with isolated singularities that do not satisfy Definition \ref{d:integrability}.

\section{Proof of Theorem \ref{t:epi_smooth}}\label{ss:theorem2}
This section is the core of the paper and deals with the proof of Theorem \ref{t:epi_smooth}. We split it into several parts. We first define the competitor and list several of its properties, from which Theorem \ref{t:epi_smooth} will follow. Then, in the subsequent subsections, we prove that these properties hold through external and internal variations.

\subsection{Definition of the competitor and its main properties}
Given a function $c\in H^1(S; \R_+)$, on the set $S=\{c>0\}\subset \de B_1$, we consider its decomposition in Fourier series over $S$
\begin{equation}\label{decomposec}
c(\theta)=c_1\phi_1(\theta)+g(\theta)\,,\quad\text{where}\quad g(\theta):=\sum_{j=2}^\infty c_j\phi_j(\theta).
\end{equation}
\noindent In this notation, the one-homogeneous extension of $c$ in $B_1$ is given by
\begin{equation}\label{whoisz}
z(r,\theta)=c_1r\phi_1(\theta)+rg(\theta).
\end{equation}
Our \emph{competitor} $h\in H^1(B_1; \R_+)$ will be defined as
\begin{equation}\label{whoish}
h(r,\theta)=h_1(r,\theta)+\psi_{\rho}(r)h_g(r,\theta),
\end{equation}
where 
\begin{itemize}
\item $\psi_{\rho}:B_1\to\R_+$ is the function defined by 
\begin{equation}\label{e:psi_rho}
\Delta\psi_\rho=0\quad\text{in}\quad B_{2\rho}\setminus B_\rho\,,\qquad \psi_\rho=1\quad\text{on}\quad B_1\setminus B_{2\rho}\,,\qquad \psi_\rho=0\quad\text{on}\quad B_\rho\,.
\end{equation}
In particular, $\psi_\rho$ depends only on the variable $r$, $\psi_\rho=\psi_\rho(r)$.
\item The radius $\rho\in(0,1)$ depends only on the dimension $d$ and is chosen in Proposition \ref{prop:higher_modes}.
\item $h_{g}:B_1\to\R$ is the harmonic extension of $g$ to the cone $\CC_S:=\{(r,\sigma)\ :\ r>0,\ \sigma\in S\}$,  
\begin{equation}\label{e:h_g}
\Delta h_g=0\quad\text{in}\quad \CC_S\,,\qquad h_g=g\quad\text{on}\quad \partial B_1\,,\qquad h_g=0\quad\text{on}\quad B_1\setminus \CC_S\, .
\end{equation}
\item $h_1 :B_1\to\R$ is defined by
\begin{equation}\label{e:h_1}
h_1(r, \theta):=
\begin{cases}
c_1\,r\,\phi_1(\theta)	&\mbox{if }r\in[\rho,1]\,,\\
\rho\, \bar h(\sfrac{r}{\rho},\theta)	&\mbox{if }r\in[0,\rho]\,,
\end{cases}
\end{equation}
where $\bar h$ is the competitor from Proposition \ref{prop:first_mode} corresponding to the trace $c_1\phi_1$.
\end{itemize}

The epiperimetric inequality \eqref{e:epii} is then a consequence of the following two propositions. The first deals with the terms corresponding to the higher eigenvalues, for which the energy can easily be improved by taking the harmonic extension.

\begin{prop}[Homogeneity improvement of the higher modes: the external variation]\label{prop:higher_modes}
Let $S\subset\partial B_1$ be an open set and 
$g\in H^1_0(S)$ be a function, expressed in Fourier series over $S$ as 
$$\ds g(\theta)=\sum_{j=k}^\infty c_j\phi_j^S(\theta)\,,\quad\text{where}\quad k\ge 1\quad\text{is such that}\quad\lambda_k^S> d-1.$$ Then, there are constants $\eps_0>0$ and $\rho_0>0$, depending only on the dimension and the gap $\lambda_k^S-(d-1)$, such that for every $0<\eps\le \eps_0$ and $0<\rho\le\rho_0$ we have 
\begin{equation}\label{e:higher_modes}
W_0(\psi_{\rho}h_g)-(1-\eps)W_0(z_g)\le 0,
\end{equation}
where $z_g(r,\theta)=rg(\theta)$, and $\psi_\rho$ and $h_g$ satisfy \eqref{e:psi_rho} and \eqref{e:h_g}, respectively. 
\end{prop}

The second proposition deals with the projection of $c$ onto the first eigenfunction and is more difficult. In this case, the energy term no longer dominates the measure term and the construction of the competitor is more complicated.

\begin{prop}[Epiperimetric inequality for the first mode: the internal variation]\label{prop:first_mode}
Let $b\in H^1(B_1, \R_+)$ be a one-homogeneous minimizer of the Alt-Caffarelli functional. There exists $\eps_1=\eps_1(b,d)>0, \gamma = \gamma(b,d) \in [0,1)$ and $\delta_0=\delta_0(b,d)>0$ such that the following holds. If 
\begin{itemize}
\item $\{c > 0\} = S\subset\partial B_1$ is a connected spherical set whose boundary is given as the (spherical) graph of the function $\zeta \in C^{2,\alpha}(\de \Omega)$ satisfying \eqref{e:epi_ass},
\item $c \equiv \kappa\, \phi_1^{S}:\partial B_1\to\R$ (a multiple of the first Dirichlet eigenfunction of $S$), extended by zero on $\partial B_1\setminus S$, where $\kappa \in \mathbb R$ satisfies $0 < \kappa < 2\kappa_0$ (defined in \eqref{e:k_0}),
\end{itemize}
then there exists a function $\bar h\in H^1(B_1,\R_+)$ such that $\bar h=c$ on $\de B_1$ and 
	\begin{equation}\label{e:epi}
	W(\bar h)-W(b)\leq \Big(1-\eps_1\big|W(z)-W(b)\big|^\gamma\Big) \big(W(z)-W(b)\big)
	\end{equation}
	where $z$ is the $1$-homogeneous extension of $c$ to $B_1$. Moreover, for every $r\in(0,1)$, $\bar h(r,\cdot)$ is a positive multiple of the first Dirichlet eigenfunction on the spherical set $S_r:=\{\bar h(r,\cdot)>0\} \subset \partial B_r$. 
Furthermore, if $b$ is integrable through rotations, then we can take $\gamma \equiv 0$ in \eqref{e:epi}. 
\end{prop}

\subsection{Proof of Theorem \ref{t:epi_smooth}.} \label{sub:proof_of_epi}
Recall that $z$ is the one-homogenous extension of $c \in L^2(\partial B_1)$ into the ball. We write $z = z_1 + z_g$ where, in the notation of \eqref{decomposec}-\eqref{whoisz},
$$
z_1(r,\theta):=c_1r\phi_1(\theta)\qquad\text{and}\qquad z_g(r,\theta):=rg(\theta).
$$ 
Since the eigenfunctions $\{\phi_j\}_{j\ge 1}$ are orthogonal in $L^2(\partial B_1)$ and $H^1(\partial B_1)$, we get that 
$$
W_0(z_1+z_g)=W_0(z_1)+W_0(z_g).
$$
Moreover, since the set $S=\{c>0\}$ is connected, we have that $\{\phi_1>0\}=\{c>0\}$. Thus 
$$
W(z)=W(z_1+z_g)=W(z_1)+W_0(z_g).
$$
We notice that, for every $r\in(0,1)$, the functions $h_1(r,\cdot)$ and $\psi_{\rho_0} h_g(r,\cdot)$ are orthogonal in $L^2(\partial B_1)$, as well as in $H^1(\partial B_1)$. Indeed, $\psi_\rho h_g(r,\cdot)\equiv 0$ if $r\le \rho_0$, while for $r\in [\rho_0,1]$ the claim follows by the Fourier decomposition of $h_g$ (see Subsection \ref{sub:higher_modes}) and the orthogonality of the eigenfunctions (recall that $\rho_0$ is the constant whose existence is guaranteed by Proposition \ref{prop:higher_modes}). Thus, 
$$
W_0(h_1+\psi_{\rho_0} h_g)=W_0(h_1)+W_0(\psi_\rho h_g).
$$
Moreover, we have that 
\begin{equation*}
\begin{array}{ll}
\ds\{h(r,\cdot)>0\}\subset S=\{h_1(r,\cdot)>0\}\quad \text{for every}\quad r\in [\rho_0,1],\\
\\
\qquad\qquad \ds\{h(r,\cdot)>0\}=\{h_1(r,\cdot)>0\}\quad \text{for every}\quad r\in [0,\rho_0].
\end{array}
\end{equation*}
Thus, 
$$
W(h)=W(h_1+\psi_\rho h_g)\le W(h_1)+W_0(\psi_\rho h_g)\,,
$$
and, for every $0<\eps\le\eps_0$ ($\eps_0$ being the constant from Proposition \ref{prop:higher_modes}), we have 
\begin{align}
\big(W(h)-W(b)\big)&-(1-\eps)\big(W(z)-W(b)\big)\nonumber\\
&\le \big(W(h_1)+W_0(\psi_{\rho}h_g)-W(b)\big)-(1-\eps)\big(W(z_1)+W_0(z_g)-W(b)\big)\nonumber\\
&= W(h_1)-W(b)-(1-\eps)\big(W(z_1)-W(b)\big)+W_0(\psi_{\rho}h_g)-(1-\eps)W_0(z_g)\nonumber\\
&\le W(h_1)-W(b)-(1-\eps)\big(W(z_1)-W(b)\big) + (\eps - \eps_0)W_0(z_g),\label{e:high_modes_est}
\end{align}
where the last inequality follows by Proposition \ref{prop:higher_modes}. Note that, since by \eqref{gapforlambda2} the gap $\lambda_2^S -(d-1) > \gamma(d,b)$, the constant $\eps_0 > 0$ depends only on the dimension and the cone $b$. Also note that if $W_0(z_g) > (W(z_1) - W(b))$, then we can let $h_1 = z_1$ and $\eps = \eps_0/2$ and get the epiperimetric inequality. Thus we can assume that $W_0(z_g) \leq \big(W(z_1) - W(b)\big)$, so that
\begin{equation}\label{e:wzcomptow} 
W(z) - W(b) \leq 2\big(W(z_1) - W(b)\big).
\end{equation}
In order to conclude the proof, we notice that 
\begin{align*}
W(h_1)&=\int_{B_1} |\nabla h_1|^2-\int_{\partial B_1}h_1^2+\big|\{h_1>0\}\cap B_1\big|\\
&=\int_{B_1\setminus B_\rho} |\nabla z_1|^2-\int_{\partial B_1}z_1^2+\big|\{z_1>0\}\cap (B_1\setminus B_\rho) \big| +\int_{B_\rho} |\nabla h_1|^2+\big|\{h_1>0\}\cap B_\rho \big|\\
&=\int_\rho^1\left[\int_{\partial B_1} \left(|\nabla_\theta z_1|^2-(d-1)z_1^2+\ind_{\{z_1>0\}}\right)\right]r^{d-1}\,dr\\
&\qquad  +\int_{B_\rho} |\nabla h_1|^2-d\int_0^\rho\left[\int_{\partial B_1} z_1^2\right]r^{d-1}\,dr+\big|\{ h_1>0\}\cap B_\rho \big|\\
&=(1-\rho^d)\int_0^1\left[\int_{\partial B_1} \left(|\nabla_\theta z_1|^2-(d-1)z_1^2+\ind_{\{z_1>0\}}\right)\right]r^{d-1}\,dr\\
&\qquad  +\int_{B_\rho} |\nabla h_1|^2-\rho^d \int_{\partial B_1} z_1^2+\big|\{h_1>0\}\cap B_\rho \big|\\
&=(1-\rho^d)W(z_1)+\rho^d W(h_1,\rho)=(1-\rho^d)W(z_1)+\rho^d W(\bar h).
\end{align*}
Let $\bar\eps>0$. By Proposition \ref{prop:first_mode}, we then have 
\begin{align}
\big(W(h_1)&-W(b)\big)-(1-\bar\eps)\big(W(z_1)-W(b)\big)\notag\\
&=(1-\rho^d)\big(W(z_1)-W(b)\big)+\rho^d \big(W(\bar h)-W(b)\big)-(1-\bar\eps)\big(W(z_1)-W(b)\big)\notag\\
&\le (\bar\eps-\rho^d)\big(W(z_1)-W(b)\big)+\rho^d \Big(1-\eps_1\large|W(z_1) - W(b)\large|^\gamma\Big)\big(W(z_1)-W(b)\big)\notag\\
&= \Big(\bar\eps-\rho^d \eps_1\large|W(z_1)-W(b)\large|^\gamma\Big)\big(W(z_1)-W(b)\big).\label{e:almost_last_est}
\end{align}
Note, all the assumptions of Proposition \ref{prop:first_mode} are obviously satisfied except for the condition that $0 < c_1 < 2\kappa_0$. This follows from the hypothesis in Theorem \ref{t:epi_smooth} that $\|c- b\|_{L^2} < \delta$ and the fact that $b=\kappa_0\,\phi_1$ on $\de B_1$ (see \eqref{e:k_0}). Note that by \eqref{e:wzcomptow} we can replace the term $\eps_1\large|W(z_1)-W(b)\large|^\gamma$ by the smaller $2^{-\gamma} \eps_1\large|W(z) - W(b)\large|^\gamma$. Setting $\rho=\rho_0$ and 
$$\bar\eps=\min\big\{\eps_0,\rho_0^d2^{-\gamma} \eps_1\large|W(z) - W(b)\large|^\gamma \big\},$$
and using the estimates \eqref{e:high_modes_est} and \eqref{e:almost_last_est}, we obtain the epiperimetric inequality \eqref{e:epii} with $\eps=\rho_0^d2^{-\gamma} \eps_1$, 
where $\eps_0$, $\rho_0, \gamma$ and $\eps_1$ are the constants from Proposition \ref{prop:higher_modes} and Proposition \ref{prop:first_mode}, depending only on $b$ and the dimension $d$.\qed\\

We prove Proposition \ref{prop:higher_modes} in Subsection \ref{sub:higher_modes} using a general argument from \cite{SpVe}. The proof of Proposition \ref{prop:first_mode} is more involved and is contained in Subsections \ref{sub:slicing}, \ref{sub:first_mode} and \ref{sub_1_mode}.

\subsection{Homogeneity improvement for the higher modes: proof of Proposition {\ref{prop:higher_modes}}}\label{sub:higher_modes}
In this subsection we prove Proposition \ref{prop:higher_modes}. The proof will be a consequence of the following two Lemmas from \cite{SpVe} of which we recall the statements below. The first lemma shows how the harmonic extension of the high modes has smaller energy than the $1$-homogeneous extension. 

\begin{lemma}[Harmonic extension  {\cite[Lemma 2.5]{SpVe}}]\label{l:harm_ext}
Let $S\subset \partial B_1$ be an open subset of the unit sphere and $g$, $h_g$ and $z_g$ be as in Proposition \ref{prop:higher_modes}. Then $W_0(z_g)\geq 0$ and
\begin{equation}\label{e:higher_improvement}
W_0(h_g)-(1-\eps)W_0(z_g)\le 0\, \quad\text{for every}\quad 0<\eps\le\frac{\alpha_k-1}{d+\alpha_k-1}\,,
\end{equation}
where $k$ is as in Proposition \ref{prop:higher_modes} and $\alpha_k=\alpha_k^S$ is the exponent from Subsection \ref{sub:prel:eigen}.
\end{lemma}

\noindent The second lemma shows that properly cutting-off the harmonic competitor near the origin preserves the energy improvement. This is needed to preserve the orthogonality between $h_g$ and $h_1$.

\begin{lemma}[Truncation of the harmonic extension  {\cite[Lemma 2.6]{SpVe}}]\label{l:mc_competitor}
Let $S\subset \partial B_1$ be an open subset of the unit sphere and $g$, $h_g$ and $z_g$ be as in Proposition \ref{prop:higher_modes}. Let $\rho>0$ and $\psi:=\psi_{\rho}$ be given by \eqref{e:psi_rho}. Then, there is a dimensional constant $C_d>0$ such that 
\begin{equation}\label{e:gigi}
W_0(\psi_\rho h_g)\le\left(1+\frac{C_d\,\alpha_k\,(2\rho)^{2(\alpha_k-1)+d}}{\alpha_k-1}\right)W_0(h_g),
\end{equation}
where $\alpha_k=\alpha_k^S$ is the homogeneity exponent corresponding to the eigenvalue $\lambda_k^S$ (see Subsection \ref{sub:prel:eigen}).
\end{lemma}

\noindent {\bf Proof of Proposition \ref{prop:higher_modes}.} We only consider the case $W_0(z_g)>0$ since otherwise the statement is trivial. Thus, it is sufficient to prove \eqref{e:higher_modes} for $\eps=\eps_0$.
By Lemma \ref{l:harm_ext} and Lemma \ref{l:mc_competitor} we have 
\begin{align*}
W_0(\psi_\rho h_g)&\le \left(1+\frac{C_d\,\alpha_k\,(2\rho)^{2(\alpha_k-1)+d}}{\alpha_k-1}\right)W_0(h_g)\\
&\le\left(1+\frac{C_d\,\alpha_k\,(2\rho_0)^{2(\alpha_k-1)+d}}{\alpha_k-1}\right)\left(1-\frac{\alpha_k-1}{d+\alpha_k-1}\right)W_0(z_g)\le(1-\eps_0)W_0(z_g),
\end{align*}
for some constants $\eps_0>0$ and $\rho_0>0$ depending only on the dimension and the gap, $\alpha_k-1>0$.\qed

\subsection{Internal variation and slicing}\label{sub:slicing}
This subsection contains a preliminary result for the proof of Proposition \ref{prop:first_mode}. We treat it separately since it offers a new perspective on the epiperimetric inequality. 
Suppose that $u:B_1\to \R^+$ is a minimizer of the Alt-Caffarelli functional $\mathcal E$ in $B_1$ such that $0\in\partial\{u>0\}$ and $b$ is a blow-up of $u$ at $0$. By definition, $u$ is the best choice for a test function in the left-hand side in the epiperimetric inequality 
$$
W(u)-W(b)\le (1-\eps)\big(W(z)-W(b)\big).
$$
The Lipschitz continuity of $u$ implies that using the coordinates $(r,\theta)\in \R^+\times \mathbb{S}^{d-1}$, the function $u$ can be written in the form $u(r,\theta)=rv_r(\theta)$, where  $v_r\in L^2(\partial B_1)$ for each $r \in (0, 1]$. Thus any minimizer, $u$, can be identified with the flow $r\mapsto v_r\in L^2(\partial B_1)$ and so it is not restrictive to search for a competitor $h$, written directly in the form $h(r,\theta)=rv_r(\theta)$. 
 
We introduce the following functional, which will be helpful in our analysis. For every function $\tilde{c}\in H^1(\partial B_1)$ on the $(d-1)$-dimensional sphere $\partial B_1\subset\R^d$, we define 
\begin{equation}\label{eqn:defofEandE0}
\begin{aligned}
 E(\tilde{c}):=&\int_{\partial B_1} \Big(|\nabla_{\theta} \tilde{c}|^2-(d-1)\tilde{c}^2\Big)\,d\HH^{d-1}+\HH^{d-1}\big(\{\tilde{c}>0\}\cap\partial B_1\big),\\
 E_0(\tilde{c}):=& \int_{\partial B_1}  \Big(|\nabla_{\theta} \tilde{c}|^2-(d-1)\tilde{c}^2\Big)\,d\HH^{d-1}.\end{aligned}
\end{equation}

\begin{oss}
We notice that for the one-homogeneous function $z_{\tilde{c}}(r,\theta)=r\,\tilde{c}(\theta)$ we have 
$$W(z_{\tilde{c}})=\int_0^1\left[\int_{\partial B_1} \Big(|\nabla_\theta \tilde{c}|^2-(d-1)\tilde{c}^2+\ind_{\{\tilde{c}>0\}}\Big)\,d\HH^{d-1}(\theta)\right]\,r^{d-1}\,dr=\frac{1}{d}E(\tilde{c}).$$
In particular, $E(\tilde{c})=d\, W(b)$ if $\tilde{c}$ is the trace of the one-homogeneous solution $b$ on the cone $\{b>0\}\cap \de B_1$. Moreover, if $\tilde{c}$ is the first eigenfunction, i.e. $\tilde{c}=\phi_1^{\Omega_\zeta}$, for some domain $\Omega_\zeta\subset \de B_1$ whose boundary $\de \Omega_\zeta$ is the graph of $\zeta$ over $\de \Omega$, then in the notation of Subsection \ref{ss:var} we have
$$E(\tilde{c})=\lambda(\zeta)-(d-1)+m(\zeta)=\mathcal F(\zeta)+m(0).$$ 
\end{oss}
\noindent The following lemma relates the energy, $W$, of a function, $u$, to the energy, $E$, of its slices, $\phi_r$.  

\begin{lemma}[Slicing Lemma]\label{l:slicing}
Let $v:\R^+\times\partial B_1\ni (t,\theta)\mapsto v_t(\theta)$ be such that $v(t,\theta) \in H^1([0,1]\times \partial B_1)$, $v_t\in H^1(\partial B_1)$ for every $t\in \R^+$ and $t\mapsto v_t$ is Lipschitz continuous as a function with values in $L^2(\partial B_1)$. Let $\eta:[0,1]\to\R^+$ be Lipschitz continuous. We set 
$$w(\rho,\theta)=\rho\,v_{\eta(\rho)}(\theta)\qquad\text{and}\qquad z_v(\rho,\theta)=\rho\,v_{\eta(1)}(\theta).$$ 
Then, for every $\eps\ge 0$, we have  
\begin{equation}\label{eqn:reductiontofandg}\begin{aligned}
W(w)-(1-\eps)W(z_v)
&=\int_0^1\Big(E(v_{\eta(r)})-(1-\eps)E(v_{\eta(1)})\Big)r^{d-1}\,dr\\
&\qquad\qquad\qquad+\int_0^1r^{d+1} (\eta'(r))^2 \int_{\partial B_1} |\partial_t v_{\eta(r)}|^2 \,dr.
\end{aligned}
\end{equation}
Moreover, the same inequality holds with $W_0$ and $E_0$ (defined as in  in place of $W$ and $E$.
\end{lemma}
\begin{proof}
We first calculate the energy of $w$.
\begin{align*}
W_0(w)&=\int_{B_1}|\nabla w|^2\,dx-\int_{\partial B_1}w^2\ d\theta\\
&=\int_0^1r^{d-1}\,\int_{\partial B_1} \Big(\big(v_{\eta(r)}+r\eta'(r)\de_tv_{\eta(r)}\big)^2+\frac1{r^2}(r|\nabla_\theta v_{\eta(r)}|)^2\Big)\,d\theta\ dr-\int_{\partial B_1}v_{\eta(1)}^2\,d\theta\\
&=\int_0^1r^{d-1}\,\int_{\partial B_1} \Big(v^2_{\eta(r)}+r\partial_r\left(v^2_{\eta(r)}\right)+r^2(\eta'(r))^2|\de_t{v}_{\eta(r)}|^2+|\nabla_\theta v_{\eta(r)}|^2\Big)\,d\theta\ dr-\int_{\partial B_1}v_{\eta(1)}^2\,d\theta\\
&=\int_0^1r^{d-1}\,\,\int_{\partial B_1} \Big( (\eta'(r))^2r^2|\de_t{v}_{\eta(r)}|^2+|\nabla_\theta v_{\eta(r)}|^2-(d-1)v_{\eta(r)}^2\Big)\,d\theta\ dr,
\end{align*}
where in the last step we integrated by parts in $r$.  Thus, we get 
\begin{align*}
W(w)&=W_0(w)+\big|B_1\cap\{w>0\}\big|\\
&=\int_0^1\left[\int_{\partial B_1} \Big((\eta'(r))^2r^2|\de_t{v}_{\eta(r)}|^2+|\nabla_\theta v_{\eta(r)}|^2-(d-1)v_{\eta(r)}^2\Big)\,d\theta\right.\\
&\qquad\qquad\qquad\qquad\qquad\qquad\qquad\qquad+\left.\mathcal H^{d-1}\big(\partial B_1\cap \{v_{\eta(r)}>0\}\big)\right]r^{d-1}\,dr.
\end{align*}
Analogously, the energy of the one-homogeneous extension $z$ is given by
$$W(z_v)=\int_0^1\left[\int_{\partial B_1} \Big(|\nabla_\theta v_{\eta(1)}|^2-(d-1)v_{\eta(1)}^2\Big)\,d\theta+\mathcal H^{d-1}\big(\partial B_1\cap \{v_{\eta(1)}>0\}\big)\right]r^{d-1}\,dr.\qquad\qedhere$$\end{proof}
The identity \eqref{eqn:reductiontofandg} suggests that in order to obtain the epiperimetric inequality \eqref{e:epi} we have to construct a flow, $t\mapsto v_t$, that decreases the energy $\frac1{d}E(v_t)-W(b)$. To do this we will look at the spectrum of the bilinear form that corresponds to the second variation of $\mathcal F$. On the other hand, changing the initial trace $c=v_{\eta(1)}$ has a cost due to the last term in \eqref{eqn:reductiontofandg}. The balance between this last term and the energy $E$ will be the main subject of the following subsection.

\subsection{Proof of Proposition \ref{prop:first_mode}}\label{sub_1_mode} 
Let $K:=\ker (\delta^2\mathcal F(0))\subset H^{\sfrac12}(\de \Omega)$, $N:=\dim K$ and $p_K, p_{K^\perp}$ be the $L^2(\de \Omega)$-projections on $K$ and $K^\perp$, respectively. Let $\zeta \in {C^{2,\alpha}(\de \Omega)}$ be given (as in \eqref{e:epi_ass}) and 
$$
\zeta=\zeta^T+\zeta^\perp=\sum_{j=1}^N\mu_j\iota_j+\zeta^\perp\,,\ \mbox{ where }\ \zeta^\perp=p_{K^\perp}\zeta\,\,\mbox{ and $\{\iota_1,\dots,\iota_N\}$ is an orthonormal basis of $K$.  }
$$ 
Moreover, slightly abusing the notation, we will write $\mu^0:=\sum_{j=1}^N\mu_j\iota_j \in K $, and also $\mu^0=(\mu_1,\dots,\mu_N)\in \R^N$, so that $|\mu^0|=\|\zeta^T\|_{H^{\sfrac12}(\de \Omega)}$. 

\noindent Let $\Upsilon \colon K\times\R \to K^\perp$ be the map of Lemma \ref{l:LS} and let $\tilde\zeta:=\zeta^\perp-\Upsilon(\mu^0,s^0) \in K^\perp$. Then we can decompose $\tilde \zeta=\zeta_1+\zeta_2$, where $\zeta_1,\zeta_2$ are defined by
$$
\zeta_1:=\sum_{\lambda_i>0}p_i \xi_i
\qquad \mbox{and}\qquad
\zeta_2:=\sum_{\lambda_i<0}p_i\,\xi_i\,,
$$ 
where $(\xi_i)_i$ is the orthonormal basis of $K^\perp$ given by Lemma \ref{l:main} (so $(\xi_i)_{i\in \N}\cup (\iota_j)_{j=1}^N$ is an orthonormal basis of $H^{\sfrac12}(\de \Omega)$) and $\lambda_i$ is the eigenvalue of $\delta^2\mathcal F(0)$ relative to $\xi_i$. 

\noindent Finally recall the function $\mathcal G: C^{2,\alpha}(\partial \Omega)\oplus \R \rightarrow \R$ defined in \eqref{e:defofG} 
$$\mathcal G(\zeta,s) = (\kappa_0^2 + s^3)(\lambda(\zeta) - (d-1)) + m(\zeta)-m(0).$$

\noindent{\bf Definition of the competitor.} We will define the competitor $\bar h(r,\theta)$ in such a way that, for each $r > 0$, the positivity set $\{\overline{h}(r, \cdot) > 0\} \subset \partial B_1$ will be given by the spherical set $\Omega_{g_r}$, whose boundary is the graph of $g_r:\partial\Omega\to\R$, where
\begin{equation}\label{e:competitor}
g_r(\theta)=g(r,\theta):=\eta_1(r) \zeta_1(\theta)+\eta_2(r)\,\zeta_2(\theta)+\Bigl(\mu(\eta_3(r))+\Upsilon\big(\mu(\eta_3(r)), s(\eta_3(r))\big)\Bigr)\,,
\end{equation}
where the function $\R\ni t\mapsto (\mu(t), s(t))\in\R^N\oplus \R$ is defined through the gradient flow
\begin{equation}\label{e:grad_flow}
\begin{cases}
\ds(\mu'(t),s'(t))=\frac{-\nabla G(\mu(t), s(t))}{|\nabla G(\mu(t), s(t))|} \\
(\mu(0), s(0))=(\mu^0, s^0),
\end{cases}
\end{equation}
 the function $G\colon \R^N \oplus \R\to \R$ is the analytic function of Lemma \ref{l:LS} defined as $G(\mu,s):=\mathcal G(\mu+\Upsilon(\mu,s),s)$ and $s^0$ will be defined below.  If $|\nabla G(\mu, s)| = 0$, then we set $(\mu', s') = 0$. The functions, $\eta_1,\eta_2,\eta_3$, will be chosen later, with the properties that $\eta_1(1)=1=\eta_2(1)$ and $\eta_3(1)=0$, so that
$$
g(1,\cdot)=\zeta_1+\zeta_2+\Bigl(\mu^0+\Upsilon(\mu^0, s^0) \Bigr)=\zeta^\perp-\Upsilon(\mu^0, s^0)+\Bigl(\zeta^T+\Upsilon(\mu^0, s^0) \Bigr)=\zeta^\perp+\zeta^T=\zeta\,.
$$ 
For every $r\in (0,1)$, we denote the first Dirichlet eigenfunction, the first eigenvalue and the measure of $\Omega_{g_r}$ by
$$
\phi_{g_r}:=\phi_1^{\Omega_{g_r}}\ ,
\qquad \lambda(g_r):=\lambda_1^{\Omega_{g_r}}
\qquad\mbox{and}\qquad
 m(g_r):=\HH^{d-1}(\Omega_{g_r})\,.
$$

Recall that the one-homogeneous extension $z$ of the trace $c=\kappa \,\phi_\zeta=\kappa \,\phi_{g_1}$ is given by $z(r, \theta) = r \kappa\phi_{g_1}(\theta)$. We define the competitor $\bar h$ in polar coordinates as 
\begin{equation}\label{eqn:conicalcompetitor}
\bar h(r, \theta) = r \kappa_{r}\phi_{g_r}(\theta),
\end{equation}
where $\kappa_r^2 = \kappa_0^2 + s^3(\eta_3(r))$, and $s(t)$ is given by \eqref{e:grad_flow}, where $s(0)=s^0\in \R$ is chosen such that $\kappa_0^2 + (s^0)^3 = \kappa^2$. 
By the Slicing Lemma \ref{l:slicing}, we have (in the notation of \eqref{e:defofG})
\begin{equation}\label{e:useslicing}
\begin{aligned}
\Big|\Big(W&\big(\bar h(r,\theta)\big) - (1-\eps)W(z)+\eps W(b)\Big) - \int_0^1 \Big[\mathcal G\big(g_r, s(\eta_3(r))\big) - (1-\eps)\mathcal G(\zeta, s^0)\Big]r^{d-1}\, dr\Big|\\
&\le 2\int_0^1 r^{d+1} \frac{9s(\eta_3(r))^4}{\kappa_r^2} [s'(\eta_3(r))]^2 [\eta'_3(r)]^2\ dr + 2\int_0^1 \kappa_r^2 r^{d+1} \int_{\partial B_1} \big|\delta\phi_{g_r}[g'(r)]\big|^2 \,d\sigma \,dr.
\end{aligned}
\end{equation}
Setting  for simplicity $s=s(\eta_3(r))$, $\mu=\mu(\eta_3(r))$ and $g=g(r)$, we can write
\begin{align}\label{e:exp_h}
\mathcal G(g, s)-(1-\eps)\mathcal G(\zeta, s^0)
	&=\mathcal G(g, s)-\mathcal G\big(\mu+\Upsilon(\mu, s),s\big)+\mathcal G\big(\mu+\Upsilon(\mu,s),s\big)\notag\\
	&\qquad-(1-\eps)\Bigl( \mathcal G(\zeta, s^0)-\mathcal G\big(\mu^0+\Upsilon(\mu^0, s^0), s^0\big)+\mathcal G\big(\mu^0+\Upsilon(\mu^0,s^0), s^0\big) \Bigr)\notag\\
	&=\underbrace{\mathcal G(g,s)-\mathcal G\big(\mu+\Upsilon(\mu,s),s\big)-(1-\eps) \Big(\mathcal G(\zeta,s^0)-\mathcal G\big(\mu^0+\Upsilon(\mu^0,s^0),s^0\big)\Big)}_{=:E^\perp}\notag\\ 
	&\qquad+\underbrace{G(\mu,s)-(1-\eps)G\big(\mu^0,s^0\big)}_{=:E^T}\,.
\end{align}
We will estimate separately $E^\perp$, $E^T$ and the two error terms in the right-hand side of \eqref{e:useslicing}.
\smallskip

\noindent{\bf Estimate of $E^\perp$.} For what concerns $E^\perp$, we first notice that, by \eqref{e:competitor}, 
$$g(r)-\mu(\eta_3(r))-\Upsilon\big(\mu(\eta_3(r)),s(\eta_3(r))\big)=\eta_1(r)\zeta_1+\eta_2(r)\zeta_2=:\hat\zeta(r)\in K^\perp.$$
Using the identity
$$f(1)=f(0)+f'(0)+\frac12A+\int_{0}^1(1-t)\big(f''(t)-A\big)\,dt\,,\quad\text{for}\quad A\in\R,$$ 
for the function $f(t):=\mathcal G\big(\mu+\Upsilon(\mu,s)+t\hat\zeta,s\big)$ and $\ds A=\delta^{2}\mathcal F(0)[\hat\zeta,\hat\zeta]$ we get
 \begin{align}
\Big|\Big(\mathcal G(g,s)-&\mathcal G\big(\mu+\Upsilon(\mu,s),s\big)\Big)- \frac12\delta^{2}\mathcal F(0)[\hat\zeta,\hat\zeta]\Big|\notag\\
	 &\le {\sup_{t\in [0,1]}\left|\delta^{2}\mathcal G\big(\mu+\Upsilon(\mu,s)+t\hat\zeta,s \big)-\delta^2\mathcal F(0)\right|}\,\|\hat\zeta\|^2_{H^{\sfrac12}}=:S_1\,\|\hat\zeta\|^2_{H^{\sfrac12}},\label{e:S_1}
 \end{align}
 where we used the fact that $\delta^2\mathcal G(0,0) = \delta^2 \mathcal F(0)$ and,  by \eqref{e:LS2},  
 $$f'(0)=\delta\mathcal G\big(\mu+\Upsilon(\mu,s),s\big)[(\hat\zeta,0)]=P_{K^\perp}\left(\delta\mathcal G\big(\mu+\Upsilon(\mu,s),s\big)\right)[(\hat\zeta,0)]=0.$$
 Moreover, since $\hat \zeta(r) =\eta_1(r)\zeta_1+\eta_2(r)\zeta_2$ and $\zeta_1$ and $\zeta_2$ are sums of orthogonal eigenfunctions of $\delta^2\mathcal F(0)$, we get that 
 $$\delta^{2}\mathcal F(0)[\hat\zeta,\hat\zeta]=\delta^2\mathcal F(0)[\eta_1 \zeta_1+\eta_2\zeta_2, \eta_1 \zeta_1+\eta_2\zeta_2]=\eta_1^2\,\delta^2\mathcal F(0)[\zeta_1,\zeta_1]+\eta_2^2\,\delta^2\mathcal F(0)[\zeta_2,\zeta_2].$$
Analogously, since $\zeta=\mu^0+\Upsilon(\mu^0,s^0)+\tilde\zeta$ we have
\begin{align}
  \Big|\Big(\mathcal G(\zeta,s^0)-&\mathcal G \big(\mu^0+\Upsilon(\mu^0,s^0),s^0\big)\Big)-
	   \frac12\Big(\delta^2\mathcal F(0)[\zeta_1,\zeta_1]+\delta^2\mathcal F(0)[\zeta_2,\zeta_2]\Big)\Big|\notag\\
	  &\le {\sup_{t\in [0,1]}\left|\delta^{2}\mathcal G\big(\mu^0+\Upsilon(\mu^0, s^0)+t \tilde\zeta,s^0 \big)-\delta^2\mathcal F(0)\right|}\,\|\tilde\zeta\|^2_{H^{\sfrac12}}=: S_2\,\|\tilde\zeta\|^2_{H^{\sfrac12}}.\label{e:S_2}
 \end{align}
 Choosing $\eta_1(r) = 1- (d+1)\eps(1-r)$
 and $\eta_2(r) = 1$ we get that $\|\hat\zeta(r)\|^2_{H^{\sfrac12}}\le\|\zeta_1\|^2_{H^{\sfrac12}}+\|\zeta_2\|^2_{H^{\sfrac12}}=\|\tilde\zeta\|^2_{H^{\sfrac12}}$. 
 Thus, combining the inequalities \eqref{e:S_1} and \eqref{e:S_2}, we get that there is a dimensional constant $C_d>0$ such that for every $\eps\le \sfrac14$ we have
 \begin{align*}
 \int_0^1E^\perp r^{d-1}\,dr	
	 &\leq \frac12\delta^2\mathcal F (0)[\zeta_1,\zeta_1]\int_0^1(\eta_1^2(r)-(1-\eps))r^{d-1}\,dr\notag\\
	 &\quad+\frac12\delta^2\mathcal F (0)[\zeta_2,\zeta_2]\int_0^1(\eta_2^2(r)-(1-\eps))r^{d-1}\,dr+\int_{0}^1\left(|S_1|+|S_2|\right)\,r^{d-1}\,dr\,\|\tilde\zeta\|^2_{H^{\sfrac12}}\notag\\
	 &\leq   C_d\,\eps\left(-\left(\min_{\{i\mid \lambda_i > 0\}} \lambda_i\right)\|\zeta_1\|^2+\left(\max_{\{i\mid\lambda_i < 0\}}\lambda_i\right) \|\zeta_2\|^2 \right)r+\int_{0}^1\left(|S_1|+|S_2|\right)\,r^{d-1}\,dr\,\|\tilde\zeta\|^2_{H^{\sfrac12}}\notag\\
      &\leq \left(- C_{d, b} \eps +   \int_{0}^1\left(|S_1|+|S_2|\right)\,r^{d-1}\,dr\right)\,\|\tilde\zeta\|^2_{H^{\sfrac12}}\,.                
 \end{align*}
In order to bound $S_1$ and  $S_2$, we notice that the second variation of $\mathcal G$ is given by
\begin{align*} \delta^2\mathcal G(g,s)[(\zeta,0), (\zeta,0)]=\, \delta^2 \mathcal F(g)[\zeta, \zeta]+ s^3 \delta^2 \lambda(g)[\zeta, \zeta]\,, \quad\text{for every}\quad s\in\R\,,\  g,\zeta\in C^{2,\alpha}(\partial\Omega).\end{align*}
Applying this formula to the second variation $\delta^{2}\mathcal G\big(\mu+\Upsilon(\mu,s)+t\hat\zeta,s \big)$ in the directions $(\hat\zeta,0)$ and $(\tilde \zeta,0)$,
 then using \eqref{e:modulus_of_continuity} and the estimate $\delta^2 \lambda(g)[\tilde\zeta, \tilde\zeta] \leq C\|\tilde\zeta\|_{H^{\sfrac12}}^2$ (see Subsection \ref{sub:modulus_of_continuity}), we get that there is a 
modulus of continuity $\omega$ such that
$$|S_1|+|S_2|\leq C_{d,b} \big( \omega(\mu+\|\hat\zeta\|_{C^{2,\alpha}})  + \omega(\mu^0+\|\tilde\zeta\|_{C^{2,\alpha}}) + s^3\big).$$
Recall that $\tilde \zeta=\zeta_1+\zeta_2$, where $\zeta_2$ is a finite linear combination of eigenfunctions of the operator $T$ (see Lemma \ref{l:main}) with coefficients which are bounded by $\| \zeta\|_{H^{1/2}}$.  In particular, there is a constant $C_{d,b}$, depending only on the $C^{2,\alpha}$ norms of the eigenfunctions in the index of the cone and in the kernel of $\delta^2\mathcal F(0)$, such that
$$\|\zeta_1\|_{C^{2,\alpha}}+\|\zeta_2\|_{C^{2,\alpha}}\le\|\zeta\|_{C^{2,\alpha}}+\|\mu_0+\Upsilon(\mu_0,s_0)+\zeta_2\|_{C^{2,\alpha}}+\|\zeta_2\|_{C^{2,\alpha}}\le  C_{d,b} \|\zeta\|_{C^{2,\alpha}}.$$
Thus, we get 
$$|S_1|+|S_2|\leq C_{d,b} \,\big(\omega(\|\zeta\|_{C^{2,\alpha}})+s^3\big).$$
 so that, up to choosing the $\|\zeta\|_{C^{2,\alpha}}$ and $s$ small enough, we have
 \begin{equation}\label{e:improvement_1}
 \int_0^1E^\perp r^{d-1}\,dr\leq - C_{d, b} \,\eps\,\|\tilde\zeta\|_{H^{\sfrac12}}^2.
 \end{equation}

\noindent{\bf Estimate of $E^T$.}
We can assume without loss of generality that $G(\mu^0, s^0)>0$, otherwise set $\mu\equiv\mu^0$ and $s\equiv s^0$, and the conclusion holds for all sufficiently small $\eps$ depending only on $b$ and $\delta > 0$. By the \L ojasiewicz inequality for the analytic function $G$ (see \cite{Loj}), there exist a neighborhood $U$ of $(0,0)\in\R^N\times\R$ and constants $C, \gamma>0$, depending on $b$ and the dimension $d$, with $\gamma\in (0,\sfrac12]$, such that
\begin{equation}\label{e:loj}
|G(\mu,s)|^{1-\gamma}\leq C\, |\nabla G(\mu,s)|\,, \qquad \mbox{ for every }(\mu,s)\in U\,.
\end{equation}
Therefore, as long as $0<G(\mu(t),s(t))$, we can use the flow \eqref{e:grad_flow} to estimate,
\begin{equation}\label{e:mon}
\begin{aligned}
G(\mu(t), s(t))-G(\mu^0, s^0)=&\int_0^t\frac{d}{d\tau}G(\mu(\tau), s(\tau))\,d\tau=\int_0^t\nabla G(\mu(\tau),s(\tau))\cdot(\mu'(\tau), s'(\tau))\,d\tau\\
=&-\int_0^t|\nabla G(\mu(\tau),s(\tau))|\,d\tau\leq0,
\end{aligned}
\end{equation}
so that the function $t\mapsto  G(\mu(t), s(t))$ is non-increasing and $\geq 0$, and therefore there exists a time $t_1>0$ such that
$$
\begin{cases}
G(\mu(t), s(t))\geq \frac12 G(\mu^0, s^0)>0 & \mbox{if } 0\leq t\leq t_1\\
G(\mu(t), s(t))\leq \frac12 G(\mu^0, s^0)  & \mbox{if } t\geq t_1\,.
\end{cases}
$$
If $\eta_3(r) \leq t_1$ then using \eqref{e:mon} we get
\begin{align}
G\big(\mu(\eta_3(r)), s(\eta_3(r))\big)-(1-\eps)G(\mu^0, s^0)
	&= -\int_{0}^{\eta_3(r)}|\nabla G(\mu(\tau), s(\tau))|\,d\tau +\eps\, G(\mu^0, s^0)\notag \\
	&\leq-C\int_{0}^{\eta_3(r)}|G(\mu(\tau), s(\tau))|^{1-\gamma}\,d\tau+\eps\, G(\mu^0, s^0)\notag\\
	&\leq -C\,|G(\mu(\eta_3(r)), s(\eta_3(r)))|^{1-\gamma} \,\eta_3(r)+\eps\, G(\mu^0, s^0)\notag \\
	&\leq -C\,|G(\mu^0, s^0)|^{1-\gamma} \,\eta_3(r)+\eps\, G(\mu^0, s^0)\notag\\
	&= -\left(C\eta_3(r) -\eps G(\mu^0, s^0)^{\gamma}\right)\, G(\mu^0, s^0)^{1-\gamma},
\end{align}
where in the first inequality we used the \L ojasiewicz inequality \eqref{e:loj}, the second follows by the monotonicity of $g$ and the third by the assumption $\eta_3(r) \leq t_1$.

\noindent If $\eta_3(r) > t_1$, then choosing $\ \ds\eta_3(r) := \frac{d+2}{C}\eps G(\mu^0, s^0)^\gamma (1-r)\ $, we have 
$$\begin{aligned} G\big(\mu(\eta_3(r)), s(\eta_3(r))\big)-(1-\eps)G(\mu^0, s^0) &\le -\left(\frac{1}{2}-\eps \right)G(\mu^0, s^0)\\ &\le -\big(C\,\eta_3(r)-\eps G(\mu^0, s^0)^\gamma \big)\, G(\mu^0, s^0)^{1-\gamma}.\end{aligned}$$
Therefore, for every $r\in[0,1]$, we have the estimate \begin{equation}\label{e:estimatingflow}\begin{aligned}
\int_0^1 E^T r^{d-1} dr =& \int_0^1 \Big(G\big(\mu(\eta_3(r)),s(\eta_3(r))\big)-(1-\eps)G(\mu^0,s^0)\Big) r^{d-1} dr\\
 \leq& -G(\mu^0, s^0)^{1-\gamma}\int_0^1  \left(C\eta_3(r) -\eps G(\mu^0, s^0)^\gamma \right) r^{d-1}\ dr\\
=&-\eps G(\mu^0, s^0) \int_0^1 \big((d+1)-(d+2)r\big)r^{d-1}dr = -\frac{\eps G(\mu^0, s^0)}{d(d+1)}.
\end{aligned}
\end{equation}

\noindent{\bf Estimating the two error terms in the right-hand side of \eqref{e:useslicing}.}
For the radial term, we notice that since $s$ is $1$-Lipschitz,  for $\kappa_0/2 \le s_0\le 2\kappa_0$ and $\eps$ and $G(\mu^0,s^0)$ small enough, we have
$$\frac{\kappa_0}4\le s^0-\eta_3(r)\le s(\eta_3(r))\le s^0+\eta_3(r)\le 4\kappa_0.$$
Using this estimate in the $\kappa$ direction we get
\begin{equation}\label{e:errorink}
\int_0^1 r^{d+1} \frac{9s(\eta_3(r))^4}{\kappa_r^2}[s'(\eta_3(r))]^2 [\eta'_3(r)]^2\,dr  \leq C_{d,b}\,\eps^2 \,G(\mu^0, s^0)^{2\gamma}.
\end{equation} 
We now estimate the second term in the right-hand side of \eqref{e:useslicing}. By \eqref{e:competitor} we have 
\begin{equation*}
g'(r)=\partial_r g(r,\theta)=\eta_1'(r) \zeta_1(\theta)+\eta'_3(r)\Bigl(\mu'(\eta_3(r))+\big(\mu'(\eta_3(r)), s'(\eta_3(r))\big)\cdot\nabla\Upsilon\big(\mu(\eta_3(r)), s(\eta_3(r))\big)\Bigr)\,.
\end{equation*}
Since $\|\zeta_1\|_{L^\infty(\partial \Omega)}\le C_{d,b}$, $|\mu'|\le 1$ and $|\nabla\Upsilon|\le C_{d,b}$ in a neighborhood of $(0,0)$, we get 
\begin{equation*}
\|g'(r)\|_{L^2(\partial\Omega)}\le C_{d,b}\big(\|\zeta_1\|_{L^2}|\eta_1'(r)|+|\eta_3'(r)|\big)\,,
\end{equation*}
so, using Lemma \ref{l:main} (iii) and the definition of $\eta_3$, we conclude that
\begin{equation}\label{e:erroring}
\int_0^1 \kappa_r^2 r^{d+1} \int_{\partial B_1} \big|\delta\phi_{g_r}[g'(r)]\big|^2 \,d\sigma \,dr \leq C_{d,b} \Big(\eps^2 \|\zeta_1\|_{L^2}^2 + \eps^2 G(\mu^0, s^0)^{2\gamma}\Big).
\end{equation}

\noindent{\bf Conclusion of the proof.} In order to conclude the proof, we consider two cases. 

\noindent {\bf Case 1:} $\|\tilde{\zeta}\|_{H^{\sfrac 12}} > C^{-1} G(\mu^0, s^0)^{1/2}$ for some universal constant $C > 0$ depending only on $b$ and $d$.   In this scenario, let $\eta_3 \equiv 0$ and combine \eqref{e:useslicing}, \eqref{e:exp_h}, \eqref{e:improvement_1}, \eqref{e:errorink} and \eqref{e:erroring} to get 
$$
W(\bar{h}) - W(b) - (1-\eps)(W(z) - W(b)) \leq -C_{d,b} \eps\|\tilde{\zeta}\|_{H^{\sfrac 12}}^2 + \eps G(\mu^0, s^0) \leq \left(-C_{d,b} \eps+\frac1C\right) \|\tilde{\zeta}\|_{H^{\sfrac12}}^2 < 0\,,
$$ 
where we used the fact that $\eta'_3\equiv 0$ so that no error term come from the flow.  Therefore, the epiperimetric inequality holds for $\eps > 0$ small but universal.

\smallskip

\noindent {\bf Case 2:} Otherwise, we need to chose $\eps > 0$ depending on $G(\mu^0, s^0)$ so that we can absorb the errors in \eqref{e:erroring}, \eqref{e:errorink}, into the gain \eqref{e:estimatingflow}. Letting $\eps = \eps_1 G(\mu^0, s^0)^{1-2\gamma}$ for some $\eps_1 > 0$ small (but universal) and combining  
\eqref{e:improvement_1}, \eqref{e:estimatingflow}, \eqref{e:errorink} and \eqref{e:erroring}, we get
\begin{align}\label{e:final}
W \big(\bar h(r,\theta)\big)  - (1-\eps)W(z)+\eps W(b)&\leq - C_{d, b} \,\eps\,\|\tilde\zeta\|^2_{H^{\sfrac12}(\de \Omega)}-\frac{\eps G(\mu^0, s^0)}{d(d+1)}\notag\\
 &\qquad +C_{d,b} \Big(\eps^2 \|\tilde\zeta\|_{H^{\sfrac 12}(\de \Omega)}^2 + \eps^2 G(\mu^0, s^0)^{2\gamma}\Big)\notag\\
 &\leq -C_{d,b}\,\eps_1\,G(\mu^0, s^0)^{2-2\gamma}\,.
\end{align} 
Finally, writing $W(z) - W(b)$ as in \eqref{e:exp_h} and using the estimate \eqref{e:S_2}, we obtain 
$$
\begin{aligned} 
W(z) - W(b) =\int_0^1 \mathcal G(\zeta, s^0)\,r^{d-1}\,dr=& \frac1d\left(\mathcal G(\zeta, s^0) - \mathcal G\big(\mu^0 + \Upsilon(\mu^0, s^0), s^0\big) + G(\mu^0, s^0)\right)\\
\leq& C_{d,b}\left(\|\tilde{\zeta}\|_{H^{\sfrac12}}^2 + G(\mu^0, s^0)\right) \leq C_{d,b}G(\mu^0, s^0)\,,
\end{aligned}
$$
so that
$$
-\eps\leq -C_{d,b}\,\eps_1\,\left( W(z) - W(b)\right) ^{1-2\gamma}
$$
and the conclusion follows by replacing this in \eqref{e:final} with $\gamma'=1-2\gamma\in [0,1)$.

\subsection{Proof of Proposition \ref{prop:first_mode}: the integrable case.}\label{sub:first_mode}
Before starting with the proof of Proposition \ref{prop:first_mode} in the case when $b$ is integrable, we need a preliminary Lemma, which will allow us to kill rotations, that is to choose a parametrization for which the kernel of $\delta^2\mathcal F(0)$ is trivial.  Throughout this subsection we assume that $b$ is integrable through rotations (Definition \ref{d:integrability}). 

\begin{lemma}[Killing the linear part]\label{l:kill_linear}
There exists a dimensional constant $\delta>0$ such that if $\zeta\colon \de \Omega \rightarrow \mathbb R$ satisfies $\|\zeta\|_{C^{2,\alpha}(\de \Omega)} < \delta$, then there exists a rotation $U$ of the coordinate axes such that $\graph_{\de \Omega} (\zeta)=\graph_{U(\de \Omega)}(\tilde\zeta)$, where $\tilde \zeta$ satisfies
\begin{enumerate}[(i)]
\item $\|\tilde \zeta\|_{C^{2,\alpha}(U(\de \Omega))} \le C\,\| \zeta\|_{C^{2,\alpha}(\de \Omega)}$, for a dimensional constant $C$;
\item if $\xi_i$ is such that $\delta^2 \mathcal F(0)[\xi_i,\cdot]=0$, then $\ds\int_{U( \de \Omega)}  \tilde \zeta\,\xi_i\,d\mathcal H^{d-2} = 0$.  
\end{enumerate}
\end{lemma}

\begin{proof} Consider the family of functions defined by
	$$
	\mathcal{R}:=\left\{ u\colon \de \Omega \to \R\,:\, u:=a_1\,\xi_{1}+\dots+a_k\,\xi_k\,,\,\,a_1,\dots,a_k\in \R\right\}
	$$
	where $k = \dim \mathrm{ker}\; \delta^2 \mathcal{F}(0)$ and $\xi_i$, $i=1,\dots,k$, are the eigenfunctions of $\delta^2 \mathcal{F}(0)$ relative to the eigenvalue $0$ with $\|\xi_i\|_{L^2(\partial \Omega)}=1$. By the definition of integrability it follows that each $\xi_i$ is an infinitesimal generator of a rotation of $\de \Omega_b$. For any $u \in \mathcal{R}$, let $U_u$ be the rotation generated by $u$ of magnitude $\sum_{i=1}^{k} |a_i|$. 
	
	We consider the functional $\Psi\colon C^{2,\alpha}(\de \Omega)\times \mathcal R\to \R^{k}$, defined in a neighborhood of $(0,0)$, by
	$$
	\Psi (f, u):=\left( \int_{U_u( \de \Omega)}f_u\,\xi_1^u\,d\mathcal H^{d-2},\dots, \int_{U_u( \de \Omega)}f_u\,\xi_{k}^u\,d\mathcal H^{d-2}  \right)
	$$
	where, as above, $f_u$ satisfies $\graph_{\de \Omega} (f)=\graph_{U_u(\de \Omega)}(f_u)$. Notice that, since $\graph_{\de \Omega} (0)=\graph_{U_{t\xi_i}(\de \Omega)}(-t\xi_i)$, we have with $u = t\xi_i$ that
	$$
	\Psi(0, t\xi_i)=-\left(t \int_{U_u(\partial \Omega)} \xi_i^u \,\xi_1^u  ,\dots,t \int_{U_u(\partial \Omega)} \xi_{i}^u \,\xi_{k}^u \right)
	$$ 
	so that, identifying $\mathcal R$ with $\R^{k}$ using the basis $(\xi_i)_i$, we compute
	$\Psi(0,0)=0$ and $\nabla_{\mathcal R} \Psi(0,0)=-\mbox{Id}$.
	By the implicit function theorem, the conclusion immediately follows.
\end{proof}

\begin{proof}[Proof of Proposition \ref{prop:first_mode} for integrable cones] 
We first notice that by Lemma \ref{l:kill_linear} we can assume that $\mu^0 = 0$. 
Thus, in the notation of Subsection \ref{sub_1_mode}, we are left with $\zeta = \zeta^\perp = \tilde{\zeta} = \zeta_1 + \zeta_2$. The positivity set of our competitor will be determined by $g(r,\theta) = \eta_1(r)\zeta_1(\theta) + \zeta_2(\theta)$, while the competitor itself is given by $\overline{h}(r,\theta) = r\kappa_r\phi_{g_r}(\theta)$, where 
$$\kappa_r^2  = \kappa_0^2+\eta_3(r)(s^0)^3\ ,\qquad (s^0)^3=\kappa^2 - \kappa_0^2\ ,\qquad \eta_3(r) = 1 - \varepsilon a(1-r),$$
and $a \in \mathbb R$ is given by $a = \text{sign}(s_0)\,\delta \lambda(0)[\zeta] = \text{sign}(s_0)\,\delta\lambda(0)[\zeta_1]$, since by the integrability assumption $\delta\lambda(0)[\zeta_2]=0$.  Note that $\kappa_1^2 = \kappa^2$, which means that $\overline{h}(r,\theta)$ satisfies the boundary condition $\overline{h}(1,\theta)=\kappa\phi_\zeta(\theta)=c(\theta)$. 
As above, the slicing lemma gives 
 \begin{equation}\label{e:useslicing2}
\begin{aligned}
\Big|\Big(W\big(\bar h(r,\theta)\big) - (1&-\eps)W(z)+\eps W(b)\Big) - \int_0^1 \Big[\mathcal G\big(g_r, s^0 \eta_3^{1/3}(r)\big) - (1-\eps)\mathcal G(\zeta, s^0)\Big]r^{d-1}\, dr\Big|\\
&\le3(s^0)^3\eps^2 a^2  \int_0^1 r^{d+1} \eta_3(r)^2\ dr + \int_0^1 \kappa_r^2 r^{d+1} \int_{\partial B_1} \big|\delta\phi_{g_r}[g'_r]\big|^2 \,d\theta \,dr.
\end{aligned}
\end{equation}
\noindent For simplicity let $g = g_r$ and $s = s^0\eta_3(r)^{1/3}$. 
As in \eqref{e:exp_h}, we write: 
$$
\mathcal G(g, s)-(1-\eps)\mathcal G(\zeta, s^0) = \underbrace{\mathcal G(g, 0) - (1-\eps)\mathcal G(\zeta, 0)}_{=:\widetilde{E}^\perp} + \underbrace{(\mathcal G(g, s) - \mathcal G(g,0)) - (1-\eps)(\mathcal G(\zeta, s^0) - \mathcal G(\zeta, 0))}_{=:\widetilde{E}^T}.
$$ 
\noindent To estimate $\widetilde{E}^\perp$ we proceed as for the term $E^\perp$ above, so that by \eqref{e:improvement_1} with $\Upsilon\equiv 0$ we get
\begin{align*}
\widetilde{E}^\perp&=\mathcal G(g,0)-(1-\eps)\mathcal G(\zeta,0)=\mathcal F(g)-(1-\eps)\mathcal F(\zeta)\\
&\le\frac12\left(\delta^2\mathcal F(0)[g,g]-(1-\eps)\delta^2\mathcal F(0)[\zeta,\zeta]\right)+\omega(\|g\|_{C^{2,\alpha}})\|g\|_{H^{\sfrac12}}^2+\omega(\|\zeta\|_{C^{2,\alpha}})\|\zeta\|_{H^{\sfrac12}}^2\\
&\le\frac12\left(\eta_1^2(r)-(1-\eps)\right)\delta^2\mathcal F(0)[\zeta_1,\zeta_1]+\frac{\eps}2\delta^2\mathcal F(0)[\zeta_2,\zeta_2]+\omega(\|\zeta_1\|_{C^{2,\alpha}}+\|\zeta_2\|_{C^{2,\alpha}})\|\zeta\|_{H^{\sfrac12}}^2.
\end{align*}
Integrating in $r$, and choosing again $\eta_1(r)=1-(d+1)\eps(1-r)$, we get 
\begin{equation}\label{e:int1}
 \int_0^1\widetilde{E}^\perp r^{d-1}\,dr\leq - C_{d, b} \,\eps\,\|\zeta\|_{H^{\sfrac12}}^2.
\end{equation}
Moreover, again as in \eqref{e:erroring} without the $G$ term, we have
\begin{equation}\label{e:2int}
\!\!\!\!\!\int_0^1 \kappa_r^2 r^{d+1} \int_{\partial B_1} \big|\delta\phi_{g_r}[g'_r]\big|^2 \,d\theta \,dr=\int_0^1 \kappa_r^2 r^{d+1} (\eta_1'(r))^2\int_{\partial B_1} \big|\delta\phi_{g_r}[\zeta_1]\big|^2 \,d\theta \,dr \leq C_{d,b} \eps^2 \|\zeta_1\|_{L^2}^2 .
\end{equation}

\noindent To estimate $\widetilde{E}^T$ we write \begin{equation}\label{e:varykappa} \begin{aligned} \widetilde{E}^T&= (s^0)^3 \Big(\eta_3(r)\big(\lambda(g) - (d-1)\big) - (1-\eps)\big(\lambda(\zeta) - (d-1)\big)\Big)\\
&\le (s^0)^3\Big( \eta_3(r)\delta \lambda(0)[g] - (1-\eps)(\delta\lambda(0)[\zeta] \Big) +|s^0|^3\sup_{t\in [0,1]} \Big(\big|\delta^2 \lambda(tg)[g,g]\big|+\big|\delta^2 \lambda(t\zeta)[\zeta,\zeta]\big| \Big) \\
&\leq (s^0)^3 \big(\eta_1(r)\eta_3(r) - (1-\eps)\big)\delta \lambda(0)[\zeta_1] +C|s^0|^3\|\zeta\|_{H^{\sfrac12}}^2,
\end{aligned}
\end{equation}
where in the last line we used the continuity of $\delta^2 \lambda$ and that the eigenfunctions of $\delta^2\mathcal F(0)$ with negative eigenvalues do not change $\lambda$ to first order.
Now, by the definition of $\eta_1$ and $\eta_3$ we have
$$
\begin{aligned}
\int_0^1 \big(\eta_1(r) \eta_3(r) - (1-\eps)\big)r^{d-1}\ dr 
&= -\frac{\eps a}{d(d+1)} + \frac{2\eps^2 a}{d(d+2)} = -\frac{a \eps}d\left(\frac1{d+1}-\frac{2\eps}{d+2}\right).
\end{aligned}
$$ 
Since $a = \text{sign}(s_0)\,\delta\lambda(0)[\zeta_1]$, choosing $\eps > 0$ small enough depending only on $d$, we get
\begin{equation}\label{e:widetildeestimate}
\int_0^1 \widetilde{E}^T r^{d-1} \,dr  \leq - C_d|s^0|^3(\delta\lambda(0)[\zeta_1])^2 \eps + C|s^0|^3 \|\zeta\|_{H^{\sfrac12}}^2.
\end{equation} 
Putting \eqref{e:useslicing2}, \eqref{e:int1}, \eqref{e:2int}, \eqref{e:widetildeestimate} and \eqref{e:improvement_1} together, we get that 
$$
\begin{aligned}
\Big(W\big(\bar h(r,\theta)\big) - (1-\eps)W(z)+\eps W(b)\Big) 
	\leq &- \eps C_{d,b}\left(\|\zeta\|_{H^{\sfrac12}}^2  +|s^0|^3(\delta\lambda(0)[\zeta_1])^2\right) \\ 
	&+\eps^2C_{d,b}\left(\|\zeta_1\|_{H^{\sfrac12}}^2 +|s^0|^3(\delta\lambda(0)[\zeta_1])^2\right)+ C|s^0|^3 \|\zeta\|_{H^{\sfrac12}}^2. 
\end{aligned}
$$ 
Letting $\eps > 0$ be small enough and then perhaps shrinking $\delta \geq \|u - b\|^2_{L^2(\partial B_1)} \geq |s^0|^3$ so that $\delta < <\eps$, we get the desired result. 
\end{proof}

\section{Proof of Theorem \ref{t:main}}\label{ss:conclusion}
Theorem \ref{t:main} will follow by applying Theorem \ref{t:epi_smooth} on dyadic annuli thanks to a suitable parametrization lemma. The Smooth Parametrization Lemma \ref{l:ann_par} uses the  strong convergence of minimizers and the fundamental regularity result of Alt and Caffarelli \cite{AlCa}, to show that if the trace of a minimizer is sufficiently close to that of a cone with isolated singularity, then we can parametrize the free boundary of the minimizer on an annulus over that of the cone. Then in Lemma \ref{l:comp_scales} we show that the condition from Lemma \ref{l:ann_par} remains uniform in the annuli away from the origin. Theorem \ref{t:epi_smooth} can then be applied in the annulus to show that the closeness decades and so the procedure can be iterated.

\subsection{Smooth parametrization lemma}
We start by proving our main parametrization lemma.
\begin{lemma}[Smooth parametrization lemma]\label{l:ann_par}
	Let $b$ be a $1$-homogeneous minimizer of $\mathcal E$ with isolated singularity in zero and let $\tau\in (0,1)$. For every $\eps>0$, there exists $\delta_1=\delta_1(\eps,\tau,b)>0$ such that if $u\in H^{1}(B_1)$ is a minimizer of $\mathcal{E}$ satisfying $0\in\partial\{u>0\}$
	\begin{equation}\label{e:par_hyp_1}
	\Theta_u(0):=W(u,0) =W(b)\,,\qquad 
	\|u-b\|_{L^2(\de B_1)}<\delta_1
	\qquad\mbox{and}\qquad
	W(u,1)-W(b)\leq \delta_1\,,
	\end{equation}
	then there exists a function $\zeta(\theta, r)\in C^{2,\alpha}(\de \{b>0\})$ (indeed analytic) such that 
	\begin{equation}\label{e:par_ann}
	\de\{u>0\}\cap \partial B_r=\graph_{\partial \Omega_b}(\zeta(-, r))
	\qquad\mbox{and} \qquad
	\|\zeta(-, r)\|_{C^{2,\alpha}}\leq \eps,\:\: \forall r\in (\tau, 1-\tau)\,.
	\end{equation}
\end{lemma}

\begin{proof} Suppose the claim is not true, then there are sequences of minimizers $u_j\in H^{1}(B_1)$ and of numbers $\delta_j\to 0$, such that $0\in\partial\{u_j>0\}$,
\begin{equation}\label{e:paramcompactness}
\Theta_{u_j}(0)=W(b)\,,\qquad
\|u_j-b\|_{L^2(\de B_1)}<\delta_j
\qquad\mbox{and}\qquad
W(u_j,1)-W(b)\leq \delta_j\,,
\end{equation}
but such that $\de\{u_j>0\}\cap (B_{1-\tau} \setminus B_{\tau})$ does not satisfy \eqref{e:par_ann}. The condition \eqref{e:paramcompactness} implies that
\begin{align*}
\int_{B_{1}}|\nabla u_j|^2\,dx
	&\leq \delta_j+W(b)+\int_{\de B_{1}}u_j^2+2\,\omega_d
	\leq 3\delta_j+W(b)+\,2\int_{\de B_{1}}b^2+2\,\omega_d\leq C(d,b)\,.
\end{align*}
	Therefore the sequence $(u_j)_j$ is uniformly bounded in $H^1(B_1)$ and so up to subsequences it converges weakly in $H^1(B_{1})\cap L^2(\de B_1)$ to a function $v\in H^1(B_{1})$. Moreover, the minimality of $u_j$ implies that the convergence is $H^1(B_1)$-strong and $|\{u_j>0\}\cap B_{1}|\to |\{v>0\}\cap B_{1}|$ (see for instance \cite[Theorem 9.1]{GuyTo} or \cite{AlCa}). 
Then we have
\begin{gather}
\|v-b\|_{L^2(\de B_1)}= \lim_{j\to \infty}\|u_j-b\|_{L^2(\de B_1)}=0\ ,\qquad W(v,1)-W(b)= \lim_{j\to \infty}(W(u_j,1)-W(b))=0\,,\notag
\end{gather} 
and $\Theta_v(0)=W(b)$, so that $v\equiv b$ on $B_{1}$. Furthermore, by the uniform Lipschitz norm and non-degeneracy of $u_j$, it follows that $\de \{u_j>0\}$ converges to $\de \{b>0\}$ in the Hausdorff sense in $B_{1-\frac\tau2}$
(see for instance \cite{EnGuyTo}), and so by Alt-Caffarelli improvement of flatness (see \cite[Theorem 8.1]{AlCa}), we can conclude that for every $j$ sufficiently large there exists a $C^{1,\alpha}$ function $\zeta_j$ such that $\de\{u_j>0\}\cap (B_{1-\tau} \setminus B_{\tau}) = \graph(\zeta_j)$. Applying \cite{Jerison} (see also \cite[Theorem 2]{KiNi}), the regularity of $\zeta_j$ can be improved to $C^{2,\alpha}$, with small $C^{2,\alpha}$ norm of the graph. Finally using the smallness of the $C^{1,\alpha}$ norm, a simple reparametrization implies that, for every $j$ big enough,
$\de\{u_j>0\}\cap (B_{1-\tau} \setminus B_{\tau}) = \graph_{\partial \{b >0\}}(\zeta_j)
\ \mbox{and}\ 
\|\zeta_j\|_{C^{2,\alpha}}\to 0\,.$
\end{proof}

Before proving the main theorem we also need a preliminary lemma about blow-up sequences at comparable scales.

\begin{lemma}[Blow-ups at comparable scales]\label{l:comp_scales}
	Let $u\in H^1(B_1, \R_+)$ be a minimizer of $\mathcal E$ such that $0\in\partial\{u>0\}$. Then for every $\delta_2>0$, there exists $r_0=r_0(\delta_2)>0$ such that for every $0<r<r_0$ the following inequality holds
	\begin{equation}\label{e:clos_ann}
	\|u_{\rho}-u_{r}\|_{L^2(\de B_1)}\leq \delta_2 \qquad \mbox{for every }\rho\in\left[\frac r8,r\right]\,.
	\end{equation}
\end{lemma}

\begin{proof} Suppose, by contradiction, that there are sequences $r_n\downarrow0$ and 
$\ds \rho_n\in \left[\frac{r_n}8,r_n \right]$ such that 
	$$
	\int_{\de B_1}|u_{\rho_n}-u_{r_n}|^2> \delta_2^2 \qquad \mbox{for every }\,n\in\N\,.
	$$
	In particular, notice that $\ds1\leq \frac{r_n}{\rho_n}\leq 8$ for every $n\leq \N$, so that 
	$
	\ds 1\leq \liminf_{n\to \infty}\frac{r_n}{\rho_n}\leq \limsup_{n\to \infty}\frac{r_n}{\rho_n}\leq 8\,.
	$
	Now let $b\in H^1(\R^d)$ be a $1$-homogeneous minimizer of $\mathcal{E}$ such that $u_{\rho_n}(x):=\frac{u(\rho_n x)}{\rho_n}\to b$ locally uniformly in $\R^d$. Now since $\ds \lim_{n\to \infty} \int_{\de B_{\sfrac{r_n}{\rho_n}}}|u_{\rho_n}-b|^2=0$, 
	we can compute
	\begin{align}\label{e:cont_dyad_2}
	\lim_{n\to \infty}\int_{\de B_1}|u_{r_n}-b|^2
	&=\lim_{n\to \infty} \left(\frac{\rho_n}{r_n}\right)^{d-1}\,\int_{\de B_{\sfrac{r_n}{\rho_n}}}|u_{\rho_n}-b|^2=0\,,
	\end{align}
	where we used the $1$-homogeneity of $b$. 
	This means that $b$ is the blow up associated to the sequence $r_n$, and so by the triangular inequality
	$$
	0<\delta_2^2\leq\lim_{n\to \infty}\int_{\de B_1}|u_{\rho_n}-u_{r_n}|^2 \leq \lim_{n\to \infty}\left(2\int_{\de B_1}|u_{\rho_n}-b|^2+2\int_{\de B_1}|b-u_{r_n}|^2\right)=0\,,
	$$
	which gives the desired contradiction.
\end{proof}

\subsection{Proof of Theorem \ref{t:main}} 
We first recall the Weiss' monotonicity formula
\begin{equation}\label{e:weiss}
\frac{d}{d\rho}W(u_\rho)=\frac{d}{\rho}\big[W(z_\rho)-W(u_\rho)\big]+\frac{1}{\rho}\int_{\partial B_1}|x\cdot \nabla u_\rho-u_\rho|^2\,d\HH^{d-1},
\end{equation}
where $z_\rho$ is the $1$-homogeneous extension of the trace $c_\rho:=u_{\rho}|_{\de B_1}$. In particular, $r\mapsto W(u_r)$ is increasing, $\lim_{r\to0}W(u_r)=W(b)$ and for every $0<s<t\le 1$, we have the estimate 
 \begin{align}\label{e:dec_3}
	 \!\!\!\!\!\int_{\de B_1}\left| u_t-u_s \right|^2 \,d\,\HH^{d-1}
	 &\leq \int_{\de B_1}\left(\int_s^t\frac{1}{r}\left|  x\cdot \nabla u_r-u_r   \right| \,dr\right)^2\,d\HH^{d-1} \notag\\
	 &\leq  \int_{\de B_1} \left( \int_s^t r^{-1}\,dr\right)\left( \int_s^t r^{-1} \left|  x\cdot \nabla u_r-u_r   \right|^2   \,dr\right)\,d\HH^{d-1} \notag\\
	 &\leq (\log(t)-\log(s)) \int_s^t r^{-1} \int_{\de B_1} \left|  x\cdot \nabla u_r-u_r   \right|^2 \,d\HH^{d-1}  \,dr \notag\\
	 &\leq\log(t/s) \int_s^t \frac{d}{dr} \left[W(u_r)-W(b)\right] \,dr\le \log(t/s)\left(W(u_t)-W(b)\right).\end{align}
Next suppose that the logarithmic epiperimetric inequality 
\begin{equation}\label{e:epiinrange}
W(u_\rho)-W(b)\leq \big(1-\eps \left|W(z_\rho)-W(b)\right|^{\gamma} \big) (W(z_\rho)-W(b)) \quad \mbox{holds for  $\rho\in(s,r_0)$,}
\end{equation}
 then we can estimate
\begin{align}
\nonumber
\frac{d}{d\rho}\big(W(u_\rho)-W(b)\big)&
=\frac{d}{\rho}\left[W(z_\rho)-W(b)-W(u_\rho)+W(b)\right]+\frac{1}{\rho}\int_{\partial B_1}|x\cdot \nabla u_\rho-u_\rho|^2\,d\HH^{d-1}\\ 
&\geq \frac{d\eps}{\rho}\frac{\left(W(u_\rho)-W(b)\right)^{1+\gamma}}{1-\eps\left(W(u_\rho)-W(b)\right)^{\gamma}}\,+\frac{1}{\rho}\int_{\partial B_1}|x\cdot \nabla u_\rho-u_\rho|^2\,d\HH^{d-1} \notag\\
&\geq \frac{d\eps}{\rho}\left(W(u_\rho)-W(b)\right)^{1+\gamma}+\frac{1}{\rho}\int_{\partial B_1}|x\cdot \nabla u_\rho-u_\rho|^2\,d\HH^{d-1}\,,\,\label{e:dec_1}
\end{align}
which in particular gives that
\begin{equation}\label{e:dec_2}
\frac{d}{d\rho}\Big( -{\left(W(u_\rho)-W(b)\right)^{-\gamma}} - \gamma d\eps \log(\rho)\Big)\geq 0,\quad \text{for every }\rho\in \left(s,r_0\right).
\end{equation}
Thus integrating \eqref{e:dec_2}, we see that 
\begin{equation}\label{e:s_t_estimate}
 W(u_\rho)-W(b)\le \frac{1}{\left(-\eps\gamma d\log(\sfrac{\rho}r_0)\right)^{\frac1\gamma}} \,, \quad \text{for every }\rho\in \left(s,r_0\right).
\end{equation}
Next let $i \leq j$, be such that $s/r_0 \in [2^{-2^{j+1}}, 2^{-2^j})$ and $t /r_0\in [2^{-2^{i+1}}, 2^{-2^i})$. Then, using \eqref{e:dec_2} in \eqref{e:dec_3}, we calculate 
\begin{equation}\label{e:fanculo}
\begin{aligned}\|u_s - u_t\|_{L^2(\partial B_1)} \leq& \|u_s - u_{2^{-2^{j}}}\|_{L^2(\partial B_1)} + \|u_t - u_{2^{-2^{i+1}}}\|_{L^2(\partial B_1)} + \sum_{k = i+1}^{j-1} \|u_{2^{-2^{k+1}}}- u_{2^{-2^{k}}} \|_{L^2(\partial B_1)} \\
\stackrel{\eqref{e:dec_3}}{\leq}&  \sum_{k=i}^{j} 2^k e(2^{-2^{k}})
\stackrel{\eqref{e:s_t_estimate}}{\leq} C\sum_{k=i+1}^{j-1} 2^k 2^{-k/\gamma}
\stackrel{\gamma \in (0,1)}{\leq} C_{\gamma}2^{(i+1)\frac{\gamma-1}{\gamma}}\\
\leq &C_{\gamma} (-\log(t/r_0))^{\frac{\gamma-1}{\gamma}}.
\end{aligned}
\end{equation}

Next we claim that \eqref{e:epiinrange} holds for $s=0$, so that \eqref{e:fanculo} shows that $(u_r)_r$ is a Cauchy sequence in $L^2(\partial B_1)$ and since homogeneous functions depend only on their traces, it proves the uniqueness of blow-up at the point.

\noindent Indeed, thanks to the assumption that $b$ is a blow-up of $u$ at $0$, we can choose $0<\delta_2<\delta_1/2$ and $r=r(\delta_1,\delta_2, \gamma)>0$ in such a way that
\begin{equation}\label{e:choice_const}
(W(u_{2r})-W(b))^{\frac{1-\gamma}2}+2\delta_2 +\|u_{2r}-b\|_{L^2(\de B_1)}+C_\gamma\left(-\log(r/r_0)\right)^{\frac{\gamma-1}{\gamma}}\leq \delta_1\,.
\end{equation}
Next suppose $r > \rho>0$ is the first radius at which \eqref{e:epiinrange} fails. Thanks to \eqref{e:choice_const}, we can use Lemma \ref{l:comp_scales} to apply Lemma \ref{l:ann_par} with $\tau=1/2$, so that the hypothesis of Theorem \ref{t:epi_smooth} are satisfied and \eqref{e:epiinrange} holds in the range $(r/2, r)$. This implies that $\rho<r/2$. On the other hand we can use \eqref{e:dec_3} in $[3\rho/2, r]$ to get
\begin{align*}
\|u_{3\rho/2}-b\|_{L^2(\de B_1)}
	&\leq \|u_{3\rho/2}-u_{r}\|+\|u_{r}-u_{2r}\|_{L^2(\de B_1)}+\|b-u_{2r}\|_{L^2(\de B_1)}\\
	&\leq \delta_2+C_\gamma\left(-\log(r/r_0)\right)^{\frac{\gamma-1}{\gamma}}+\|b-u_{2r}\|_{L^2(\de B_1)}\leq \delta_1-\delta_2\,,
\end{align*}
so that once again we can use Lemma \ref{l:comp_scales} with $\tau = 1/2$ to get $\|u_{3\rho/4} - b\|_{L^2} \leq \delta_1$. We then apply Lemma \ref{l:ann_par}, so that the hypothesis of Theorem \ref{t:epi_smooth} are satisfied and \eqref{e:epiinrange} holds in the range $(3\rho/4, r)$, which is a contradiction with the definition of $\rho>0$. This implies $\rho=0$ and concludes the claim.

Finally, let $b$ be the unique blow-up. Lemma \ref{l:ann_par}, the decay of the $L^2$ norm and the Weiss' boundary adjusted energy $W(u_r)$ imply that $\de \{ u>0 \}\cap B_r$ is a graph over $\de \{ b>0 \}\cap B_r$. Moreover, again by Lemma \ref{l:ann_par} we have that $\de \{ u>0 \}\cap B_r$ is  a $C^1$ graph over $\de \{ b>0 \}\cap B_r$, that is the graph over $\de \{ b>0 \}\cap B_1$ associated to $\de \{u_r>0\}\cap B_1\setminus B_{1/8}$ converges to zero in $C^1$ norm as $r\to0$. This convergence can be improved to $C^{1,\log}$ by a standard argument that we sketch for the readers' convenience. Indeed, since $\de \{u_r>0\}\cap B_1\setminus B_{1/8}$ is a smooth graph with controlled $C^{2,\alpha}$ norm, the (log-) epiperimetric inequality holds at a uniform scale at every point  $x_0\in \de \{u_r>0\}\cap B_{1/2}\setminus B_{1/4}$. Thus, the oscillation of the normals $|\nu_{x_0,r}-\nu_{y_0,s}|$, where $x_0\in \de \{u_r>0\}\cap B_{1/2}\setminus B_{1/4}$ and $y_0\in \de \{u_s>0\}\cap B_{1/2}\setminus B_{1/4}$ is controlled by a power of $r_0$ (for some $0<r_0<1$) and the $L^2$-distance $\|(u_r)_{x_0,r_0}-(u_s)_{y_0,r_0}\|_{L^2(\partial B_1)}$. Now this last distance has a logarithmic decay due to the logarithmic decay of $\|u_r-b\|_{L^2(\partial B_1)}$ proved above, which implies the $C^{1,\log}$ convergence of the graphs $\de \{u_r>0\}\cap B_{1/2}\setminus B_{1/4}$. 
\qed

The proof of Theorem \ref{t:main2}, the integrable case, follows similarly, but is simpler and mostly standard so we will omit it. Let us just remark out that we must still take care to show that the ``closeness" assumptions of Theorem \ref{t:epi_smooth} are satisfied on all scales. However, this argument works in essentially the same way in the integrable and non-integrable setting.  

\section{The Index of the De Silva-Jerison cone in the sphere}\label{s:DeSiJecone}

In this section we prove that the De Silva-Jerison cone $\mathcal C_{\nu, \theta_0}$ satisfies the conditions of Definition \ref{d:integrability}: namely that the elements of the kernel of $\delta^2\mathcal F(0)$ are generated by rotations and that each perturbation in $\mathrm{index}(\delta^2\mathcal F(0))$ integrates to zero along $\partial \Omega_{b_{\nu, \theta_0}}$. This completes the proof of Corollary \ref{c:DeSiJe}. 

To do so, we will produce an eigenbasis of deformations and show that, except for the deformations infinitesimally generated by rotations, each associated eigenvalue is non-zero. Furthermore, the perturbation generated by a constant will be an element of the eigenbasis and we will check that it is associated to a positive eigenvalue. This implies that all the other elements of the eigenbasis (including those in the index) integrate to zero on $\partial \Omega_{b_{\nu, \theta_0}}$. We note that the positivity of the constant perturbation requires a computational check, which we do explicitly for $d = 7$. A more general argument verifies the inequality for $d \geq 21$ and one can check the numerics for $8\leq d \leq 20$ via the procedure outlined below for $d= 7$.

Throughout this section we will write the points of the sphere $\partial B_1=\mathbb S^{d-1}$ in the spherical coordinates $(\theta,\phi)$, where $\phi\in \mathbb S^{d-2}$ and $\theta\in\left[0,\pi\right]$. Thus, the trace of the De Silva-Jerison cone on the sphere is simply given by $\ds\Omega_{b_{\nu,\theta_0}}=\left\{(\theta,\phi)\in \mathbb{S}^{d-1}\ :\ \frac\pi2-\theta_0<\theta<\frac\pi2+\theta_0\right\}.$

\subsection{A basis of eigenfunctions}
Let $\{\psi_j\}$ be the eigenfunctions of the Laplace-Beltrami operator on $\mathbb S^{d-2}$ (i.e. the spherical harmonics in dimension $d-1$). Then $\zeta^{\pm}_j$ is an orthogonal basis of $L^2(\partial \Omega_{b_{\nu, \theta_0}})$ and $H^{1/2}(\partial \Omega_{b_{\nu, \theta_0}})$, where \begin{equation}\label{defofzetapm}\begin{aligned} \zeta_j^+(\pi/2 \pm \theta_0, \varphi) =&- \psi_j(\varphi)\qquad\text{and}\qquad
\zeta_j^-(\pi/2 \pm \theta_0, \varphi) =\mp\, \psi_j(\varphi).\end{aligned}
\end{equation}
For $j > 1$, we define 
\begin{equation}\label{e:solutionpmj}
	\ds -\Delta u_j^\pm=(d-1)u_j^\pm\quad\text{in}\quad \Omega_{b_{\nu, \theta_0}},\qquad
	u_j^\pm=-\zeta^{\pm}_j \quad\text{on}\quad \partial\Omega_{b_{\nu, \theta_0}},\qquad\int_{\Omega_{b_{\nu, \theta_0}}}bu_j^\pm=0.
	\end{equation}
The case $j =1$ is slightly more complicated, as the variation in the direction of $\zeta_j^+$ (which can be geometrically interpreted as increasing the opening of the cone), changes the measure and the first eigenvalue of the domain to first order. Therefore, we define 
	\begin{equation}\label{e:solution+1}
	\ds -\Delta u_1^+=(d-1)u_1^++\eta b\quad\text{in}\quad \Omega_{b_{\nu, \theta_0}},\qquad
	u_1^+=-\zeta^+_1\quad\text{on}\quad \partial\Omega_{b_{\nu, \theta_0}},\qquad\int_{\Omega_{b_{\nu, \theta_0}}}bu_1^+=0,
	\end{equation}
	\begin{equation}\label{e:solution-1}
	\ds -\Delta u_1^-=(d-1)u_1^-\quad\text{in}\quad \Omega_{b_{\nu, \theta_0}},\qquad
	u_1^-=-\zeta^-_1 \quad\text{on}\quad \partial\Omega_{b_{\nu, \theta_0}},\qquad\int_{\Omega_{b_{\nu, \theta_0}}}bu_1^-=0,
	\end{equation}
	where 
	$$
	\eta:=\frac1{\kappa^2_0}\frac{\HH^{d-2}(\partial\Omega_{b_{\nu, \theta_0}})}{\sqrt{\HH^{d-2}(\mathbb{S}^{d-2})}}.
	$$
Recall from Appendix \ref{s:app}, Subsection \ref{sub:variations_in_zero}, that solutions to \eqref{e:solutionpmj}, \eqref{e:solution+1}, \eqref{e:solution-1} exist and are unique. 
In particular, we can write $u_j^\pm$ explicitly for $j > 1$ (this formula also holds for $u_1^-$) by separating the variables as $u_j^\pm(\theta, \varphi) = \psi_j(\varphi) f^{\pm}_j(\theta)$, where $f^{\pm}_j(\theta)$ are defined by 
\begin{equation}\label{e:solution+f}
	\begin{cases}
	\begin{array}{cc}
	\ds -\frac{1}{\sin^{d-2}\theta}\frac{\partial}{\partial\theta}\left(\sin^{d-2}\theta\,\frac{\partial f_j^+}{\partial\theta} \right) =\left(-\frac{\lambda_j^{d-2}}{\sin^2\theta}+(d-1)\right)f_j^+\quad\text{for}\quad \frac\pi2-\theta_0<\theta<\frac\pi2+\theta_0,\\
	\ds f_j^+\left(\frac\pi2-\theta_0\right)=f_j^+\left(\frac\pi2+\theta_0\right)=1,\\
	\end{array}
	\end{cases}
	\end{equation}
	\begin{equation}\label{e:solution-f}
	\begin{cases}
	\begin{array}{cc}
	\ds -\frac{1}{\sin^{d-2}\theta}\frac{\partial}{\partial\theta}\left(\sin^{d-2}\theta\,\frac{\partial f_j^-}{\partial\theta} \right) =\left(-\frac{\lambda_j^{d-2}}{\sin^2\theta}+(d-1)\right)f_j^-\quad\text{for}\quad \frac\pi2-\theta_0<\theta<\frac\pi2+\theta_0,\\
	\ds -f_j^-\left(\frac\pi2-\theta_0\right)=f_j^-\left(\frac\pi2+\theta_0\right)=1.\\
	\end{array}
	\end{cases}
	\end{equation}
	Above, and throughout, $\lambda_j^{d-2}$ refers to the $j$-th eigenvalue of the Laplace-Beltrami operator on $\mathbb S^{d-2}$, counted with multiplicity. 
Note that solutions to \eqref{e:solution+f} and \eqref{e:solution-f} are unique and in fact minimize (under the respective boundary conditions) the functionals 
	$$J_j(f)=\int_{\pi/2-\theta_0}^{\pi/2+\theta_0}\left(|f'(\theta)|^2+\frac{\lambda_j^{d-2}}{\sin^2\theta}f^2(\theta)-(d-1)f^2(\theta)\right)\sin^{d-2}\theta\,d\theta.$$
In order to show that the deformations $\zeta_j^{\pm}$ diagonalize $\delta^2\mathcal F(0)$ (and thus satisfy \eqref{e:diagonalization})
we recall that by \eqref{e:var''0} we have 
\begin{equation}\label{f''onDSJ}
\delta^2\mathcal F(0)[\xi_1, \xi_2] = -2(d-2)\tan(\theta_0)\int_{\de \Omega_{b_{\nu, \theta_0}}} \xi_1\xi_2 \, \,d\HH^{d-2}
			+2\,\int_{\de \Omega_{b_{\nu, \theta_0}}} \xi_1\,\,T\xi_2\,d\HH^{d-2}.
\end{equation}
Applying \eqref{f''onDSJ} to $\zeta_j^{\pm}$ and $\zeta_k^{\pm}$ and integrating by parts we get 
\begin{equation}\label{f''onDSJincoordinates}
\begin{aligned}
\frac12\delta^2 \mathcal F(0)[\zeta_j^{\pm}, \zeta_k^{\pm}] &= \int_{\partial \Omega_{b_{\nu, \theta_0}}} u_j^\pm \partial_\nu u_k^\pm \,\, d\HH^{d-2} -(d-2)\tan(\theta_0)\int_{\de \Omega_{b_{\nu, \theta_0}}} \zeta_j^{\pm}\zeta_k^{\pm} \, \,d\HH^{d-2} \\
&=\big \langle u_j^\pm, u_k^\pm\big\rangle -(d-2)\tan(\theta_0)\int_{\de \Omega_{b_{\nu, \theta_0}}} \zeta_j^{\pm}\zeta_k^{\pm} \, \,d\HH^{d-2},
\end{aligned}
\end{equation}
where for simplicity we have set $\ds\big \langle u_j^\pm, u_k^\pm\big\rangle:=\int_{\Omega_{b_{\nu, \theta_0}}} \left(\nabla u_j^\pm \cdot \nabla u_k^\pm - (d-1)u_j^\pm u_k^{\pm}\right)\,d\HH^{d-1}$.

We now claim that the family of functions $\{\zeta_j^{\pm}\}_{j\in\N}$ is orthogonal with respect to $\delta^2\mathcal F(0)$. Indeed, as $\zeta_j^{\pm}$ are orthogonal in $L^2(\partial \Omega_{b_{\nu, \theta_0}})$, by \eqref{f''onDSJincoordinates} it suffices to establish that the $u_j^{\pm}$ are orthogonal in $H^1(\Omega_{b_{\nu, \theta_0}})$. 
Indeed, that $\big \langle u_j^+, u_k^-\big\rangle = 0$ is trivial; the $u_j^+$ are even functions of $\theta$ across the equator, whereas the $u_k^-$ are odd functions of $\theta$ across the equator (this can be seen in \eqref{e:solutionpmj}, \eqref{e:solution+1},\eqref{e:solution-1}). 
Then $\big \langle u_j^+, u_k^+\big\rangle = 0$ and $\big \langle u_j^-, u_k^-\big\rangle = 0$ follow from the separation of variables; each is equal to a function in $\theta$ times a function in $\varphi$ and for $k \neq j$ the functions in $\varphi$ are orthogonal in $\mathcal H^1(\mathbb S^{d-2})$. Since the integrals split, we get the desired orthogonality. Moreover, a standard density argument gives that the family $\{\zeta_j^{\pm}\}_{j\in\N}$ is complete that is, it generates $H^{1/2}(\partial\Omega_{b_{\nu,\theta_0}})$. 

\subsection{Integrability through rotations: proof of Corollary \ref{c:DeSiJe}}

In order to prove the integrability trough rotations of the De Silva-Jerison type cones, and so Corollary \ref{c:DeSiJe}, it suffices to estimate $\delta^2\mathcal F(0)[\zeta^\pm_j, \zeta^{\pm}_j]$ and show that $\zeta_1^+$ is a positive direction, since it is the only one that changes $\lambda$ at first order.

We start by proving that there are $d$ perturbations which correspond to negative eigenvalues (i.e. that the De Silva-Jerison cone has index $d$ in the sphere).  

\begin{prop}[Index of De Silva-Jerison cone]
The eigenvalues associated to the eigenfunctions $\zeta_j^+$, for $2 \leq j \leq d$, and $\zeta_1^-$ are strictly negative.	
\end{prop}

\begin{proof}
To compute the energy of $u_j^+$, for $2 \leq j \leq d$, note that 
\begin{equation}\label{whoisujplus} 
u_j^+(\theta,\varphi) \equiv \psi_j(\varphi)\frac{\sin(\theta)}{\cos(\theta_0)}, \quad 2 \leq j \leq d 
\end{equation} 
satisfies \eqref{e:solutionpmj} (to see this, recall that $\lambda_j^{\mathbb S^{d-2}} = (d-2)$, for $2 \leq j \leq d$, and that $-\Delta \sin(\theta) = (d-1)\sin(\theta) - \frac{d-2}{\sin^2(\theta)} \sin(\theta)$). Plugging this into \eqref{f''onDSJincoordinates}, we get
\begin{equation}\label{energyofthebadrotations}
\begin{aligned} 
\delta^2\mathcal F(0)[\zeta_j^+, \zeta_j^+] =& 4\cos^{d-2}(\theta_0)\left(\int_{\mathbb S^{d-2}} u_j^+\partial_\nu u_j^+\ d\sigma - (d-2)\tan(\theta_0)\right)\\ =& 4\cos^{d-2}(\theta_0)\left(-\tan(\theta_0)-(d-2)\tan(\theta_0)\right) < 0, \ \text{ for }\  2 \leq j \leq d.
\end{aligned}
\end{equation}
To see that $\delta^2\mathcal F(0)[\zeta_1^-, \zeta_1^-] < 0$, note that $u_1^- = -\frac{1}{\sqrt{\mathcal H^{d-2}(\mathbb S^{d-2})}}\frac{\cos(\theta)}{\sin(\theta_0)}$. Then a simple computation gives the result.
\end{proof}

To show that $\zeta_1^-$ is a negative direction one could also use a more general principle, which will simplify the rest of the proof. Note that (except for $\zeta_1^+$) 
\begin{equation}\label{e:fromqtoj} 
\delta^2\mathcal F(0)[\zeta_j^\pm, \zeta_j^\pm] = 2J_j(f^\pm_j) - 4(d-2)\cos^{d-2}(\theta_0)\tan(\theta_0), 
\end{equation}

By the definition of $J_j$ (and the fact that $\lambda_j^{\mathbb S^{d-2}}$ increases as $j \rightarrow \infty$) for any function $w\in H^1$ we have that $J_{j+1}(w) \geq J_j(w)$ (with strict inequality if $\lambda_{j+1}^{\mathbb S^{d-2}} > \lambda_j^{\mathbb S^{d-2}}$). Since $f_j^\pm$ minimizes $J_j$ with the respective boundary conditions, 
we get $J_{j+1}(f_{j+1}^\pm) \geq J_j(f_{j+1}^\pm) \geq J_j(f_j^\pm)$, which gives 
\begin{equation}\label{thejsincrease}
\begin{aligned} 
\delta^2\mathcal F(0)[\zeta_{j+1}^+, \zeta_{j+1}^+] \geq\delta^2\mathcal F(0)[\zeta_{j}^+, \zeta_{j}^+],\ \text{ for all }\ j \geq 2\,,\\ 
\delta^2\mathcal F(0)[\zeta_{j+1}^-, \zeta_{j+1}^-] \geq\delta^2\mathcal F(0)[\zeta_{j}^-, \zeta_{j}^-], \ \text{ for all }\  j \geq 1\,,
\end{aligned}
\end{equation}
where the strict inequalities hold if and only if $\lambda_{j+1}^{\mathbb S^{d-2}} > \lambda_j^{\mathbb S^{d-2}}$. 

As a consequence of \eqref{thejsincrease}, the following proposition will conclude the proof of Corollary \ref{c:DeSiJe}.  

\begin{prop}[Kernel of De Silva-Jerison cone]\label{thm: eigenvaluesofQ}
Let $\xi \in \{\zeta^{\pm}_j\}_{\pm, j \in \mathbb N}$. Then $\delta^2\mathcal F(0)[\xi, \xi] = 0$ if and only if $\xi$ is a linear combination of $\zeta^-_j$ for $2 \leq j \leq d$. Furthermore, $\delta^2 \mathcal F(0)[\zeta_1^+, \zeta_1^+] > 0$. 
\end{prop}

\begin{proof} We divide the proof of the Theorem in two steps dealing respectively with the dilation $\zeta_1^+$ and the rotations $\zeta_j^-$, for $2\leq j\leq d$.

\noindent \textbf{Step 1. The dilation $\zeta_1^+$.} Our first claim is that $u_1^+$ minimizes $\ds\int_{\Omega_{b_{\nu, \theta_0}}} |\nabla f|^2  - (d-1)f^2 d\mathcal H^{d-1}$ amongst all $f$ equal to $-\zeta_1^+ = \frac{1}{\sqrt{\mathcal H^{d-2}(\mathbb S^{d-2})}}$ on $\partial \Omega_{b_{\nu, \theta_0}}$ and which are $L^2$ orthogonal to $b$ in $\Omega_{b_{\nu, \theta_0}}$.

Taking the first variation of the energy, subject to the orthogonality constraint, such a minimizer must satisfy an equation of the form $\ -\Delta f = (d-1)f + k b,\ $ where $k$ is some constant. However, if there exists an $f$ with the same boundary values as $u_1^+$ and which solves the above equation for some $k \neq \eta$, then it must be the case that $f$ is either a super or sub solution of the equation that $u_1^+$ solves, and thus must lie either below or above $u_1^+$ on all of $\Omega_{b_{\nu, \theta_0}}$ (we have a maximum principle because both functions are orthogonal to $b$). However, it is not possible that both $f > u_1^+$ and $f, u_1^+ \perp b$, a positive function. Similarly if $f < u_1^+$. Thus, given the boundary values $\zeta_1^+$, the only solution is $u_1^+$. 

Consider the function 
$$f = \frac{1}{\sqrt{\mathcal H^{d-2}(\mathbb S^{d-2})}}\left(1 + c\phi_0\right),$$
where $\phi_0=b/\kappa_0$, $c = - \int_{\Omega_{b_{\nu,\theta_0}}} \phi_0$ so that $f$ is $L^2$ orthogonal to $\phi_0$ on $\Omega_{b_{\nu, \theta_0}}$.  We have that \begin{equation}\label{e:u1plusfromabove}\begin{aligned} \delta^2 \mathcal F(0)[\zeta_1^+, \zeta_1^+]+E =& 2\,\int_{\Omega_{b_{\nu, \theta_0}}} |\nabla u_1^+|^2 - (d-1)(u_1^+)^2\,d\HH^{d-2}  -4(d-2)\tan(\theta_0)\cos^{d-2}(\theta_0) + E \\
=& 2\int_{\Omega_{b_{\nu, \theta_0}}} |\nabla f|^2 - (d-1)(f)^2\,d\HH^{d-2}  -4(d-2)\tan(\theta_0)\cos^{d-2}(\theta_0)\\
=& \frac{2c^2}{\mathcal H^{d-2}(\mathbb S^{d-2})}\int_{\Omega_{b_{\nu, \theta_0}}} |\nabla \phi_0|^2 - (d-1)(\phi_0)^2\,d\HH^{d-1} \\ +&\frac{2}{\mathcal H^{d-2}(\mathbb S^{d-2})}\left(2(d-1)c^2 -(d-1)\mathcal H^{d-1}(\Omega_{b_{\nu, \theta_0}})\right) - 4(d-2)\tan(\theta_0)\cos^{d-2}(\theta_0)\\
=& \frac{2(d-1)}{\mathcal H^{d-2}(\mathbb S^{d-2})}\left(2c^2 - \mathcal H^{d-1}(\Omega_{b_{\nu, \theta_0}})\right) - 4(d-2)\tan(\theta_0)\cos^{d-2}(\theta_0) ,\end{aligned}\end{equation}
where $E > 0$ is some error which reflects how far $f$ is from minimizing.

To estimate $E$ we first use that $u_1^+$ minimizes to get $$\begin{aligned}E=& 2\,\int_{\Omega_{b_{\nu, \theta_0}}} |\nabla f|^2 - (d-1)(f)^2\,d\HH^{d-2}- 2\,\int_{\Omega_{b_{\nu, \theta_0}}} |\nabla u_1^+|^2 - (d-1)(f)^2\,d\HH^{d-2}\\ =& 2\,\int_{\Omega_{b_{\nu, \theta_0}}} |\nabla (u_1^+-f)|^2 - (d-1)((u_1^+-f))^2\,d\HH^{d-2} = 2\sum_{j } \tilde{a}_j^2 (\lambda_j-(d-1)),\end{aligned}$$ where $\tilde{a}_j = \left\langle u_1^+ - f, \phi_j \right\rangle$ and $\phi_j$ is the $j$ Dirichlet eigenfunction of $\Omega_{b_{\nu,\theta}}$ with eigenvalue $\lambda_j$.

 Let us consider the Dirichlet eigenfunctions of $\Omega_{b_{\nu, \theta_0}}$; by separation of variables we can write any eigenfunction $\phi_j(\theta, \phi) = g_j(\theta) \psi_j(\varphi)$ where $\psi_j$ is a spherical harmonic on $\mathbb S^{d-2}$ and $g_j(\pi/2-\theta_0) = g_j(\pi/2+\theta_0) = 0$. In order for $\phi_j \not\perp 1, u_1^+$ but for $\phi_j \perp \phi_0$ it must be the case that $\psi_j$ is constant and $g_j$ is an even function with at least two interior zeroes. As such $g_j$ as two local critical points, at $\pi/2-\eta, \pi/2+\eta$ (for $\eta < \theta_0$) and $\partial_\theta g_j$ satisfies
$$\begin{aligned} -\partial^2_{\theta \theta}(\partial_\theta g_j(\theta)) - (d-2)\cot(\theta)\partial_\theta (\partial_\theta g_j(\theta)) =& \left(\lambda_j - \frac{d-2}{\sin^2(\theta)}\right)(\partial_\theta g_j(\theta)),\\ \partial_\theta g_j(\pi/2- \eta) =& \,\partial_\theta g_j(\pi/2 + \eta) = 0,\end{aligned}$$ where $\lambda_j$ is the eigenvalue associated to $\phi_j = g_j$.  But note that if $h(\theta)\psi_2^{\mathbb S^{d-2}}$ is an eigenfunction (and it is for some $h$), then $h$ minimizes the energy associated to the equation $L = -\Delta + \frac{(d-2)}{\sin^2(\theta)}$ but over Dirichlet boundary conditions on a larger integral. Thus the eigenvalue associated to $h(\theta)\psi_2^{\mathbb S^{d-2}}$ is smaller than $\lambda_j$ (as $\lambda_j$ is a Dirichlet eigenvalue, but not necessarily the first, associated to the same $L$ on a smaller domain). The Rayleigh quotient of $h(\theta)\psi_2^{\mathbb S^{d-2}}$ is simply the Rayleigh quotient of $h(\theta)$ times the Raleigh quotient of $\psi_2^{\mathbb S^{d-2}}$ (the latter being equal to $d-2$). The former is bounded below by $\lambda_1^{\Omega_{b_{\nu,\theta_0}}} = d-1$.  Therefore, $\tilde{a}_j \neq 0 \Rightarrow \lambda_j \geq (d-2)(d-1)$.
We claim that
$$\tilde{a}_j = \frac{\int \phi_j}{\sqrt{\mathcal H^{d-2}(\mathbb S^{d-2})}} \frac{(d-1)}{\lambda_j - (d-1)}\,,\quad\text{for every}\quad j\quad\text{such that}\quad \tilde a_j\neq0.$$  
Indeed $\left\langle u_1^+ - f, \phi_j\right\rangle = \left\langle u_1^+, \phi_j\right\rangle - \frac{1}{\sqrt{\mathcal H^{d-2}(\mathbb S^{d-2})}}\int \phi_j$. We also know that 
$\big\langle u_1^+, \phi_j\big\rangle=0$,
 for all $j> 0$ as $\phi_j$ has zero Dirichlet data and is orthogonal to $\phi_0$. Integrating by parts we get  
$$\begin{aligned}(\lambda_j - (d-1)) \left\langle u_1^+, \phi_j\right\rangle =& - \frac{1}{\sqrt{\mathcal H^{d-2}(\mathbb S^{d-2})}} \int_{\partial \Omega_{b_{\nu, \theta_0}}} \partial_\nu \phi_j\\ =&  \frac{1}{\sqrt{\mathcal H^{d-2}(\mathbb S^{d-2})}} \int_{\Omega_{b_{\nu, \theta_0}}} -\Delta \phi_j  = \frac{\lambda_j}{\sqrt{\mathcal H^{d-2}(\mathbb S^{d-2})}}\int_{\Omega_{b_{\nu, \theta_0}}} \phi_j.\end{aligned}$$ 
Putting everything together we get the claim. Also note that the above argument implies that $u_1^+\perp \phi_j$ ($\Leftrightarrow 1 \perp \phi_j$) as long as $j > 0$. So 
\begin{equation}\label{e:errorestimate} \begin{aligned}E &= \frac{2}{\mathcal H^{d-2}(\mathbb S^{d-2})} \sum_{\{j\mid \tilde{a}_j \neq 0\}} \left(\int \phi_j\right)^2 \frac{(d-1)^2}{\lambda_j - (d-1)}\\ &\stackrel{\lambda_j \geq (d-2)(d-1)}{\leq} \frac{3}{\mathcal H^{d-2}(\mathbb S^{d-2})} \sum_{\{j\mid \tilde{a}_j \neq 0\}} \left(\int \phi_j\right)^2 \leq \frac{3}{\mathcal H^{d-2}(\mathbb S^{d-2})} \left(\mathcal H^{d-1}(\Omega_{b_{\nu, \theta_0}}) - c^2\right),\end{aligned}\end{equation} 
where 
we assume that $d \geq 7$ so that $\frac{d-1}{d-3} \leq \frac{3}{2}$.
Putting \eqref{e:u1plusfromabove} together with \eqref{e:errorestimate} yields 
\begin{align} 
\delta^2 \mathcal F(0)[\zeta_1^+, \zeta_1^+]  &> \frac{2(d-1)}{\mathcal H^{d-2}(\mathbb S^{d-2})}\left(2c^2 - \mathcal H^{d-1}(\Omega_{b_{\nu, \theta_0}})\right) -4(d-2)\tan(\theta_0)\cos^{d-2}(\theta_0)\notag\\ 
&\quad- \frac{3}{\mathcal H^{d-2}(\mathbb S^{d-2})}\left(\mathcal H^{d-1}(\Omega_{b_{\nu, \theta_0}}) -c^2\right)\label{e:overestimateerror}\\
&= \frac{1}{\mathcal H^{d-2}(\mathbb S^{d-2})}\left((4d-1)c^2 - (2d+1)\mathcal H^{d-1}(\Omega_{b_{\nu, \theta_0}})\right) -4(d-2)\tan(\theta_0)\cos^{d-2}(\theta_0)\notag  .
\end{align}

\noindent\textbf{The case $d=7$.}
We will now verify that the right hand side of \eqref{e:overestimateerror} is positive when $d= 7$ via a numerical calculation. We use Mathematica, though, since these special functions are well known, one could do this by hand. In order to minimize the effect of rounding errors by the computer, we will round up negative terms and round down positive terms. As the calculation is delicate, we will have to go to four places right of the decimal. 

\noindent When $d = 7$, $\theta_0 \approx \sin^{-1}(.517331) \approx .54372$ (see \cite{DeJe}), and we have
$$\begin{aligned} \frac{\mathcal H^{d-1}(\Omega_{b_{\nu, \theta_0}})}{\mathcal H^{d-2}(\mathbb S^{d-2})} &< \int_{-.5438}^{.5438} \cos(\theta)^5 d\theta < .8650\\
4(d-2)\cos^{d-2}(\theta_0)\tan(\theta_0) &< 20\cos^5(.5437)\tan(.5438) < 5.5509 \\
\Rightarrow \frac{(2d+1)\mathcal H^{d-1}(\Omega_{b_{\nu, \theta_0}})}{\mathcal H^{d-2}(\mathbb S^{d-2})} +&4(d-2)\tan(\theta_0)\cos^{d-2}(\theta_0)  < 15*.8650+5.5609= 18.5359.
\end{aligned}$$
Now we calculate $\int \phi_0$. Following \cite{DeJe}, $$\phi_0 = c_{\theta} (1-\cos^2(\theta))^{-\frac{d-3}{4}}Q_{\frac{d-1}{2}}^{\frac{d-3}{2}}(\cos(\theta)),$$ where $c_\theta$ is the $L^2$ normalizing constant and $Q^\mu_\nu$ is the associated Legendre function of the second kind (when $d$ is even we work with $P^\mu_\nu$ the associated Legendre function of the first kind). We define these functions following the convention of Mathematica (which is the same convention used in \cite{DeJe} and \cite{TransFuncBook}), namely that, $$(1-t^2)\frac{d^2 Q^\mu_\nu(t)}{dt^2} - 2(\mu+1)t\frac{d Q^\mu_\nu(t)}{dt} + (\nu-\mu)(\nu+\mu+1)Q^\mu_\nu(t) = 0.$$

In the case $d= 7$, we can use Mathematica to calculate $$\|(1-\cos^2(\theta))^{-1}Q_{3}^{2}(\cos(\theta))\|^2_{L^2} < \mathcal H^{d-2}(\mathbb S^{d-2})\int_{-.5174}^{.5174} \left(Q_{3}^{2}(t)\right)^2 dt < 34.6188\Rightarrow c_\theta > \frac{.1699}{\sqrt{\mathcal H^{d-2}(\mathbb S^{d-2})}}$$

Then we have (using Mathematica again)  
$$\begin{aligned} \int \phi_0 >& \frac{.1699*\mathcal H^{d-2}(\mathbb S^{d-2})}{\sqrt{\mathcal H^{d-2}(\mathbb S^{d-2})}}\int_{\pi/2-\theta_0}^{\pi/2+\theta_0} \sin(\theta)^5 Q_{3}^{2}(\cos(\theta)) d\theta\\ >&.1699*\sqrt{\mathcal H^{d-2}(\mathbb S^{d-2})} \int_{-.5173}^{.5173} (1-t^2)^2 Q_{3}^2(t)dt > .8326*\sqrt{\mathcal H^{d-2}(\mathbb S^{d-2})}.\end{aligned}$$
Putting everything together we finally get 
$$\begin{aligned}\frac{4d-1}{\mathcal H^{d-2}(\mathbb S^{d-2})}\left(\int \phi_0\right)^2 > 27(.685) >& 18.7170 \qquad\text{and}\qquad \delta^2\mathcal F(0)[\zeta_1^+, \zeta_1^+] > 18.7170 - 18.5359 >0.\end{aligned}$$

\noindent {\bf An asymptotic argument in higher dimension.} Let us briefly sketch an asymptotic argument, which proves that the deformation corresponding to $\zeta_1^+$ is positive for all $d \geq 21$. This completes the proof as $7 < d \leq 20$ can be verified by hand in the manner outlined above. The first key observation is that \begin{equation}\label{e:asympforthetanaught} 
.62< \theta_0 \sqrt{d} \leq .65,\quad \forall d > 20.
\end{equation} 
The proof of \eqref{e:asympforthetanaught} is relatively straightforward: the lower bound follows from the fact that 
$$2\theta_0 \mathcal H^{d-2}(\mathbb S^{d-2}) \geq \mathcal H^{d-1}(\Omega_{b_{\nu, \theta_0}}) \geq \frac{1}{2}\mathcal H^{d-1}(\mathbb S^{d-1}),$$
the Gamma function formulas for the surface area of the sphere and Sterling approximation. 
The upper bound of \eqref{e:asympforthetanaught} is a bit harder; the sharp isoperimetric inequality tells us  that the perimeter of $\Omega_{b_{\nu, \theta_0}}$ is larger than the half sphere's perimeter. So we have $$2\cos^{d-2}(\theta_0)\mathcal H^{d-2}(\mathbb S^{d-2}) > \mathcal H^{d-2}(\mathbb S^{d-2}).$$ Approximating $\cos$ by its Taylor series and doing some elementary estimates yields the result. 
Now, from \eqref{e:asympforthetanaught} it is easy to get \begin{equation}\label{e:goodestimateonmeasure}
1.3 > 2\sqrt{d}\theta_0 >  \frac{\sqrt{d}\mathcal H^{d-1}(\Omega_{b_{\nu, \theta_0}})}{\mathcal H^{d-2}(\mathbb S^{d-2})} > \frac{\sqrt{d}}{2} \frac{\mathcal H^{d-1}(\mathbb S^{d-1})}{\mathcal H^{d-2}(\mathbb S^{d-2})} > \frac{\sqrt{2\pi}}{2} > 1.24.
\end{equation}
Furthermore, we can also estimate \begin{equation}\label{e:lowerboundonc} \frac{\sqrt{d}c^2}{\mathcal H^{d-2}(\mathbb S^{d-2})} \geq \frac{4\cos^{2d-4}(\theta_0)}{1.3}, \: \forall d > 20.\end{equation} This is a little bit trickier, but it follows once one observes that  $$c \equiv \int \phi_0 = \frac{1}{d-1}\int -\Delta \phi_0 = \frac{\mathcal H^{d-2}(\partial \Omega_{b_{\nu,\theta_0}})}{(d-1)\kappa_0}\qquad\text{and}\qquad\frac{\mathcal H^{d-1}(\Omega_{b_{\nu,\theta_0}})}{\kappa_0^2} \geq \int |\nabla \phi_0|^2 = d-1.$$
Plugging \eqref{e:goodestimateonmeasure} and \eqref{e:asympforthetanaught} into \eqref{e:overestimateerror}
we find that $\delta^2\mathcal F(0)[\zeta_1^+, \zeta_1^+] > 0$ for all $d > 20$.

\noindent\textbf{Step 2. Rotations and the kernel of $\delta^2\mathcal F(0)$.} In order to conclude the proof of Proposition \ref{thm: eigenvaluesofQ}, and to prove that $\delta^2\mathcal F(0)[\zeta_j^-, \zeta_j^-] > 0$ for $j > d$, it suffices to show that 
$$
\delta^2\mathcal F(0)[\zeta_j^-, \zeta_j^-] = 0,\; 2\leq j \leq d.
$$ 

Using the commutator relation $\ds[\partial_\theta, -\Delta] = \frac{d-2}{\sin^2(\theta)}\partial_{\theta}$, we get that 
$$
u_j^-(\theta, \varphi) =-\kappa_0\psi_j(\varphi)\partial_\theta \phi_0,\: 2\leq j \leq d.
$$
It is then easy to compute (after observing  $-\partial_{\theta \theta}\phi_0(\pi/2+\theta_0) = -(d-2)\tan(\theta_0)\partial_\theta \phi_0(\pi/2+\theta_0)$) 
$$
\delta^2\mathcal F(0)[\zeta_j^-, \zeta_j^-] =4(d-2) \cos^{d-2}(\theta_0)(\tan(\theta_0) - \tan(\theta_0)) = 0,\; \forall 2 \leq j \leq d.
$$
As we stated above, invoking \eqref{thejsincrease}, this proves that $\delta^2\mathcal F(0)[\zeta_j^-, \zeta_j^-] > 0$ for all $j > d$ and that $\delta^2\mathcal F(0)[\zeta_1^-, \zeta_1^-] < 0$.
\smallskip

 In the same manner, to prove that $\delta^2\mathcal F(0)[\zeta_j^+, \zeta_j^+] > 0$ for all $j > d$ it suffices to prove that 
 $$
 \delta^2\mathcal F(0)[\zeta_{d+1}^+, \zeta_{d+1}^+]>0.
 $$ 
 To do so, we observe that the function
 \begin{equation}\label{whoisud1plus}
u^+_{d+1}(\theta,\varphi) = -\frac{\kappa_0}{(d-2)\tan(\theta_0)}\psi_{d+1}(\varphi)\left(\partial^2_{\theta\theta}\phi_0(\theta)+ \phi_0(\theta)\right)
\end{equation}
satisfies \eqref{e:solutionpmj}, for $j = d+1$. To see this, note that $\lambda_{d+1}^{\mathbb S^{d-2}} = 2(d-1)$ and that 
$$
-\Delta \partial^2_{\theta \theta} \phi_0 = \left((d-1) - \frac{2(d-1)}{\sin^2(\theta)}\right)\partial^2_{\theta\theta}\phi_0 - \frac{2(d-1)}{\sin^2(\theta)}\phi_0.
$$
We can then compute that 
\begin{equation}\label{e:energyofthelargeplusses}\begin{aligned}
\delta^2\mathcal F(0)[\zeta_{d+1}^+, \zeta_{d+1}^+] 
&= 4\cos^{d-2}(\theta_0)\left(\int_{\mathbb S^{d-2}} u_{d+1}^+\partial_\nu u_{d+1}^+\ d\sigma - (d-2)\tan(\theta_0)\right)\\ 
&= 4\cos^{d-2}(\theta_0)\left((d-1)\tan(\theta_0) - (d-2)\tan(\theta_0)\right) > 0.
\end{aligned}
\end{equation}
	\end{proof}

\appendix 
\renewcommand{\thesection}{\Alph{section}}
\setcounter{section}{0}

\setcounter{teo}{0}
     \renewcommand{\theteo}{\Alph{section}.\arabic{teo}}
	\section{Proof of Lemma \ref{l:main}}\label{s:app}
	
	We derive the first and second variation of the Alt-Caffarelli functional restricted to the sphere. We show that the second variation is continuous at zero and that at zero it is diagonalizable. The notation is the same as in Subsection \ref{ss:var}.

	\renewcommand{\thesubsection}{\Alph{section}.\arabic{subsection}}
\setcounter{subsection}{0}
\renewcommand{\theteo}{\Alph{section}.\arabic{teo}}
\renewcommand{\theequation}{\Alph{section}.\arabic{equation}}

	\subsection{First and second variation of $\mathcal F$}\label{sub:first_and_second_variation}

	The following Lemma contains the explicit formulas for the variations of $\mathcal F$ around a domain $\Omega:=\{b>0\}\cap \de B_1\subset \de B_1$, where $b$ is a $1$-homogeneous minimizer of $\mathcal E$ with isolated singularity, so that, by \cite{AlCa}, $\de \Omega:=\de\{b>0\}\cap \de B_1$ is smooth; we denote by $\nu$ the exterior normal to $\partial\Omega$ in the sphere. Keeping the notation from Subsection \ref{ss:var}, $b$ is given by $\kappa_0\phi^{\Omega}_1:=b$. To simplify notation we denote $\phi^{\Omega}_1$ simply as $\phi_1$. Recall that,
$$-\Delta_{\de B_1} \phi_1=(d-1) \,\phi_1\quad \mbox{in}\quad\Omega\,,\qquad \int_{\Omega}\phi_1^2\,d\HH^{d-1}=1\ , \qquad \phi_1=0\quad \mbox{\and}\quad \kappa_0\,\de_\nu \phi_1=-1 \quad\mbox{on}\quad\de \Omega\,.$$	
For every $g\in C^{2,\alpha}(\partial\Omega)$, $\zeta\in C^{2,\alpha}(\partial\Omega)$ and $t\in\R$ we define the function $\Psi_{g,t}:\partial\Omega\to\mathbb S^{d-1}$ given through the spherical exponential map
$$\Psi_{g,t}(x):=\exp_x\Big(\big(g(x)+t\,\zeta(x)\big)\,\nu(x)\Big)=\cos\big(g(x)+t\,\zeta(x)\big)x+\sin\big(g(x)+t\,\zeta(x)\big)\nu(x).$$
We notice that the exponential map $(0,\pi)\times\partial\Omega\ni(s,x)\mapsto \exp_x\big(s\nu(x)\big)$ is a diffeomorphism in a neighbourhood of $\partial\Omega$. Then, for $\|g\|_{C^{2,\alpha}(\partial\Omega)}$ and $\|\zeta\|_{C^{2,\alpha}(\partial\Omega)}$ small enough, the map 
$$[-2,2]\times\partial\Omega\ni (t,x)\mapsto \Psi_{g,t}(x),$$
is injective and smooth (and also a diffeomorphism from $]-2,2[\times\{\zeta\neq0\}$ onto the image of $\Psi_{g}:=\Psi_{g,0}$) and so, the derivative 
$$\partial_t\Psi_{g,t}(x)=\zeta(x)\xi_t(x),\quad\text{where}\quad \xi_t(x):=\cos\big(g(x)+t\, \zeta(x)\big)\,\nu_{\partial\Omega}(x)-\sin\big(g(x)+t\, \zeta(x)\big)\,x\,,$$
defines a vector field $X$ on the closed set $\Psi_g([-2,2]\times\partial\Omega)$ such that 
\begin{equation}\label{e:Psi_t_explicit}
\qquad X(\Psi_{g,t}(x))=\zeta(x)\,\xi_t(x)\,,\quad\text{for every}\quad x\in\partial\Omega,
\end{equation}
which can be extended to the entire sphere by the Whitney's extension theorem. 
Moreover, the map $\Psi_{g,t}:\partial\Omega\to \partial\Omega_{g,t}$ is a diffeomorphism, where $\partial\Omega_{g,t}$ is the boundary of a spherical set $\Omega_{g,t}$. Setting, $\Omega_g:=\Omega_{g,0}$, we notice that the flow associated to this vector field $X$ is an extension to $\de B_1$ of the map $\Psi_{g,t}\circ\Psi_g^{-1}:\partial\Omega_g\to\partial \Omega_{g,t}$, that is we have   
$$\de_t \Psi_{g,t}(\Psi_g^{-1}(x)):=X(\Psi_{g,t}(\Psi_g^{-1}(x)))\,\qquad\mbox{and}\qquad \Psi_g(\Psi_g^{-1}(x))=x\quad \mbox{for}\quad x\in \partial\Omega_g.$$
In the next lemma we calculate the first and the second variations of $\mathcal F$.
\begin{lemma}\label{l:variation}
Let $\Omega$, $k_0$ and $X$ be as above. 
Then, for every $\zeta,g \in C^{2,\alpha}(\de \Omega)$, we have 				\begin{align}
	\delta\mathcal F(g)[\zeta]&=\int_{\de \Omega_{g}}\,(X\cdot \nu_g)\,d\HH^{d-2}-\kappa_0^2 \int_{\de \Omega_g} \partial_X\phi_g\,\de_{\nu_g} \phi_g\, d\HH^{d-2} \notag,\\
			\delta^2\mathcal F(g)[\zeta,\zeta]
			&=\int_{\de \Omega_g} {\rm div}_{\de B_1} X\,(X\cdot \nu_g) \,d\HH^{d-2}+\kappa_0^2\int_{\de \Omega_g} \big(-\de_{XX} \phi_g-2\partial_X \phi_g'\,\big)\de_{\nu_g}\phi_g\,d\HH^{d-2}\notag,
\end{align}
where $\nu_g$ is the outward unit normal to $\partial\Omega_g$, $H_{\partial\Omega_g}$ is the scalar mean curvature of $\de \Omega_g$, $\phi_g=\phi_1^{\Omega_g}$ is the first eigenfunction of $\Omega_g$ and $\phi_g'$ is the solution of 
$$-\Delta_{\de B_1} \phi_g'=\lambda_g\phi_g'-\phi_g\int_{\partial\Omega_g}\partial_X\phi_g\,\partial_{\nu_g}\phi_g\quad\text{in}\quad\Omega_g,\qquad\phi_g=0\quad\text{on}\quad\partial\Omega_g,\qquad\int_{\Omega_g}\phi'_g\phi_g=0.$$			
\end{lemma}
\begin{oss}
We notice that the function $\phi'_g$ depends linearly on $\zeta$. A more precise (but heavier) notation would be $\phi'_g=\delta\phi(g)[\zeta]$. 
\end{oss}

	\begin{proof}  
	Recall the following Hadamard formula, whose proof can be found in \cite[Section 5.2]{HP}
	\begin{equation}\label{e:Had1}
	\frac{d}{dt} \int_{\Omega_{g,t}} f(t,x) \, d\HH^{d-1}(x) = \int_{\Omega_{g,t}} \partial_t f(t,x)  \, d\HH^{d-1}(x)  + \int_{\partial\Omega_{g,t}} f(t,x) (X(x) \cdot \nu_t(x))  \, d\HH^{d-2}(x)\,,   
	\end{equation}
	where $\nu_t$ denotes the outward pointing normal to a domain $ \Omega_{g,t}$ in the sphere. 
	Applying this law we see immediately that
	\begin{equation}\label{e:m'}
	\frac{d}{dt}\HH^{d-1}(\Omega_{g,t})=\delta m(g+t\zeta)[\zeta]=\int_{\de \Omega_{g,t}}\,(X\cdot \nu_t)\,d\HH^{d-2} =\int_{\Omega_{g,t}} {\rm div}_{\de B_1} X\,  d\HH^{d-1}.
	\end{equation}
	Then we can apply \eqref{e:Had1} again to get
	\begin{align}\label{e:m''}
	\delta^2m(g)[\zeta,\zeta]
		&=\int_{\Omega_g} \de_t({\rm div_{\de B_1}} X)\, d\HH^{d-1}+\int_{\de \Omega_g} {\rm div}_{\de B_1} X\,(X\cdot \nu_g)\, d\HH^{d-2}=\int_{\de \Omega_g} {\rm div}_{\de B_1} X\,(X\cdot \nu_g) \,d\HH^{d-2}
\end{align}
where we used the fact that $X$ is autonomous to conclude that $\de_t ({\rm div} X)={\rm div} (\de_t X)=0$. 
		In particular, when $g=0$, we have that $X=\zeta\,\nu$ on $\de \Omega$ and so
	\begin{gather}
	\delta m(0)[\zeta]=\int_{\de \Omega } \zeta\, \,d\HH^{d-2}\qquad\text{and}\qquad
	\delta^2m(0)[\zeta,\zeta]=\int_{\de \Omega} H_{\de \Omega}\,\zeta^2\, d\HH^{d-2}. \label{e:m''0}
	\end{gather}
	
	We now calculate the first and the second variation of the functional $\lambda: C^{2,\alpha}(\partial \Omega)\to \R$. 
	We first notice that by \cite[Theorem 5.7.4]{HP}, the map $t\mapsto \lambda(g+t\zeta)$ is $C^\infty$ in a neighborhood of zero and the first and the second derivatives have been computed in \cite[Theorem 5.7.1]{HP} and \cite[Section 5.9.6]{HP} for sets in $\R^d$. We also notice that the $C^3$ regularity condition from \cite[Section 5.9.6]{HP} can be replaced by $C^{2,\alpha}$ as it was shown in \cite{dambrine} and \cite{BrDeVe}. Below, we formally derive the exact expressions of $\delta\lambda$ and $\delta^2\lambda$ on the sphere.
		For the sake of simplicity we set 
	$$\phi_t:=\phi_{g+t\zeta}=\phi_1^{\Omega_{g,t}}\ ,\quad \Omega_t:=\Omega_{g,t}\ ,\quad \lambda_t=\lambda(g+t\zeta)\ ,\quad \Psi_t:=\Psi_{g,t}.$$ 
	Using \eqref{e:Had1} once again and the fact that $\phi_t=0$ on $\de \Omega_{t}$, we obtain
	\begin{equation}\label{e:ort}
	0=\de_t \int_{\Omega_t} \phi_t^2=2\int_{\Omega_t} \phi_t\,\phi'_t+\int_{\de \Omega_t} \phi^2_t\,(X\cdot \nu_t)= 2\int_{\Omega_t} \phi_t\,\phi'_t\,.
	\end{equation}
	By definition of $\phi_t$ and $\Psi_t$ we have  
	$$
	0=\phi_t(\Psi_t(x))=\phi_t\big(\cos (g(x)+t \zeta(x))\, x+\sin (g(x)+t\,\zeta(x))\,\nu(x)\big)\qquad  \mbox{for every $x\in \de\Omega$}
	$$
	so that differentiating we get for every $x\in \de \Omega$
	\begin{align}
	\phi_t'(\Psi_t)&=\,\zeta\big(\sin(g+t\zeta)\,x-\cos(g+t\, \zeta)\,\nu \big)\cdot D\phi_t(\Psi_t)=-\zeta\xi_t\cdot D\phi_t(\Psi_t)\,, 	\label{e:bd1}\\
	\phi''_t(\Psi_t)&=-\zeta^2\big(\sin(g+t\zeta)\,x-\cos(g+t\, \zeta)\,\nu \big)\cdot D^2 \phi_t(\Psi_t)\,\big[\sin(g+t\zeta)\,x-\cos(g+t\, \zeta)\,\nu \big]\notag\\
	&\qquad+2 \,\zeta\big(\sin(g+t\zeta)\,x-\cos(g+t\, \zeta)\,\nu \big) \cdot D \phi_t'(\Psi_t)\notag\\
	&\qquad\qquad-\zeta^2\big(\cos(g+t\zeta)\,x+\sin(g+t\, \zeta)\,\nu\big)\cdot D \phi_t(\Psi_t)	\notag\\
	&=-\zeta^2\,\xi_t\cdot D^2 \phi_t(\Psi_t)\,\xi_t-2 \,\zeta\,\xi_t \cdot D \phi_t'(\Psi_t)\label{e:bd2},
	\end{align}
where, we used that $\Psi_t\cdot D \phi_t(\Psi_t)=0$.
\noindent Differentiating formally the equation for $\phi_t$,  we obtain
	\begin{gather}
	-\Delta_{\de B_1} \phi'_t=\lambda_t\,\phi'_t+\lambda'_t\,\phi_t \quad \mbox{in}\quad\Omega_t\notag\,,\\
	-\Delta_{\de B_1} \phi''_t=\lambda_t\, \phi''_t+2\,\lambda'_t\,\phi'_t+\lambda''_t\,\phi_t\quad \mbox{in}\quad\Omega_t\notag\,,
	\end{gather} 
	where $\lambda'_t=\delta\lambda(g+t\zeta)[\zeta]$ and $\lambda''_t=\delta^2\lambda(g+t\zeta)[\zeta,\zeta]$.
\noindent	Multiplying the first of these two equations by $\phi_t$, we get
	\begin{align}
	\kappa_0^2\,\delta\lambda(g+t\zeta)[\zeta]&=\kappa_0^2 \int_{\de \Omega_t} \phi'_t\,\de_{\nu_t} \phi_t\,d\HH^{d-2}=\kappa_0^2\int_{\de \Omega} \phi'_t(\Psi_t)\,\de_{\nu_t} \phi_t(\Psi_t)\, J\Psi_t\,d\HH^{d-2}\label{e:lambda't}\\
	&\stackrel{\eqref{e:bd1}}{=}-\kappa_0^2 \int_{\de \Omega} \zeta\,\xi_t\cdot D\phi_t(\Psi_t)\,\de_{\nu_t} \phi_t(\Psi_t)\, J\Psi_t\,d\HH^{d-2}=-\kappa_0^2 \int_{\de \Omega_t} \partial_X\phi_t\,\de_{\nu_t} \phi_t\, d\HH^{d-2}\notag.
	\end{align}
	where the first equality follows from
	\begin{align*}
	\int_{\Omega_t} \Delta_{\de B_1}\phi_t'\,\phi_t\,d\HH^{d-1}
	&=-\int_{\Omega_t} \nabla \phi_t'\,\nabla \phi_t\,d\HH^{d-1}+\underbrace{\int_{\de\Omega_t} \de_{\nu_t}\phi_t'\, \phi_t\,d\HH^{d-2}}_{=0\mbox{ since }\phi_t=0\mbox{ on }\de\Omega_t}\\
	&=\underbrace{\int_{\Omega_t}  \phi_t'\,\Delta_{\de B_1} \phi_t\,d\HH^{d-1}}_{=0\mbox{ by }\Delta_{\de B_1}\phi_t=\lambda_t\phi_t\mbox{ and \eqref{e:ort}}}-\int_{\de\Omega_t} \phi_t'\, \de_{\nu_t}\phi_t\,d\HH^{d-2}
	\end{align*}
	Using the definition of $\kappa_0$, for $g=0$ and $t=0$ yields
	\begin{equation}\label{e:lambda'0}
	\kappa_0^2\,\delta\lambda(0)[\zeta]=-\kappa_0^2 \int_{\de \Omega} \zeta\,|\de_{\nu} \phi_0|^2\,d\HH^{d-2}=-\int_{\de \Omega} \zeta\,d\HH^{d-2}\,.
	\end{equation}
	In a similar way, multiplying the equation for $\phi''_t$ by $\phi_t$, integrating by parts and using \eqref{e:ort} and \eqref{e:bd1}, we get
	\begin{align*}
	\kappa_0^2\, \delta^2\lambda(g+t\zeta)[\zeta,\zeta]
	=\kappa_0^2\int_{\de \Omega_t} \phi''_t\,\de_{\nu_t} \phi_t\,d\HH^{d-2}=\kappa_0^2\int_{\de \Omega} \phi''_t(\Psi_t)\,\de_{\nu_t} \phi_t(\Psi_t)\,J\Psi_t\,d\HH^{d-2}.
	\end{align*} 
	By \eqref{e:bd2} we get 
	\begin{align}\label{e:lambda''}
	\kappa_0^2\,\delta^2\lambda(g+t\zeta)[\zeta,\zeta]
	&=\kappa_0^2\,\int_{\de \Omega} \Big[-\zeta^2\,\xi_t\cdot D^2 \phi_t(\Psi_t)\,\xi_t-2 \,\zeta\,\xi_t \cdot D \phi_t'(\Psi_t)\Big]\,\de_{\nu_t} \phi_t(\Psi_t)\,J\Psi_t\,d\HH^{d-2}\\\notag
	&=\kappa_0^2\int_{\de \Omega_t} \big(-\de_{XX} \phi_t-2\partial_X \phi_t'\big)\,\de_{\nu_t}\phi_t\,d\HH^{d-2}.
	\end{align} 
	Evaluating the above expression at $g=0$ and $t=0$ and using the definition of $\kappa_0$, we conclude
	\begin{equation}\label{e:lambda''0}
	\kappa_0^2\, \delta^2\lambda(0)[\zeta,\zeta]=\int_{\de \Omega} \zeta^2 \, H_{\de\Omega}\,d\HH^{d-2}
	+2\,\kappa_0^2\,\int_{\de \Omega} \phi'_0\,\de_{\nu}\phi'_0\,d\HH^{d-2}
	\end{equation}
	Combining \eqref{e:m''}, \eqref{e:m''0}, \eqref{e:bd1}, \eqref{e:lambda'0}, \eqref{e:lambda''} and \eqref{e:lambda''0}, and setting $u_\zeta:=\kappa_0\phi'_0$, we conclude the proof of the lemma.
	\end{proof}

	\subsection{The first and the second variation in zero}\label{sub:variations_in_zero} 
	Let $\zeta \in C^{2,\alpha}(\de \Omega)$. We first notice that, by Lemma \ref{l:variation}, equations \eqref{e:m''0}, \eqref{e:lambda'0} and \eqref{e:lambda''0}, we have 
\begin{gather}
\delta\mathcal F(0)[\zeta]=0
\qquad\text{and}\qquad
\delta^2\mathcal F(0)[\zeta,\zeta]=2\int_{\de \Omega} \zeta^2 \, H_{\de\Omega}\,d\HH^{d-2}
+2\,\int_{\de \Omega} \zeta\,\,T\zeta\,d\HH^{d-2}\label{e:var''0},
\end{gather}	
where $T\colon H^{\sfrac12}(\de \Omega)\to H^{-\sfrac12}(\de \Omega)$ is the Dirichlet to Neumann operator defined by $T\zeta=\de_\nu u_\zeta$, where $u_\zeta$ is the solution of 
\begin{equation}\label{e:psi'0}
\begin{cases}
\begin{array}{rcl}\ds -\Delta_{\de B_1} u_\zeta  &=&  \ds (d-1) u_\zeta-\frac{1}{\kappa_0^2}\left(\int_{\de \Omega}\zeta\,d\HH^{d-2}\right)\,b \quad \mbox{in}\quad \Omega\\
\ds u_\zeta &=& -\zeta\quad \mbox{on}\quad\de \Omega\qquad\text{and}\qquad
\ds \int_{\Omega}b\,u_\zeta\,d\HH^{d-1}=0\,.
\end{array}
\end{cases}
\end{equation}
\begin{oss}
In the notation of Lemma \ref{l:variation}, we have $u_\zeta:=\kappa_0\phi_0'$.
\end{oss}
In this subsection we prove \eqref{e:diagonalization} by diagonalizing the bilinear form $\delta^2\mathcal F(0)$. To this aim we first notice that the linear operator $T$ is well defined. Indeed, the solution of \eqref{e:psi'0} is unique since if there were two solutions $u_1$ and $u_2$, then the difference $v:=u_1-u_2$ would be a solution of the eigenvalue problem
	$$\ds -\Delta_{\de B_1} v=(d-1)v\quad\text{in}\quad \Omega,\qquad
	v=0\quad\text{on}\quad \partial\Omega.	$$
	Thus, $v=C\,b$ for some constant. Now, the orthogonality condition $\ds\int_{\Omega}v\,b=0$ implies that the constant is zero. The existence of a solution of \eqref{e:psi'0} now follows by the Fredholm alternative.
	
	Therefore we can consider the linear operator
	$$
	H^{\sfrac12}(\de \Omega) \ni \zeta \mapsto B(\zeta):=T\zeta+(H_{\de \Omega}+\Lambda)\,\zeta\in H^{-\sfrac12}(\de\Omega)
	$$
	where $\Lambda:=\|H_{\de \Omega}\|_{L^\infty}+ C_{\Omega}$ ($C_{\Omega} > 0$ is a constant which depends only on $\Omega$; it is proportional to the sum of the norm of the trace operator and to the bounds given by elliptic regularity for the Laplacian on $\Omega$). $B$ is a self-adjoint, positive linear operator, with compact inverse. Indeed, it is sufficient to notice that 
\begin{align*}
\int_{\de \Omega} T(\zeta_1)\,\zeta_2\,d\HH^{d-2}&=\int_{\Omega}( \Delta _{\de B_1}u_{\zeta_1}\, u_{\zeta_2} +\nabla u_{\zeta_1}\cdot\nabla u_{\zeta_2} )  \, d\HH^{d-2}\\
&= \int_{\Omega} \left(-(d-1) u_{\zeta_1}\, u_{\zeta_2}+\nabla u_{\zeta_1}\cdot\nabla u_{\zeta_2} \right)  \, d\HH^{d-2}\,,
\end{align*}
where in the last equality we used the orthogonality of $u_{\zeta}$ and $b$. The theory of compact operators now implies that there exists a sequence of positive eigenvalues $(\tilde \lambda_i)_i$ accumulating to $\infty$ and of eigenfunctions $(\xi_i)_i$ which form an orthonormal basis of $L^2(\de \Omega)$ satisfying
	$$
	B\xi_i=\tilde \lambda_i\,\xi_i\qquad \mbox{for every }i\in \N\,.
	$$
	 In particular the sequence $\lambda_i:=\tilde{\lambda_i}-\Lambda$ and the functions $(\xi_i)_i$, for $i\in \N$, satisfy \eqref{e:diagonalization}. We also note that elliptic regularity tells us that the eigenfunctions of $T + H_{\partial \Omega}$ are as regular as $H_{\partial \Omega}$; since $\partial \Omega$ is locally analytic, we can conclude that $\xi_i \in C^{2,\alpha}$.

	 The upper bound \eqref{e:itr} follows straightforwardly from the previous analysis.

	 \subsection{Uniform bounds on $\phi_{g+t\zeta}$}	
	Recall that $\phi_{g+t\zeta}$ satisfies the elliptic equation 
$$(-\Delta_{\de B_1}-\lambda_{g+t\zeta}) \phi_{g+t\zeta} = 0\quad\text{in}\quad\Omega_{g,t},\qquad\phi_{g+t\zeta}=0\quad\text{on}\quad\partial\Omega_{g,t},\qquad\int_{\Omega_{g,t}}\phi_{g+t\zeta}^2=1.$$ 
The spherical domain $\Omega_g$ is $C^{2,\alpha}$ smooth and the $C^{2,\alpha}$ norm depends only on $\|g\|_{C^{2,\alpha}}$. Thus by classical elliptic regularity/Schauder estimates (see \cite{GiTr}), there is a universal constant $C$ such that
\begin{equation}\label{e:phi_t_bound}
\|\phi_g\|_{C^{2,\alpha}(\overline{\Omega}_g)} \le C(\|g\|_{C^{2,\alpha}}+1).
\end{equation}
\subsection{Bounds on $\phi_g'$}\label{sub:bounds_on_phi'} In this subsection we prove \eqref{e:bound_on_g}, which in the notation of this section reads as $\|\phi_g'\|_{L^2(\Omega_g)}\le C\|\zeta\|_{L^2(\partial\Omega)}.$
We first prove {a bound on $ \delta\lambda(g)[\zeta]$}. Recall that by \eqref{e:lambda't} 
$$\lambda_g'=\delta\lambda(g)[\zeta]= -\int_{\de \Omega_g} \partial_X\phi_g\,\de_{\nu_g} \phi_g\, d\HH^{d-2}.$$
This, together with \eqref{e:phi_t_bound}, implies that
\begin{equation}\label{eqn:dotlambdaisbounded} \big|\delta\lambda(g)[\zeta]\big| \leq \|\phi_g\|^2_{C^1(\overline{\Omega}_g)} \|\zeta\|_{L^2} \sqrt{\mathcal H^{d-2}(\partial \Omega_g)}  < C^2\left(\|g\|_{C^{2,\alpha}}^2+1\right)\|\zeta\|_{L^2},\end{equation}
where $C$ is the constant from \eqref{e:phi_t_bound}.
Recall that by \eqref{e:bd1} $\phi'_g$ is a solution of the equation 
$$(-\Delta_{\de B_1}-\lambda_g)\phi'_g=\lambda'_g\phi_g\quad\text{in}\quad \Omega_g,\qquad\phi'_g=-\partial_{X}\phi_g\quad\text{on}\quad\partial\Omega_g,\qquad\int_{\Omega_g}\phi_g'\phi_g\,d\HH^{d-1}=0.$$
Thus, $\phi'_g$ can be decomposed as 
\begin{equation}\label{e:phi'_t_decomposition}
\phi_g'=h_g+\psi_g-\phi_g\int_{\Omega_g}h_g\phi_g,
\end{equation}
where $h_g$ and $\psi_g$ are the solutions of the problems 
$$\Delta_{\de B_1} h_g=0\quad\text{in}\quad \Omega_g,\qquad h_g=-\partial_{X}\phi_g\quad\text{on}\quad\partial\Omega_g.$$
\begin{equation}\label{e:psit}
(-\Delta_{\de B_1}-\lambda_g)\psi_g=\lambda_gh_g+\lambda'_g\phi_g\quad\text{in}\quad \Omega_g,\qquad\psi_g=0\quad\text{on}\quad\partial\Omega_g,\qquad\int_{\Omega_g}\psi_g\phi_g\,d\HH^{d-1}=0.
\end{equation}
Notice that $\lambda_g$ is the lowest eigenvalue on $\Omega_g$ and its eigenspace is one-dimensional and generated by $\phi_g$. Thus, the orthogonality $\ds\int_{\Omega_g}\psi_g\phi_g=0$ implies that there is some constant $c_b>0$ such that 
$$c_b\int_{\Omega_g}\big(|\nabla\psi_g|^2+\psi_g^2\big)\,d\HH^{d-1}\le \int_{\Omega_g}\Big(|\nabla\psi_g|^2-\lambda_g\psi_g^2\Big)\,d\HH^{d-1}.$$
Thus, multiplying by $\psi_g$ and integrating by parts in \eqref{e:psit}, we get 
$$c_b\|\psi_g\|_{H^1(\Omega_g)}^2\le \int_{\Omega_g}\psi_g(\lambda_gh_g+\lambda'_g\phi_g)\,dx=\lambda_g\int_{\Omega_g}\psi_g h_g\,dx\le \lambda_g\|\psi_g\|_{L^2(\Omega_g)} \|h_g\|_{L^2(\Omega_g)},$$
which in turn gives  
$$\|\phi_g'\|_{L^2(\Omega_g)}\le \|\psi_g\|_{L^2(\Omega_g)}+2\|h_g\|_{L^2(\Omega_g)}\le \left(\frac{\lambda_g}{c_b}+2\right)\|h_g\|_{L^2(\Omega_g)}.$$
Now notice that by the maximum principle we have $\|h_g\|_{L^\infty(\Omega_g)}\le C\|\zeta\|_{L^\infty(\partial\Omega)}$ and, by \cite[Lemma 10]{dambrine}, $\|h_g\|_{L^1(\Omega_g)}\le C\|\zeta\|_{L^1(\partial\Omega)}$. Thus, by interpolation, we get 
$\|h_g\|_{L^2(\Omega_g)}\le C\|\zeta\|_{L^2(\Omega)}$ and finally we obtain \eqref{e:bound_on_g}. 
		
\subsection{Proof of \eqref{e:modulus_of_continuity}} \label{sub:modulus_of_continuity} The continuity of $\delta^2\mathcal F(g)[\zeta, \zeta]$ at zero follows by the continuity of $\delta^2\lambda(g)[\zeta, \zeta]$ and $\delta^2 m(g)[\zeta, \zeta]$. We give the details for $\delta^2 \lambda$ (the more complicated case) but the arguments for $\delta^2m$ are in the same vein. 
By \eqref{e:lambda''} 
we have 
\begin{align*}
\delta^2\lambda(g)[\zeta,\zeta]-\delta^2\lambda(0)[\zeta,\zeta]&=\int_{\de \Omega} \Big[-\zeta^2\,\xi_g\cdot D^2 \phi_g(\Psi_g)[\xi_g]-2 \,\zeta\,\xi_g \cdot D \phi_g'(\Psi_g)\Big]\,\de_{\nu_g} \phi_g(\Psi_g)\,J\Psi_g\,d\HH^{d-2}\\
&\quad-\int_{\de \Omega} \Big[-\zeta^2\,\nu\cdot D^2 \phi_0[\nu]-2 \,\zeta\,\nu \cdot D \phi_0'\Big]\,\de_{\nu} \phi_0\,d\HH^{d-2},
\end{align*}
which in turn can be split into
\begin{align*}
I_1+I_2+I_3&:=\int_{\de \Omega} \zeta^2\Big[-\,\xi_g\cdot D^2 \phi_g(\Psi_g)[\xi_g]\,\de_{\nu_g} \phi_g(\Psi_g)\,J\Psi_g+\nu\cdot D^2 \phi_0[\nu]\,\de_{\nu} \phi_0\Big]\,d\HH^{d-2}\\
&\quad-2\int_{\de \Omega} \Big[\zeta\,\xi_g \cdot D \phi_g'(\Psi_g)\,\de_{\nu_g} \phi_g(\Psi_g)-\zeta\,\nu_g(\Psi_g) \cdot D \phi_g'(\Psi_g)\,\de_{\xi_g} \phi_g(\Psi_g)\Big]\,J\Psi_g\,d\HH^{d-2}\\
&\quad\quad-2\int_{\de \Omega} \Big[\zeta\,\nu_g(\Psi_g) \cdot D \phi_g'(\Psi_g)\,\de_{\xi_g} \phi_g(\Psi_g)\,J\Psi_g-\zeta\,\nu \cdot D \phi_0'\,\de_{\nu} \phi_0\Big]\,d\HH^{d-2}.
\end{align*}
We first notice that there is a universal constant $C$ such that for every $x\in\partial\Omega$ and $g\in C^{2,\alpha}(\partial\Omega)$
$$|\xi_g-\nu|\le C|g|\qquad\text{and}\qquad |\nu_g(\Psi_g)-\nu|\le C\big(|\nabla g|+|g|\big).$$ 
Moreover, the uniform Schauder estimates on the boundary of $\Omega_g$ give that for some universal modulus of continuity $\omega$ we have 
$$\| D^2 \phi_g(\Psi_g)- D^2 \phi_0\|_{L^\infty(\Omega)}\le \omega(\|g\|_{C^{2,\alpha}})\qquad\text{and}\qquad \| D\phi_g(\Psi_g)- D\phi_0\|_{L^\infty(\Omega)}\le \omega(\|g\|_{C^{2,\alpha}}).$$
Thus, we get that 
$$|I_1|\le \|\zeta\|_{L^2(\partial\Omega)}^2\,\omega(\|g\|_{C^{2,\alpha}}).$$ In order to estimate the third integral $I_3$, we notice that \eqref{e:bd1}, a change of variables and an integration by parts give
\begin{align*}
\frac{1}2 I_3&=-\int_{\de \Omega} \zeta\Big[\nu_g(\Psi_g) \cdot D \phi_g'(\Psi_g)\,\de_{\xi_g} \phi_g(\Psi_g)\,J\Psi_g-\nu \cdot D \phi_0'\,\de_{\nu} \phi_0\Big]\,d\HH^{d-2}\\
&=\int_{\de \Omega} \nu_g(\Psi_g) \cdot D \phi_g'(\Psi_g)\, \phi_g'(\Psi_g)\,J\Psi_g\,d\HH^{d-2}-\int_{\de \Omega}\nu \cdot D \phi_0'\,\phi_0'\,d\HH^{d-2}\\
&=\int_{\de \Omega_g} \partial_{\nu_g}\phi_g'\, \phi_g'\,d\HH^{d-2}-\int_{\de \Omega}\partial_\nu \phi_0'\,\phi_0'\,d\HH^{d-2}\\
&=\underbrace{\int_{\Omega_g} |\nabla \phi_g'|^2\,d\HH^{d-1}-\int_{\Omega} |\nabla \phi_0'|^2\,d\HH^{d-1}}_{=: E_1}  - \underbrace{\lambda_g \int_{\Omega_g} |\phi_g'|^2\,d\HH^{d-1} +\lambda_0 \int_{\Omega} |\phi_0'|^2\,d\HH^{d-1}}_{=:E_2}\\
\end{align*}
Using the decomposition \eqref{e:phi'_t_decomposition} and the fact that $h_g$ is harmonic on $\Omega_g$, we get 
$$\int_{\Omega_g} |\nabla \phi_g'|^2\,d\HH^{d-1}=\int_{\Omega_g} |\nabla h_g|^2\,d\HH^{d-1}+\int_{\Omega_g} |\nabla (\psi_g-\phi_g\langle h_g,\phi_g\rangle)|^2\,d\HH^{d-1},$$
where $\ds\langle h_g,\phi_g\rangle:=\int_{\Omega_g}h_g\phi_g$. Next we notice that by \cite[Lemma A.2, eq. (A.20)]{BrDeVe} or by \cite[Lemma 4.10]{DaLa} we have  
 $$\left|\int_{\Omega_g} |\nabla h_g|^2-\int_{\Omega} |\nabla h_0|^2\right|\le \omega(\|g\|_{C^{2,\alpha}})\|\zeta\|_{H^{1/2}}^2.$$
Thus, it is sufficient to prove that
$$\left|\int_{\Omega_g} |\nabla (\psi_g-\phi_g\langle h_g,\phi_g\rangle)|^2\,d\HH^{d-1}-\int_{\Omega} |\nabla (\psi_0-\phi_0\langle h_0,\phi_0\rangle)|^2\,d\HH^{d-1}\right|\le \omega(\|g\|_{C^{2,\alpha}})\|\zeta\|_{L^2}^2,$$
which follows by a simple argument by contradiction, which we sketch for completeness. Indeed, we notice that the functions $\|\zeta\|_{L^2}^{-1}\psi_g$ have uniformly bounded $H^1$-norm and so are converging weakly in $H^1$ (and since they are solutions of a PDE they converge strongly in $H^1$) to $\|\zeta\|_{L^2}^{-1}\psi_0$ as $\|g\|_{C^{2,\alpha}}\to0$. On the other hand, $\|\zeta\|_{L^2}^{-1}h_g$ have uniformly bounded $L^2$ norm and converge weakly in $L^2$ to $\|\zeta\|_{L^2}^{-1}h_0$. Finally, the strong $H^1$ convergence of $\phi_g$ to $\phi_0$ gives that $\|\zeta\|_{L^2}^{-1}\phi_g\langle h_g,\phi_g\rangle\to \|\zeta\|_{L^2}^{-1}\phi_0\langle h_0,\phi_0\rangle$ strongly in $H^1$.

The estimate for $E_2$ is analogous, see again \cite[Lemma 4.10]{DaLa}. This concludes the estimate of $I_3$.

We now estimate $I_2$. Since $D\phi_g$ is parallel to $\nu_g$ on $\partial\Omega_g$ we get that 
 \begin{align*}
I_2&=-2\int_{\de \Omega} \Big[\zeta\,\xi_g \cdot D \phi_g'(\Psi_g)\,\de_{\nu_g} \phi_g(\Psi_g)-\zeta\,\nu_g(\Psi_g) \cdot D \phi_g'(\Psi_g)\,\de_{\xi_g} \phi_g(\Psi_g)\Big]\,J\Psi_g\,d\HH^{d-2}\\
&=-2\int_{\de \Omega} \Big[\zeta\,\xi_g \cdot D \phi_g'(\Psi_g)\,\de_{\nu_g} \phi_g(\Psi_g)-\zeta\,\nu_g(\Psi_g) \cdot D \phi_g'(\Psi_g)\,\big(\xi_g\cdot\nu_g(\Psi_g)\big)\de_{\nu_g} \phi_g(\Psi_g)\Big]\,J\Psi_g\,d\HH^{d-2}\\
&=-2\int_{\de \Omega} \zeta\,\big(\xi_g-(\xi_g\cdot\nu_g(\Psi_g))\nu_g(\Psi_g)\big) \cdot D \phi_g'(\Psi_g)\,\de_{\nu_g} \phi_g(\Psi_g)\,J\Psi_g\,d\HH^{d-2}.
\end{align*}
We first notice that we have the pointwise estimates 
$$\big|\xi_g-(\xi_g\cdot\nu_g(\Psi_g))\nu_g(\Psi_g)\big|\le C\big(|\nabla g|+|g|\big)\qquad\text{and}\qquad |\de_{\nu_g} \phi_g(\Psi_g)|\le C,$$
and that $\xi_g-(\xi_g\cdot\nu_g(\Psi_g))\nu_g(\Psi_g)$ is parallel to $\partial\Omega_g$. Then, on $\partial\Omega_g$ we define 
$$z=\zeta(\Psi_g^{-1})\ ,\quad V=\xi_g(\Phi_g^{-1})\quad\text{and}\quad W=V-(V\cdot\nu_g)\nu_g,$$
so we obtain
\begin{align*}
I_2&=-2\int_{\de \Omega_g} z\,(W \cdot D \phi_g')\,\de_{\nu_g} \phi_g\,d\HH^{d-2}=-2\int_{\de \Omega_g} z\,(W \cdot D (z V\cdot D\phi_g))\,\de_{\nu_g} \phi_g\,d\HH^{d-2}\\
&=-2\int_{\de \Omega_g} z^2\,(W \cdot D (V\cdot D\phi_g))\,\de_{\nu_g} \phi_g\,d\HH^{d-2}-\int_{\de \Omega_g} W \cdot D (z^2)\, (V\cdot D\phi_g)\,\de_{\nu_g} \phi_g\,d\HH^{d-2}\\
&=-2\int_{\de \Omega_g} z^2\,(W \cdot D (V\cdot D\phi_g))\,\de_{\nu_g} \phi_g\,d\HH^{d-2}+\int_{\de \Omega_g} z^2\, \text{div}_{\de B_1}\,\big(W(V\cdot D\phi_g)\,\de_{\nu_g} \phi_g\big)\,d\HH^{d-2}\\
&=-\int_{\de \Omega_g} z^2\,(W \cdot D (V\cdot D\phi_g))\,\de_{\nu_g} \phi_g\,d\HH^{d-2}+\int_{\de \Omega_g} z^2\, (\text{div}_{\de B_1}\,W)\big((V\cdot D\phi_g)\,\de_{\nu_g} \phi_g\big)\,d\HH^{d-2},
\end{align*}
where for the second equality we used the boundary condition $\phi_g'(\Psi_g)=-\zeta\xi_g\cdot D\phi_g(\Psi_g)$, in the fourth one we integrated by parts and in the last we distributed the divergence and used the identity $W Df=0$ for every $f$, by definition of $W$. Thus, we get that 
$$|I_2|\le \omega\big(\|g\|_{C^{2,\alpha}(\partial\Omega)}\big)\int_{\partial\Omega}|\zeta|^2\,d\HH^{d-2},$$
which concludes the proof of Lemma \ref{l:main}.

\renewcommand{\thesection}{\Alph{section}}
\setcounter{section}{1}
	\section{Lyapunov-Schmidt reduction for $\mathcal F$}\label{s:LS} We prove the following lemma, inspired by \cite{Simon0}. We shall denote by $K:=\ker (\delta^2(\mathcal F(0)))$ and by $N:=\dim K$, its dimension (see Subsection \ref{ss:var}). We introduce an auxiliary functional 
\begin{equation}\label{e:defofG}
\mathcal G:C^{2,\alpha}(\partial\Omega)\times\R\to\R\,,\qquad\mathcal G(\zeta,s) = (\kappa_0^2 +s^3)( \lambda(\zeta) - (d-1)) + m(\zeta)-m(0),
\end{equation} 
where $\kappa_0$ is the constant from Subsection \ref{ss:var}.
Thus, the variable  $s$ accounts for the possibility that the coefficient $\kappa$, from Proposition \ref{prop:first_mode}, in front of $\phi_1$ may not be equal to $\kappa_0$. It is easy to check that the first and the second variation of $\mathcal G$ are given by 
	\begin{equation}\label{e:propertiesofG}
	\mathcal G(0,0) = 0,\quad \delta \mathcal G(0, 0)[\xi,r] = 0, \quad \delta^2 \mathcal G(0,0)[(\xi, r), (\eta,t)] = \delta^2 \mathcal F(0)[\xi, \eta].
	\end{equation}
\noindent In particular, we have that $\ker \delta^2 \mathcal G(0,0) = \ker \delta^2\mathcal F(0) \oplus \mathbb R$. In the lemma below,  we let $P_K$ and $P_{K^\perp}$ be the $(L^2(\de\Omega)\oplus \mathbb R)$-projections on $K \oplus \mathbb R$ and $K^\perp\oplus \{0\}$, respectively.
	
	\begin{lemma}[Lyapunov-Schmidt decomposition]\label{l:LS}
	There exists a neighborhood $U$ of $0$ in $C^{1,\alpha}\oplus \mathbb R$ and an analytic map $\Upsilon\colon U\cap(K \oplus \mathbb R)\to K^\perp \subset C^{2,\alpha}(\de \Omega)$ such that
	\begin{equation}\label{e:LS1}
	\Upsilon(0,0)=0\,,\qquad \delta\Upsilon(0,0)=0\,,
	\end{equation}
	and moreover
	\begin{equation}\label{e:LS2}
	\begin{cases}
	P_{K^\perp}\left(\delta\mathcal G(\zeta+\Upsilon(\zeta, s),s)\right)=0&(\zeta,s)\in (K\oplus \mathbb R)\cap U\\
	P_{K}\left( \delta\mathcal G(\zeta+\Upsilon(\zeta,s),s)  \right)=-\nabla G(\zeta,s)&(\zeta,s)\in (K\oplus \mathbb R)\cap U,
	\end{cases}
	\end{equation}
	where $G(\zeta,s)=\mathcal G(\zeta+\Upsilon(\zeta,s),s)$ for every $(\zeta,s) \in (K\oplus \mathbb R)\cap U$. Moreover, all the critical points of $\mathcal G$ in $U$ are given by
	$$
	\mathcal C:=\big\{(\zeta+\Upsilon(\zeta,s),s)\,:\,(\zeta,s)\in (K\oplus \mathbb R)\cap U\quad\mbox{and}\quad \nabla G(\zeta,s)=0\big\}
	$$
	which is an analytic submanifold of the $N+1$-dimensional analytic manifold 
	$$
	\mathcal M:=\left\{(\zeta+\Upsilon(\zeta,s), s)\,:\,(\zeta,s)\in (K\oplus \mathbb R)\cap U\right\}\,.
	$$
	\end{lemma} 
	
	\begin{proof}
	
	Consider the operator 
	$$
	\mathcal N(\zeta,s):= P_{K^\perp}(\delta\mathcal G(\zeta,s)) +P_{K}(\zeta,s): L^2(\de \Omega)\oplus \mathbb R \rightarrow L^2(\de \Omega)\oplus \mathbb R
	$$
	and notice that $\mathcal N(0,0)=0$, since $\delta\mathcal G(0,0)=0$. Moreover,
	$$
	\delta \mathcal N(0,0)[(\zeta,s)]=\frac{d}{dt}\Big|_{t=0}\mathcal N(t\zeta,ts)=\big(\delta^2\mathcal F(0)[\zeta,-],0\big)+P_K(\zeta, s)\,,
	$$
	where we used that $P_K$ and $\delta^2\mathcal F(0)$ are linear (and that $\delta^2\mathcal F(0)[\zeta, P_{K^\perp}\zeta] = \delta^2 \mathcal F(0)[\zeta, \zeta]$). In particular $\delta \mathcal N(0,0)$ has trivial kernel on $C^{2,\alpha}(\de \Omega)\oplus \mathbb R$ by construction. {By standard elliptic theory and Schauder estimates}, we conclude that the operator $\delta \mathcal N(0,0)=(\delta^2\mathcal F(0)+P_K, 1)=(T+H_{\de \Omega}+P_K, 1)$ is an isomorphism of $C^{2,\alpha}\oplus \R$ to $C^{1,\alpha}\oplus \R$, for every $\alpha \in (0,1)$, and therefore we can apply the inverse function theorem to the $C^{2,\alpha}$ operator $\mathcal N\colon C^{2,\alpha}\oplus \R\to C^{1,\alpha}\oplus \R$, producing $\Psi:=\mathcal N^{-1}$ which is a bijection from a neighborhood $W$ of $0$ in $C^{1,\alpha}\oplus \R$ to a neighborhood $U$ of $0$ in $C^{2,\alpha}\oplus R$. 
	
	Now the map we are looking for is simply given by $\Upsilon:=P_{K^\perp}\circ \Psi\colon K\oplus \R \to K^\perp\oplus \{0\}$. Indeed notice that the first conclusion of \eqref{e:LS1} is obvious since $\Psi(0,0)=\Psi(\mathcal N(0,0))=(0,0)$, while the second one follows from the more general observation that for every $\zeta \in K, s\in \R$ we have
	\begin{equation}\label{e:LS3}
	\delta\Upsilon (\zeta,s)[(\eta,r)]=\delta (P_{K^\perp}\Psi(\zeta,s))[(\eta,r)]=P_{K^\perp}(\delta\Psi(\zeta,s))[\eta,r]=0,\; \mbox{for every }\eta\in K, r\in \R\,,
	\end{equation}
	by the linearity of $P_{K^\perp}$.
	
	For what concerns \eqref{e:LS2} for every $(\zeta,s) \in C^{2,\alpha}\oplus \R$ we have
	\begin{align*}
	P_K(\zeta,r)+P_{K^\perp}(\zeta,r)
	&=(\zeta,r)=\mathcal N(\Psi(\zeta,r))=P_{K^\perp}\delta\mathcal G(\Psi(\zeta,r))+P_K(\Psi(\zeta,r))
	\end{align*}
	which implies, by applying $P_K$ and $P_{K^\perp}$ respectively on both sides, that
	$$
	P_K(\zeta,r)=P_K(\Psi(\zeta,r))
	\qquad \mbox{and} \qquad 
	P_{K^\perp}(\zeta,r)=P_{K^\perp}\delta\mathcal G(\Psi(\zeta,r))\,.
	$$
	In particular, using the first identity in the second one we get
	$$
	P_{K^\perp}(\zeta,r)=P_{K^\perp}\delta\mathcal G(P_K\Psi(\zeta,r)+P_{K^\perp}\Psi(\zeta,r))=P_{K^\perp}\delta\mathcal G(P_K\zeta+\Upsilon(\zeta,r),r)\,,
	$$
	so that, if $\zeta \in K\cap U, r \in \R$, we conclude
	$$
	P_{K^\perp}\delta\mathcal G(\zeta+\Upsilon(\zeta,r),r)=0\,.
	$$
	The second conclusion of \eqref{e:LS2} follows by differentiating the function $G(\zeta,r)$ as follows. Let $\eta \in K$, then we have
	$$
	\begin{aligned}
	\left\langle \nabla G(\zeta,r), (\eta,s)\right\rangle&=\delta \mathcal G(\zeta+\Upsilon(\zeta,r),r)[\eta+\delta\Upsilon(\zeta,r)[\eta,s],s]\\ 
	&\!\!\!\stackrel{\eqref{e:LS3}}{=}\delta \mathcal G(\zeta+\Upsilon(\zeta,r),r)[\eta,s]=P_K(\delta \mathcal G(\zeta+\Upsilon(\zeta,r),r))[\eta,s]\,,
	\end{aligned}$$
	where the last equality follows from the first one in \eqref{e:LS2}.
	\end{proof}

\bibliographystyle{plain}
\bibliography{references-Cal}

\end{document}